\documentclass[10pt]{tsoart}

\begin{document}

\title[The total surgery obstruction]{The total surgery obstruction revisited}

\author{Philipp K\"uhl, Tibor Macko and Adam Mole}

\maketitle

\begin{center}
\today
\end{center}

\begin{abstract}
The total surgery obstruction of a finite $n$-dimensional Poincar\'e complex $X$ is an element $s(X)$ of a certain abelian group $\SS_n (X)$ with the property that for $n \geq 5$ we have $s(X) = 0$ if and only if $X$ is homotopy equivalent to a closed $n$-dimensional topological manifold. The definitions of $\SS_n (X)$ and $s(X)$ and the property are due to Ranicki in a combination of results of two books and several papers. In this paper we present these definitions and a detailed proof of the main result so that they are in one place and we also add some of the details not explicitly written down in the original sources.
\end{abstract}

\addtocontents{toc}{\protect\setcounter{tocdepth}{1}}

\tableofcontents

\section{Introduction} \label{sec:intro}


An important problem in the topology of manifolds is deciding
whether there is an $n$-dimensional closed topological manifold in
the homotopy type of a given $n$-dimensional finite Poincar\'e
complex $X$.


Recall that the ``classical surgery theory'' alias
``Browder-Novikov-Sullivan-Wall-Kirby-Siebenmann theory'' provides
a method to decide this question in the form of a two-stage
obstruction theory, when $n \geq 5$. A result of Spivak provides us
with the Spivak normal fibration (SNF) $\nu_X \co X \ra \BSG$, which
is a spherical fibration, stably unique in some sense. If $X$ is
homotopy equivalent to a closed manifold then $\nu_X$ reduces to a
stable topological block bundle, say $\bar \nu_X \co X \ra \BSTOP$.
The existence of such a reduction is the first obstruction. In terms
of classifying spaces, the composition
\begin{equation} \label{eqn:first-obstruction}
H \circ \nu_X \co X \ra \BSG \ra \textup{B}(\GTOP)
\end{equation}
has to be homotopic to a constant map. Any reduction $\bar \nu_X$
determines a degree one normal map $(f,b) \co M \ra X$ from some
$n$-dimensional closed topological manifold $M$ to $X$ with a
surgery obstruction, which we call the quadratic signature of
$(f,b)$ and denote
\begin{equation} \label{eqn:second-obstruction}
\qsign_{\ZZ[\pi_1 (X)]} (f,b) \in L_n (\ZZ \pi_1 (X)).
\end{equation}
The complex $X$ is homotopy equivalent to a closed manifold if and
only if there exists a reduction for which $\qsign_{\ZZ[\pi_1 (X)]}
(f,b) = 0$.


The ``algebraic theory of surgery'' of Ranicki replaces the above
theory with a single obstruction, namely the total surgery
obstruction
\begin{equation} \label{eqn:tso}
s (X) \in \SS_n (X)
\end{equation}
where $\SS_n (X)$ is the $n$-dimensional structure group of $X$ in
the sense of the algebraic theory of surgery, which is a certain
abelian group associated to $X$. It is the aim of this paper to
discuss the definitions of $\SS_n (X)$ and $s(X)$ and explain how
they replace the classical theory. 


The advantage of the algebraic theory is two-fold. On the one hand
it is a single obstruction theory which by itself can be more
convenient. On the other hand it turns out that the group $\SS_n
(X)$ has an $L$-theoretic definition, in fact it is isomorphic to a
homotopy group of the homotopy fiber of a certain assembly map in
$L$-theory. Hence the alternative approach allows us to solve our problem by entirely $L$-theoretic methods, for example by showing that the assembly map induces an isomorphism on homotopy groups and so $\SS_n (X) = 0$. This possibility is in contrast with the classical surgery theory, where the first obstruction (\ref{eqn:first-obstruction}) is not $L$-theoretic in nature.

However, in practice, often slightly different assembly maps turn out to be more accessible for studying, as is the case for example in the recent papers \cite{Bartels-Lueck(2009)}, \cite{Bartels-Lueck-Weinberger(2009)}. Then the theory needs to be modified to accommodate in addition the integer valued Quinn resolution obstruction. In the concluding section \ref{sec:conclusion} we offer more comments on this generalization and applications as well as examples.


The ingredients in the theory surrounding the total surgery obstruction are:

\begin{itemize}
 \item The algebraic theory of surgery of Ranicki from \cite{Ranicki-I-(1980)}, \cite{Ranicki-II-(1980)}. This comprises  various sorts of $L$-groups of chain complexes over various
 additive categories with chain duality and with various notions of
 Poincar\'e duality. The sorts of $L$-groups are ``symmetric'',
 ``quadratic'' and ``normal''. The last notion is also due to Weiss in \cite{Weiss-I(1985)}, \cite{Weiss-II(1985)}.
 \item The classical surgery theory in the topological category from \cite{Browder(1971)}, \cite{Wall(1999)}. The
 algebraic theory is not independent of the classical theory, in the
 sense that the proof that the algebraic theory answers our problem
 uses the classical theory. 
 \item Topological transversality in all dimensions and codimensions
 as provided by Kirby-Siebenmann \cite{Kirby-Siebenmann(1977)} and
 Freedman-Quinn \cite{Freedman-Quinn(1990)}.
 \item The surgery obstruction isomorphism, \cite[Essay V, Theorem C.1]{Kirby-Siebenmann(1977)}:
\[
 \qsign_\ZZ \co \pi_n (\G/\TOP) \xra{\cong} L_n (\ZZ) \; \textup{for} \; n\geq 1.
\]
 \item Geometric normal spaces and geometric normal transversality,
 both of which were invented by Quinn. However, the whole theory as
 announced in \cite{Quinn(1972)} is not needed. It is replaced by
 the algebraic normal $L$-groups from the first item.
\end{itemize}


\subsection{The basics of the algebraic theory of surgery}
\label{subsec:basics-of-alg-sur} Mishchenko and Ranicki defined for a ring $R$ with involution the
symmetric and quadratic $L$-groups $L^n (R)$ and $L_n (R)$
respectively, as cobordism groups of chain complexes over $R$ with
symmetric and quadratic Poincar\'e structure respectively. The quadratic $L$-groups are isomorphic to the surgery obstruction groups of Wall \cite{Wall(1999)}.

Let $W$ be the standard $\ZZ[\ZZ_2]$-resolution of $\ZZ$. An
$n$-dimensional \emph{symmetric structure} on a chain complex $C$ is
an $n$-dimensional cycle
\[
 \varphi \in \Ws{C} := \Hom_{\ZZ[\ZZ_2]} (W,C \otimes_R C) \cong \Hom_{\ZZ[\ZZ_2]} (W,\Hom_R (C^{-\ast},C)).
\]
It can be written out in components $\varphi = (\varphi_i \co
C^{n-\ast} \ra C_{\ast+i})_{i \in \NN}$. If $\varphi_0 \co
C^{n-\ast} \ra C$ is a chain homotopy equivalence, then the
structure is called {\it Poincar\'e}. Given an $n$-dimensional cycle $x
\in C(X)$, there is a symmetric structure $\varphi (x)$ on $C(\tilde
X)$ over $\ZZ[\pi_1 (X)]$ with $\varphi (x)_0 =  -\cap x \co
C(\tilde X)^{n-\ast} \ra C(\tilde X)$ given by an equivariant
version of the familiar Alexander-Whitney diagonal approximation
construction. If $X$ is a Poincar\'e complex with the fundamental
class $[X]$, then we obtain the \emph{symmetric signature} of $X$,
\begin{equation}
 \ssign_{\ZZ[\pi_1 (X)]} (X) = [(C(\tilde X),\varphi ([X]))] \in L^n (\ZZ[\pi_1 (X)]).
\end{equation}
An $n$-dimensional \emph{quadratic structure} on a chain complex $C$
is an $n$-dimensional cycle
\[
 \psi \in \Wq{C} := W \otimes_{\ZZ[\ZZ_2]} (C \otimes_R C) \cong W \otimes_{\ZZ[\ZZ_2]} (\Hom_R (C^{-\ast},C)).
\]
There is a symmetrization map $1+T \co \Wq{C} \ra \Ws{C}$ which allows us to see quadratic structures as refinements of symmetric structures. A quadratic structure is called {\it Poincar\'e} if its symmetrization is Poincar\'e. Such a quadratic structure is more subtle to obtain from a geometric situation. As explained in Construction \ref{constrn:quadconstrn}, given an $n$-dimensional cycle $x \in C(X)$ and a stable map $F \co \Sigma^p X_+ \ra \Sigma^p M_+$ there is a quadratic structure $\psi (x)$ over $\ZZ [\pi_1 (X)]$ on $C (\tilde M)$. A degree one normal map $(f,b) \co M \ra X$ between $n$-dimensional Poincar\'e complexes induces a map of Thom spaces $\Th (b) \co \Th (\nu_M) \ra \Th (\nu_X)$ which in turn, using $S$-duality,\footnote{see section \ref{sec:normal-cplxs} for more details if needed} produces a stable map $F \co \Sigma^p X_+ \ra \Sigma^p M_+$ for some $p$. The quadratic construction $\psi$ produces from the fundamental class $[X]$ a quadratic structure on $C(\tilde M)$. Considering the Umkehr map $f^{!} \co C(\tilde X) \ra \Sigma^{-p} C (\Sigma^p \tilde{X}_+) \ra \Sigma^{-p} C(\Sigma^p \tilde{M}_+) \ra C(\tilde{M})$ and the inclusion into the algebraic mapping cone $e \co C(\tilde M) \ra \sC (f^!)$ we obtain an $n$-dimensional quadratic Poincar\'e complex called the \emph{quadratic signature} of $(f,b)$
\begin{equation} \label{eqn:quad-sign-deg-one-normal-map}
 \qsign_{\ZZ[\pi_1 (X)]} (f,b) = [(C(f^{!}),e_{\%} \psi ([X]))] \in L_n (\ZZ[\pi_1 (X)])
\end{equation}
If $(f,b)$ is a degree one normal map with $M$ an $n$-dimensional
manifold, then the quadratic signature
(\ref{eqn:quad-sign-deg-one-normal-map}) coincides with the
classical surgery obstruction.

\subsection{The structure group $\SS_n (X)$} A generalization of the
theory in \ref{subsec:basics-of-alg-sur} is obtained by replacing
the ring $R$ with an algebraic bordism category $\Lambda$. Such a
category contains an underlying additive category with chain duality
$\AA$. The category $\Lambda$ specifies a subcategory of the
category of structured chain complexes\footnote{meaning chain
complexes with a symmetric or quadratic structure} in $\AA$ and a
type of Poincar\'e duality. We obtain cobordism groups of such chain
complexes $L^n (\Lambda)$ and $L_n (\Lambda)$ and also spectra
$\bL^\bullet (\Lambda)$ and $\bL_\bullet (\Lambda)$ whose homotopy
groups are these $L$-groups.

The notion of an additive category with chain duality allows us to
consider structured chain complexes over a simplicial complex $X$,
with $\pi = \pi_1 (X)$. Informally one can think of such a
structured chain complex over $X$ as a compatible collection of
structured chain complexes over $\ZZ$ indexed by simplices of $X$.
``Forgetting'' the indexing ``assembles'' such a complex over $X$ to
a complex over $\ZZ$ and an equivariant version of this process
yields a complex over $\ZZ [\pi]$. The algebraic bordism categories
allow us to consider various types of Poincar\'e duality for structured
complexes over $X$. There is the \emph{local Poincar\'e duality}
where it is required that the structure over each simplex is
Poincar\'e, with the category of all such complexes denoted
$\Lambda(\ZZ)_\ast (X)$. It turns out that\footnote{here $\bL_\bullet (\ZZ)$ is short for $\bL_\bullet (\Lambda (\ZZ)_\ast (\textup{pt.}))$}
\begin{equation}
 L_n (\Lambda(\ZZ)_\ast (X)) \cong H_n (X ; \bL_\bullet (\ZZ)) \qquad L^n (\Lambda(\ZZ)_\ast (X)) \cong H_n (X ; \bL^\bullet (\ZZ)).
\end{equation}
Then there is the \emph{global Poincar\'e duality} where only the
assembled structure is required to be Poincar\'e, with the category
of all such complexes denoted $\Lambda(\ZZ[\pi])$ for the purposes
of this introduction\footnote{The notation here is justified by
Proposition \ref{prop:algebraic-pi-pi-theorem} which says that the
$L$-theory of this category is indeed isomorphic to the $L$-theory
of the group ring.}. The assembly gives a functor $A \co
\Lambda(\ZZ)_\ast (X) \ra \Lambda(\ZZ [\pi])$ which induces on the
$L$-groups the assembly maps 
\begin{equation} \label{eqn:assembly-map}
A \co H_n (X;\bL^\bullet (\ZZ)) \ra L^n (\ZZ[\pi]) \qquad A \co H_n (X;\bL_\bullet (\ZZ)) \ra L_n (\ZZ[\pi])
\end{equation}

In a familiar situation such chain complexes arise as follows. A
triangulated $n$-dimensional manifold $X$ has a dual cell
decomposition, where for each simplex $\sigma \in X$ the dual cell
$(D(\sigma),\del D(\sigma))$ is an $(n-|\sigma|)$-dimensional
submanifold with boundary. The collection of chain complexes
$C(D(\sigma),\del D(\sigma))$ together with corresponding symmetric
structures provides the symmetric signature of $X$ over $X$, which
is a locally Poincar\'e symmetric complex over $X$
\begin{equation} \label{eqn:sym-sign-over-X}
 \ssign_X (X) \in H_n (X;\bL^\bullet (\ZZ)) \quad A (\ssign_X (X)) = \ssign_{\ZZ [\pi]} (X) \in L^n (\ZZ [\pi]).
\end{equation}
A degree one normal map $(f,b) \co M \ra X$ from an $n$-dimensional
manifold $M$ to a triangulated $n$-dimensional manifold $X$ can be
made transverse to the dual cells of $X$. Denoting $M(\sigma) = f^{-1} (D(\sigma))$ this yields a collection
of degree one normal maps of manifolds with boundary
$(f(\sigma),b(\sigma)) \co (M(\sigma),\del (M(\sigma)) \ra
(D(\sigma),\del D(\sigma))$. The collection of chain complexes $\sC
(f(\sigma)^{!})$ with the corresponding quadratic structures
provides the quadratic signature of $(f,b)$ over $X$ which is a
locally Poincar\'e quadratic complex over $X$
\begin{equation}
 \qsign_X (f,b) \in H_n (X;\bL_\bullet (\ZZ)) \quad A (\qsign_X (f,b)) = \qsign_{\ZZ [\pi]} (f,b) \in L_n (\ZZ [\pi]).
\end{equation}

The $L$-groups relevant for our geometric problem are modifications
of the above concepts obtained by using certain connective
versions.\footnote{This is a technical point addressed in section
\ref{sec:conclusion}}

The cobordism group of $n$-dimensional quadratic $1$-connective
complexes that are locally Poincar\'e turns out to be isomorphic to
the homology group  $H_\ast (X,\bL_\bullet \langle 1 \rangle)$,
where the symbol $\bL_\bullet \langle 1 \rangle$ denotes the
$1$-connective quadratic $L$-theory spectrum. The cobordism group of
$n$-dimensional quadratic $1$-connective complexes that are globally
Poincar\'e turns out to be isomorphic to $L_n (\ZZ[\pi])$. The
assembly functor induces an assembly map analogous to
(\ref{eqn:assembly-map}).

The structure group $\SS_n (X)$ is the cobordism group of
$(n-1)$-dimensional quadratic chain complexes over $X$ that are
locally Poincar\'e, locally $1$-connective and globally
contractible. All these groups fit into the algebraic surgery exact
sequence:
\begin{equation} \label{eqn:ses}
\cdots \ra H_n (X,\bL_\bullet \langle 1 \rangle) \xra{A} L_n
(\ZZ[\pi_1 (X)]) \xra{\del} \SS_n (X) \xra{I} H_{n-1} (X;\bL_\bullet
\langle 1 \rangle) \ra \cdots
\end{equation}
The map $I$ is induced by the inclusion of categories and the map
$\del$ will be described below.

\subsection{The total surgery obstruction $s(X) \in \SS_n (X)$} We
need to explain how to associate to $X$ an $(n-1)$-dimensional
quadratic chain complex over $X$ that is locally Poincar\'e, locally
$1$-connective and globally contractible.

Being Poincar\'e is by definition a global condition. What local
structure does a Poincar\'e complex have? The answer is the
structure of a normal complex.

An $n$-dimensional normal complex $(Y,\nu,\rho)$ consists of a space
$Y$, a $k$-dimensional spherical fibration $\nu \co Y \ra \BSG (k)$
and a map $\rho \co S^{n+k} \ra \Th (\nu)$. There is also a notion
of a normal pair, a normal cobordism, and normal cobordism groups. A
Poincar\'e complex $X$ embedded into a large euclidean space has a
regular neighborhood. Its boundary produces a model for the SNF
$\nu_X$ and collapsing the boundary gives a model for the Thom space
$\Th (\nu_X)$ with the collapse map $\rho_X$. In a general normal
complex the underlying space $Y$ does not have to be Poincar\'e.
Nevertheless it has a preferred homology class $h(\rho) \cap u(\nu)
= [Y] \in C_n (Y)$, where $u (\nu) \in C^k (\Th (\nu))$ is some
choice of the Thom class and $h$ denotes the Hurewicz homomorphism.
The class $[Y]$ produces a preferred equivalence class of symmetric
structures on $C(Y)$.

There exists a notion of an $n$-dimensional normal algebraic complex
$(C,\theta)$ over any additive category with chain duality $\AA$. At
this stage we only say that the normal structure $\theta$ contains a
symmetric structure, and should be seen as a certain refinement of
that symmetric structure.\footnote{The details are presented in
section \ref{sec:normal-cplxs}} Again one can consider normal
complexes in an algebraic bordism category $\Lambda$, specifying an
interesting subcategory and the type of Poincar\'e duality on the
underlying symmetric structure. The cobordism groups are denoted
$NL^n (\Lambda)$ and there are also associated spectra $\bNL^\bullet
(\Lambda)$. For a ring $R$ we have the cobordism groups $NL^n (R)$
of $n$-dimensional normal complexes over $R$ with no Poincar\'e
duality in this case! A geometric normal complex $(Y,\nu,\rho)$
gives rise to a normal algebraic complex, called the normal
signature
\begin{equation} \label{eqn:norm-sign}
 \nsign_{\ZZ [\pi_1 (Y)]} (Y) \in NL^n (\ZZ[\pi_1 (Y)]).
\end{equation}
whose symmetric substructure is the one associated to its
fundamental class $[Y]$.

So how is a Poincar\'e complex $X$ locally normal? For a simplex
$\sigma \in X$ consider the dual cell $(D(\sigma),\del D(\sigma))$,
which is a pair of spaces, not necessarily Poincar\'e. The SNF
$\nu_X$ can be restricted to $(D(\sigma),\del D(\sigma))$, remaining
a spherical fibration, say $(\nu_X (\sigma),\nu_X (\del \sigma))$. A
certain trick\footnote{Presented in section
\ref{sec:normal-signatures-over-X}} is needed to obtain a map
\[
\rho(\sigma) \co (D^{n+k-|\sigma|},S^{n-1+k-|\sigma|}) \ra (\Th
(\nu_X (\sigma)),\Th (\nu_X (\del \sigma)))
\]
providing us with a normal complex ``with boundary''. The collection
of these gives rise to a compatible collection of normal algebraic
complexes over $\ZZ$ and we obtain an $n$-dimensional normal
algebraic complex over $X$ whose symmetric substructure is globally
Poincar\'e. As such it can be viewed as a normal complex in two
distinct algebraic bordism categories.

There is the category $\widehat \Lambda (\ZZ) \langle 1 / 2 \rangle
(X)$, which is a $1/2$-connective version of all normal complexes
over $X$ with no Poincar\'e duality. We obtain the $1/2$-connective
normal signature of $X$ over $X$\footnote{The connectivity condition
turns out to be fulfilled, see section \ref{sec:conn-versions} for
explanation if needed}
\begin{equation}  \label{eqn:norm-sign-over-X}
 \nsign_X (X) \in NL^n (\widehat \Lambda (\ZZ) \langle 1/2 \rangle (X)) \cong  H_n (X, \bNL^\bullet \langle 1/2 \rangle).
\end{equation}
Similarly as before the assembly of (\ref{eqn:norm-sign-over-X})
becomes (\ref{eqn:norm-sign}).

Then there is the category $\Lambda (\ZZ) \langle 1 / 2 \rangle
(X)$, which is a $1/2$-connective version of all normal complexes
over $X$ with global Poincar\'e duality. The cobordism group $NL^n
(\Lambda (\ZZ) \langle 1 / 2 \rangle (X))$ is called the
$1/2$-connective visible symmetric group and denoted $VL^n (X)$. We
obtain the visible signature of $X$ over $X$
\begin{equation} \label{eqn:visible-signature}
 \vsign_X (X) \in VL^n (X).
\end{equation}
The forgetful functor $\Lambda (\ZZ) \langle 1 / 2 \rangle (X) \ra
\widehat \Lambda(\ZZ) \langle 1 / 2 \rangle (X)$ induces a map on
$NL$-groups which sends (\ref{eqn:visible-signature}) to
(\ref{eqn:norm-sign-over-X}).

But we are after an $(n-1)$-dimensional quadratic complex. To obtain
it we need in addition the concept of a boundary of a structured
chain complex. Consider an  $n$-dimensional symmetric complex
$(C,\varphi)$ over any additive category with chain duality $\AA$.
Its boundary $(\del C ,\del \varphi)$ is an $(n-1)$-dimensional
symmetric complex in $\AA$ whose underlying chain complex is defined
as $\del C = \Sigma^{-1} \sC (\varphi_0)$. The $(n-1)$-dimensional
symmetric structure $\del \varphi$ is inherited from $\varphi$. It
becomes Poincar\'e in $\AA$, meaning $\sC ((\del \varphi)_0)$ is
contractible in $\AA$, by a formal argument.\footnote{The choice of
terminology is explained below Definition \ref{defn:symbdy}} The
boundary $(\del C,\del \varphi)$ can be viewed as measuring how the
complex $(C,\varphi)$ itself is Poincar\'e in $\AA$. It is shown in
Proposition \ref{lem:normal-gives-quadratic-boundary} that an
$n$-dimensional symmetric complex $(C,\varphi)$ which is a part of a
normal complex $(C,\theta)$ comes with a quadratic refinement $\del
\psi$ of the symmetric structure $\del \varphi$ on the
boundary.\footnote{The normal structure on $C$ provides a second
stable symmetric structure on $\del C$ in addition to $\del \varphi$
and the two structures stably coincide. Such a situation yields a
quadratic structure.}

From this description it follows that the boundary produces the
following two maps:
\begin{equation}
\del \co L_n (\ZZ[\pi]) \ra \SS_n (X) \quad \textup{and} \quad \del
\co VL^n (X) \ra \SS_n (X).
\end{equation}

The total surgery obstruction $s(X)$ is defined as the
$(n-1)$-dimensional quadratic complex over $X$, obtained as the
boundary of the visible signature
\begin{equation}
 s (X) = \del \; \vsign_X (X) \in \SS_n (X).
\end{equation}
It is locally Poincar\'e, because by the above discussion any
boundary of a complex over $X$ is locally Poincar\'e. It is also
globally contractible since $X$ is Poincar\'e and hence the boundary
of the assembled structure is contractible. The connectivity
assumption is also fulfilled.

\


\begin{mainthm} \textup{\cite[Theorem 17.4]{Ranicki(1992)}}  \label{main-thm}

Let $X$ be a finite Poincar\'e complex of formal dimension $n \geq
5$. Then $X$ is homotopy equivalent to a closed $n$-dimensional
topological manifold if and only if
\[
 0 = s(X) \in  \SS_n (X).
\]
\end{mainthm}

The proof is based on:

\begin{maintechthm} \quad \label{thm:main-technical-thm}

Let $X$ be a finite Poincar\'e complex of formal dimension $n \geq
5$ and denote by $t(X) = I(s(X)) \in H_{n-1} (X,\bL_\bullet \langle
1 \rangle)$. Then we have
\begin{enumerate}
\item[(I)] $t(X) = 0$ if and only if there exists a topological block bundle reduction of the SNF $\nu_X \co X \ra \BSG$.
\item[(II)] If $t(X) = 0$ then we have
\begin{align*}
\partial^{-1} s(X) = \{ & - \qsign_{\ZZ[\pi_1 (X)]} (f,b) \in L_n (\ZZ[\pi_1 (X)])
\; | \\
& (f,b) : M \ra X \; \textrm{degree one normal map}, \; M \;
\textrm{manifold} \}.
\end{align*}
\end{enumerate}
\end{maintechthm}

\begin{proof}[Proof of \mainth \;assuming \maintechth]

\

If $X$ is homotopy equivalent to a manifold then $t(X) = 0$ and by
(II) the set  $\partial^{-1} s(X)$ contains $0$, hence $s(X) = 0$.

If $s(X)=0$ then $t(X)=0$ and hence by (I) the SNF of $X$ has a
topological block bundle reduction. Also $\partial^{-1} s(X)$ must
contain $0$ and hence by (II) there exists a degree one normal map
with the target $X$ and with the surgery obstruction $0$.
\end{proof}

\begin{rem}
The condition (I) might be puzzling for the following reason. As
recalled earlier, the classical surgery gives an obstruction to the
reduction of the SNF in the group $[X,\textup{B}(\G/\TOP)] = H^1
(X;\G/\TOP)$. It is important to note that here the
$\Omega^\infty$-space structure used on $\G/\TOP$ corresponds to the
Whitney sum and hence not the one that is compatible with the
well-known homotopy equivalence $\G/\TOP \simeq \bL \langle 1
\rangle_0$. On the other hand $t(X) \in H_{n-1} (X; \bL_\bullet
\langle 1 \rangle)$. We note that the claim of (I) is NOT that the
two groups are isomorphic, it merely says that one obstruction is
zero if and only if the other is zero.
\end{rem}


\subsection{Informal discussion of the proof of \maintechth}

\

Part (I): The crucial result is the relation between the quadratic,
symmetric, and normal $L$-groups of a ring $R$ via the long exact
sequence
\begin{equation} \label{eqn:quad-sym-norm-L-groups}
   \xymatrix{
      \ldots         \ar[r]             &
      L_n (R)        \ar[r]^{1+T}       &
      L^n (R)        \ar[r]^J           &
      NL^n (R)       \ar[r]^{\partial}     &
      L_{n-1} (R)     \ar[r]             &
      \ldots
    }
\end{equation}
Here the maps $1+T$ and $\del$ were already discussed. The map $J$
exists because a symmetric Poincar\'e structure on a chain complex
yields a preferred normal structure, reflecting the observation that
a Poincar\'e complex has the SNF and hence gives a geometric normal
complex. Using suitable connective versions there is a related
homotopy fibration sequence of spectra (here the implicit ring is
$\ZZ$)
\begin{equation} \label{fib-seq:quad-sym-norm-long}
\bL_\bullet \langle 1 \rangle \ra \bL^\bullet \langle 0 \rangle \ra
\bNL^\bullet \langle 1/2 \rangle \ra \Sigma \bL_\bullet \langle 1
\rangle.
\end{equation}
The exactness of (\ref{eqn:quad-sym-norm-L-groups}) and fibration
property of (\ref{fib-seq:quad-sym-norm-long}) are not easily
observed. It is a result of \cite{Weiss-I(1985),Weiss-II(1985)} for
which we offer some explanation in section \ref{sec:normal-cplxs}.
The sequence (\ref{fib-seq:quad-sym-norm-long}) induces a long exact
sequence in homology
\begin{equation} \label{les:hlgy-of-L-thy-spectra}
\cdots \ra H_n (X ; \bL^\bullet \langle 0 \rangle) \xra{J} H_n (X;
\bNL^\bullet \langle 1/2 \rangle) \xra{\del} H_{n-1} (X ;
\bL_\bullet \langle 1 \rangle) \ra \cdots
\end{equation}

Another tool is the $S$-duality from stable homotopy theory, which
gives
\begin{equation} \label{eqn:S-duality}
H_n (X ; \bE) \cong H^k (\Th (\nu_X) ; \bE) \quad \textup{with} \;
\bE = \bL_\bullet \langle 1 \rangle, \bL^\bullet \langle 0 \rangle,
\; \textup{or} \; \bNL^\bullet \langle 1/2 \rangle.
\end{equation}
and transforms the exact sequence (\ref{les:hlgy-of-L-thy-spectra}) into an exact sequence in cohomology of the Thom space $\Th (\nu_X)$. 

The proof of the theorem is organized within the following
commutative braid:

\vspace{1cm}

\[
\xymatrix@C=0.5em{ H_n (X;\bL^\bullet \langle 0 \rangle) \ar[dr]
\ar@/^2.5pc/[rr] & & H_n (X;\bNL^\bullet \langle 1/2 \rangle)
\ar[dr] \ar@/^2.5pc/[rr] & &
L_{n-1}\ (\ZZ[\pi]) \\
& VL^n (X) \ar[ur] \ar[dr] & & H_{n-1} (X;\bL_\bullet \langle 1 \rangle) \ar[ur] \ar[dr] \\
L_n (\ZZ[\pi]) \ar[ur] \ar@/_2.5pc/[rr] & & \SS_{n} (X) \ar[ur]
\ar@/_2.5pc/[rr] & & H_{n-1}\ (X;\bL^\bullet \langle 0 \rangle) }
\]

\vspace{1cm}

We observe that
\begin{equation} \label{eqn:reduction-obstruction}
  t (X) = \del \; \nsign_X (X) \in H_{n-1} (X;\bL_\bullet \langle 1 \rangle)
\end{equation}
with the normal signature over $X$ from
(\ref{eqn:norm-sign-over-X}).

Assuming the above, the proof proceeds as follows. If $\nu_X$ has a
reduction, then it has an associated degree one normal map $(f,b)
\co M \ra X$ which can be made transverse to the dual cells of $X$.
For each $\sigma$ the preimage $(M(\sigma),\del M(\sigma))$ of the
dual cell $(D(\sigma),\del D(\sigma))$ is an
$(n-|\sigma|)$-dimensional submanifold with boundary and
generalizing (\ref{eqn:sym-sign-over-X}) we obtain
\begin{equation} \label{eqn:sym-sign-of-mfd-over-X}
 \ssign_X (M) \in L^n (\Lambda(\ZZ) \langle 0 \rangle_\ast (X)) \cong H_n (X ; \bL^\bullet \langle 0 \rangle)
\end{equation}
The mapping cylinder of the degree one normal map $(f,b)$ becomes a
normal cobordism between $M$ and $X$ and, as such, it produces a
normal algebraic cobordism between $J (\ssign_X (M))$ and $\nsign_X
(X)$. In other words the symmetric signature $\ssign_X (M)$ is a lift of the normal signature $\nsign_X (X)$ from (\ref{eqn:norm-sign-over-X}) and it follows from the exact sequence
(\ref{les:hlgy-of-L-thy-spectra}) that $t(X)$ vanishes.

The crucial concept used in the proof of the other direction is that
of an orientation of a spherical fibration with respect to a ring
spectrum, such as $\bL^\bullet \langle 0 \rangle$ and $\bNL^\bullet
\langle 1/2 \rangle$. For such a ring spectrum an $\bE$-orientation
of the SNF $\nu_X$ is an element in $H^k (\Th (\nu_X) ; \bE)$ with a
certain property. By the $S$-duality (\ref{eqn:S-duality}) it
corresponds to a homology class in $H_n (X, \bE)$. It turns out that
the SNF $\nu_X$ has a certain canonical $\bNL^\bullet \langle 1/2
\rangle$-orientation which corresponds to the normal signature
(\ref{eqn:norm-sign-over-X}) in this way. Similarly if there is a
reduction of $\nu_X$ with a degree one normal map $(f,b) \co M \ra
X$, then it gives an $\bL^\bullet \langle 0 \rangle$-orientation of
$\nu_X$ which corresponds to the symmetric signature
(\ref{eqn:sym-sign-of-mfd-over-X}) of $M$ over $X$.

Theorem \ref{thm:lifts-vs-orientations} says that a spherical
fibration has a topological block bundle reduction if and only if
its canonical $\bNL^\bullet \langle 1/2 \rangle$-orientation has an
$\bL^\bullet \langle 0 \rangle$-lift. The~proof is by analyzing
classifying spaces for spherical fibrations with orientations, a
certain diagram (Proposition
\ref{prop:canonical-L-orientations-on-class-spaces}) is shown to be
a homotopy pullback. Here is used the fact that the surgery
obstruction map $\pi_n (\G/\TOP) \ra L_n (\ZZ)$ is an isomorphism
for $n > 1$.

Part (II): To show the inclusion of the right hand side one needs to study the quadratic signatures over $X$
of degree one normal maps $(f,b) \co M \ra X$ with $M$ an $n$-dimensional closed manifold and $X$ an $n$-dimensional Poincar\'e complex. That means studying the local structure of such maps which boils down to studying quadratic signatures of degree one normal maps $(g,c) \co N \ra Y$ where $Y$ is only a normal complex. In this case one obtains a non-Poincar\'e quadratic complex whose boundary can be related to the quadratic boundary of the normal complex $Y$ as shown in Proposition \ref{prop:degree-one-normal-map-mfd-to-normal-cplx}. Passing to complexes over $X$ one obtains a quadratic complex over $X$, still denoted $\qsign_X (f,b)$ although it is not locally Poincar\'e, whose boundary is described in Proposition \ref{prop:degree-one-normal-map-mfd-to-poincare-over-X} establishing the required inclusion. 

To study the other inclusion a choice is made of a degree one normal map $(f_0,b_0) \co M_0 \ra X$. Recall that all degree one normal maps with the target $X$ are organized in the cobordism set of the normal invariants $\sN (X)$. One considers the surgery obstruction map relative to $(f_0,b_0)$  
\begin{equation} \label{eqn:surgery-obstruction-map}
\qsign_{\ZZ [\pi_1 (X)]} (-,-) - \qsign_{\ZZ [\pi_1 (X)]} (f_0,b_0) \co \sN (X) \ra L_n (\ZZ [\pi_1 (X)]).
\end{equation}
The signature $\qsign_X$ over $X$ relative to $(f_0,b_0)$  produces a map from the normal invariants $\sN (X)$ to the homology group $H_n (X ; \bL_\bullet \langle 1 \rangle)$. The main technical result is now Proposition \ref{prop:identification} which states that this map provides us with an identification of (\ref{eqn:surgery-obstruction-map}) with the assembly map (\ref{eqn:assembly-map}) for $X$. In particular it says that $\qsign_X$ relative to $(f_0,b_0)$ produces a bijection. Via the standard identification $\sN (X) \cong [X;\G/\TOP]$ and the bijection $[X,\G/\TOP] \cong H^0 (X;\bL_\bullet \langle 1 \rangle)$ (using the Kirby-Siebenmann isomorphism again) this boils down to identifying $\qsign_X$ with the Poincar\'e duality with respect to the spectrum $\bL_\bullet \langle 1 \rangle$. Here, similarly as in part (I), a relationship between the signatures and orientations with respect to the $L$-theory spectra plays a prominent role (Proposition \ref{prop:S-duals-of-orientations-are-signatures-relative-case-non-mfd} and Lemma \ref{lem:refined-orientations-vs-cup-product}).

\subsection*{The purpose of the paper}
As the title suggests this article revisits the existing theory which was developed over decades by Andrew Ranicki, with contributions also due to Michael Weiss. On one hand it is meant as a guide to the theory. We decided to write such a guide when we were learning the theory. It turned out that results of various sources needed to be combined and we felt that it might be a good idea to have them in one place. The sources are \cite{Ranicki(1979),Ranicki(1981),Levitt-Ranicki(1987),Ranicki(1992)}, and also \cite{Weiss-I(1985),Weiss-II(1985)}. On the other hand, we found certain statements which were correct, but without proofs, which we were able to supply. These are:
\begin{itemize}

\item The fact that the quadratic boundary of a certain (normal, Poincar\'e) geometric pair associated to a degree one normal map from a manifold to a Poincar\'e space agrees with the surgery obstruction of that map is proved in our Example \ref{expl:normal-symm-poincare-pair-gives-quadratic}. The claim was stated in \cite[page 622]{Ranicki(1981)} without proof. The proposition preceding the claim suggests the main idea of the proof, but we felt that writing it down is needed.

\item The construction of the normal signature $\nsign_X (X)$ in section 11 for an $n$-dimensional geometric Poincar\'e complex $X$. This was claimed to exist in \cite[Example 9.12]{Ranicki(1992)} (see also \cite[Errata for page 103]{Errata-Ranicki(1992)}), for $X$ any $n$-dimensional geometric normal complex. We provide details of this construction when $X$ is Poincar\'e, which is enough for our purposes.


\item In the proof of Theorem \ref{thm:lifts-vs-orientations} a certain map has to be identified with the surgery obstruction map. The identification was claimed in \cite[page 291]{Ranicki(1979)} without details. Theorem \ref{thm:lifts-vs-orientations} is also essentially equivalent to \cite[Proposition 16.1]{Ranicki(1992)}, which has a sketch proof and is referenced back to \cite{Ranicki(1979)} for further details.

\item The relation between the quadratic complex associated to a degree one normal map from a manifold to a normal complex and the quadratic boundary of the normal complex itself as described in Proposition \ref{prop:degree-one-normal-map-mfd-to-normal-cplx}. We also provide the proof of Proposition \ref{prop:degree-one-normal-map-mfd-with-boundary-to-normal-pair} which is a relative version of Proposition \ref{prop:degree-one-normal-map-mfd-to-normal-cplx} and it is also an ingredient in the proof of Proposition \ref{prop:degree-one-normal-map-mfd-to-poincare-over-X} which gives information about the quadratic signature over $X$ of a degree one normal map from a manifold to a Poincar\'e complex $X$. Proposition \ref{prop:degree-one-normal-map-mfd-to-normal-cplx} was stated as \cite[Proposition 7.3.4]{Ranicki(1981)}, but only contained a sketch proof. Proposition \ref{prop:degree-one-normal-map-mfd-to-poincare-over-X} is used in the proof of \cite[Theorem 17.4]{Ranicki(1992)}.
\end{itemize}

Over time we have also heard from several other mathematicians in the area the need for such clarifications. We believe that with this paper we provide an answer to these questions and that the proof of the main theorem as presented here is complete. We also hope that our all-in-one-package paper makes the presentation of the whole theory surrounding the total surgery obstruction more accessible. We would be grateful for comments from an interested reader should there still be unclear parts.

It should be noted however, that we do not bring new technology to the proof, nor do we state any new theorems. Our supplying of the proofs as listed above is in the spirit of the two main sources \cite{Ranicki(1979)} and \cite{Ranicki(1992)}.

\subsection*{Structure}
The reader will recognize that our table of contents closely follows part I and the first two sections of part II of the book \cite{Ranicki(1992)}. We find most of the book a very good and readable source. So in the background parts of this paper we confine ourselves to survey-like treatment. In the parts where we felt the need for clarification, in particular the proof of the main theorem, we provide the details.

The reader of this article should be familiar with the classical surgery theory and at least basics of the algebraic surgery theory. Sections \ref{sec:algebraic-cplxs} to \ref{sec:surgery-sequences} contain a summary of the results from part I of \cite{Ranicki(1992)} which are needed to explain the theory around the main problem, sometimes equipped with informal comments. The reader familiar with these results can skip those sections and start reading section \ref{sec:normal-signatures-over-X}, where the proof of the main theorem really begins. In case the reader is familiar with everything except normal complexes, he may consult in addition section \ref{sec:normal-cplxs}.

\subsection*{Literature}
Besides the above mentioned sources some background can be found in \cite{Ranicki-I-(1980)}, \cite{Ranicki-II-(1980)}, \cite{Ranicki-foundations-(2001)}, \cite{Ranicki-structure-set-(2001)}.

\subsection*{Note.} Parts of this work will be used in the PhD thesis of Philipp K\"uhl.

\subsection*{Acknowledgments.}
We would like to thank Andrew Ranicki for stimulating lectures on algebraic surgery in Fall 2008, for answering numerous questions, and for generous support. We would also like to thank Ian Hambleton for inspiring conversations and Frank Connolly, Diarmuid Crowley, Qayum Khan, Wolfgang L\"uck and Martin Olbermann for comments and suggestions on the first draft of this paper.


\section{Algebraic complexes} \label{sec:algebraic-cplxs}


In this section we briefly recall the basic concepts of algebraic surgery. The details can be found in \cite{Ranicki-I-(1980),Ranicki-II-(1980)} and \cite[chapter 1]{Ranicki(1992)}.

Throughout the paper $\AA$ denotes an additive category and $\BB(\AA)$ denotes the category of bounded chain complexes in $\AA$. The total complex of a double chain complex can be used to extend a contravariant functor $T \co \AA \ra \BB (\AA)$ to a contravariant functor $T \co \BB(\AA) \ra \BB (\AA)$ as explained in detail in
\cite[page 26]{Ranicki(1992)}.

\begin{defn} \label{defn:chain-duality}
A \emph{chain duality} on an additive category $\AA$ is a pair
$(T,e)$ where
    \begin{itemize}
        \item $T$ is a contravariant functor $T:\AA\ra\BB(\AA)$
        \item $e$ is a natural transformation
              $e:T^{2}\ra(\id:\AA\ra\BB(\AA))$ such that
            \begin{itemize}
                \item $e_{M}:T^{2}(M)\ra M$ is a chain equivalence.
                \item $e_{T(M)}\circ T(e_{M}) = \id$.
            \end{itemize}
    \end{itemize}
\end{defn}

The extension $T \co \BB(\AA) \ra \BB (\AA)$ mentioned before the
definition defines the dual $T(C)$ for a chain complex $C \in \BB
(\AA)$. A chain duality $T \colon \AA \rightarrow \BB(\AA)$ can be
used to define a tensor product of two objects $M$, $N$ in $\AA$
over $\AA$ as
\begin{equation} \label{tensor-product}
M \otimes_\AA N = \Hom_\AA (T(M),N),
\end{equation}
which is a priori just a chain complex of abelian groups. This definition generalizes for chain complexes
$C$ and $D$ in $\BB(\AA)$:
\[
C\otimes_{\AA}D \coloneqq \Hom_{\AA}(T(C),D).
\]

\begin{expl} \label{expl:R-duality}
Let $R$ be a ring with involution $r\mapsto\bar r$, for example for $R = \ZZ[\pi]$, the group ring of a group $\pi$, we have involution given by $\bar g=g^{-1}$ for $g\in\pi$.

The category $\AA(R)$ of finitely generated free left $R$-modules possesses a chain duality by $T(M) = \Hom_R (M,R)$. The involution can be used to turn an a~priori right
$R$-module $T(M)$ into a left $R$-module. The dual $T(C)$ of a bounded chain complex $C$ over $R$ is $\Hom_R(C,R)$.
\end{expl}

Chain duality is important because it enables us to define various concepts of Poincar\'e duality as we will see. Although the chain dual $T(M)$ in the above example is concentrated in dimension $0$, this is not necessarily the case in general. In section \ref{sec:cat-over-cplxs} we will see examples where this generality is important.

\begin{notation}\label{defn:W}
Let $W$ and $\widehat{W}$  be the canonical free $\ZZ[\ZZ_2]$-resolution and the free periodic $\ZZ[\ZZ_2]$-resolution of $\ZZ$ respectively:
\[
 \xymatrix{
      W := &
      \ldots   \ar[r]^-{1+T} &
      \Ztwo    \ar[r]^-{1-T} &
      \Ztwo    \ar[r]      &
      0 \quad & \quad
 }
\]
\[
 \xymatrix{
      \widehat{W} := &
      \ldots   \ar[r]^{1+T} &
      \Ztwo    \ar[r]^{1-T} &
      \Ztwo    \ar[r]^{1+T} &
      \Ztwo    \ar[r]^-{1-T} &
      \ldots
    }
\]
\end{notation}

The chain duality $T$ can be used to define an involution $T_{C,C}$ on $C\otimes_{\AA}C$ which makes it into a $\ZZ[\ZZ_{2}]$-module chain complex, see \cite[page 29]{Ranicki(1992)}.

\begin{defn} We have the following chain complexes of abelian
groups:
\begin{align*}
\Wq C & \coloneqq W\otimes_{\ZZ[\ZZ_{2}]}(C\otimes_{\AA}C) \\
\Ws C & \coloneqq \Hom_{\ZZ[\ZZ_{2}]}(W,C\otimes_{\AA}C) \\
\Wh C & \coloneqq \Hom_{\ZZ[\ZZ_{2}]}(\widehat{W},C\otimes_{\AA}C)
\end{align*}
\end{defn}

\begin{notation}\label{defn:fstar}
Let $f \colon C \rightarrow D$ be a chain map in $\BB(\AA)$. Then
the map of $\ZZ[\ZZ_2]$-chain complexes $f \otimes f \colon C
\otimes_\AA C \rightarrow D \otimes_\AA D$ induces chain maps
\[
f_{\%} \colon \Wq{C} \rightarrow \Wq{D} \quad f^{\%} \colon \Ws{C}
\rightarrow \Ws{D} \quad \widehat{f}^{\%} \colon \Wh{C} \rightarrow
\Wh{D}
\]
\end{notation}

\begin{defn} \label{defn:structures-on-chain-complexes}
Let $C$ be a chain complex in $\BB(\AA)$. An $n$-dimensional
\emph{symmetric structure} on $C$ is an $n$-dimensional cycle
$\varphi \in \Ws C _n$. An $n$-dimensional \emph{quadratic
structure} on $C$ is an $n$-dimensional cycle $\psi \in \Wq C _n$.
An $n$-dimensional \emph{hyperquadratic structure} on $C$ is an
$n$-dimensional cycle $\theta \in \Wh C _n$.
\end{defn}

Note that the dimension $n$ refers only to the degree of the element $\varphi$, $\psi$, or $\theta$ and does not mean that the chain complex $C$ has to be concentrated between degrees $0$ and $n$. 

\begin{notation} \label{notn:suspension}
On chain complexes we use the operations of {\it suspension} defined by $(\Sigma C)_n = C_{n-1}$ and {\it desuspension} defined by $(\Sigma^{-1} C)_n = C_{n+1}$. If $X$ is a well-based topological space we can consider the reduced suspension $\Sigma X$. For the singular chain complexes $C(X)$ and $C(\Sigma X)$ and there is a natural chain homotopy equivalence which we denote $\Sigma \co C(X) \ra \Sigma^{-1} C (\Sigma X)$, see \cite[section 1]{Ranicki-I-(1980)} if needed. Sometimes we use the same symbol for the associated map of degree one of chain complexes $\Sigma \co C(X) \ra C(\Sigma X)$. 
\end{notation}

\begin{remark}
The structures on a chain complex $C$ from Definition \ref{defn:structures-on-chain-complexes} can also be described in terms of theirs components. Abbreviating $C^{m-\ast} = \Sigma^m TC$, an element $\varphi \in \Ws C _n$ is a collection of maps $\{ \varphi_s \colon C^{n+s-\ast} \rightarrow C | s \in \mathbb{N} \}$, an element $\psi \in \Wq C _n$ is a collection of maps $\{ \psi_s \colon C^{n-s-\ast} \rightarrow C | s \in \mathbb{N} \}$, and an element $\theta \in \Wh C _n$ is a collection of maps $\{ \theta_s \colon C^{n+s-\ast} \rightarrow C | s \in \ZZ \}$, all of them satisfying certain identities, see \cite[page 30]{Ranicki(1992)}. In the symmetric case these identities describe each $\varphi_s$ as a chain homotopy between $\varphi_{s-1}$ and $T\varphi_{s-1}$.
\end{remark}

\begin{defn} \label{defn:Q-grps}
For a $C \in \BB(\AA)$ the \emph{$Q$-groups} of $C$ are defined by
\[
\Qq n C = H_n (\Wq C) \qquad \Qs n C = H_n (\Ws C) \qquad \Qh n C =
H_n (\Wh C)
\]
\end{defn}

\begin{prop} \textup{\cite[Proposition 1.2]{Ranicki-I-(1980)}} \label{propn:LESQ}
For a chain complex  $C \in \BB(\AA)$ we have a long exact sequence of $Q$-groups
  \[
    \xymatrix{
      \ldots          \ar[r]             &
      \Qq{n}{C}       \ar[r]^-{1+T}      &
      \Qs{n}{C}       \ar[r]^-J          &
      \Qh{n}{C}       \ar[r]^-H          &
      \Qq{n-1}{C}     \ar[r]             &
      \ldots
    }
  \]
\end{prop}

The sequence is induced from the short exact sequence of chain complexes
\[
\xymatrix{ 0 \ar[r] & \Ws{C} \ar[r] & \Wh{C} \ar[r] & \Sigma \Wq{C}
\ar[r] & 0 }
\]
The connecting map
\begin{equation} \label{eqn:symmetrization-map}
1+T \co \Wq C \ra \Ws C \qquad ((1+T)\psi)_s = \begin{cases} (1+T)\psi_0 & \textup{if} \; s = 0 \\ 0 & \textup{if} \; s \geq 1 \end{cases}
\end{equation}
is called the \emph{symmetrization map}.

\begin{definition}\label{defn:nSAPC-nQAPC}
An $n$-dimensional \emph{symmetric algebraic complex}
(SAC) in $\AA$ is a pair $(C,\varphi)$ where $C \in \BB(\AA)$ and
$\varphi$ is an $n$-dimensional symmetric structure on $C$. It is called \emph{Poincar\'{e}} (SAPC) if
$\varphi_0$ is a chain homotopy equivalence.

An $n$-dimensional \emph{quadratic algebraic complex}
(QAC) in $\AA$ is a pair $(C,\psi)$ where $C \in \BB(\AA)$ and
$\psi$ is an $n$-dimensional quadratic structure on $C$. It is called \emph{Poincar\'{e}} (QAPC) if
$((1+T) \cdot \psi)_0$ is a chain homotopy equivalence.
\end{definition}

An analogous notion for hyperquadratic complexes is not defined. The
following construction helps to understand the exact sequence of
Proposition \ref{propn:LESQ}.

\begin{definition}\label{defn:suspension}
Let~$C$ be a chain complex $\BB(\AA)$. The \emph{suspension} maps
\[
S \colon \Ws{C} \rightarrow \susp^{-1} (\Ws{\susp C}) \qquad S
\colon \Wh{C} \rightarrow \susp^{-1} (\Wh{\susp C})
\]
are defined by
\[
  (S(\varphi))_k := \varphi_{k-1} \qquad (S(\theta))_k := \theta_{k-1}
\]
\end{definition}

\begin{prop}  \label{prop:sups-Q-groups}
The hyperquadratic $Q$-groups are the stabilization of the symmetric
$Q$-groups:
\[
  \Qh{n}{C} = \underset{k \rightarrow \infty}{\textup{colim}} \,
  \Qs{n+k}{ \susp^k C }.
\]
Moreover, the suspension induces an isomorphism on hyperquadratic
$Q$-groups:
\[
 S \co \Qh{n}{C} \xra{\cong} \Qh{n+1}{\Sigma C}.
\]
\end{prop}
The proposition is proved in \cite[section 1]{Ranicki-I-(1980)}. It
follows that a symmetric structure has a quadratic refinement if and
only if its suspension $S^k$ is zero in $\Qs{n+k}{\Sigma^k C}$ for some $k$.
This can be improved in a sense that a preferred quadratic refinement can be chosen if a preferred path of the suspension $S^k$ to $0$ is chosen in $\Sigma^{-k} \Ws{\Sigma^k C}$.

\begin{rem} \label{rem:Qh-is-cohomology}
There exists the operation of a direct sum on the structured chain complexes \cite[section 1]{Ranicki-I-(1980)}. We remark that the quadratic and symmetric $Q$-groups do not respect this operation, but the hyperquadratic $Q$-groups do. In fact the assignments $C \mapsto \Qh{n}{C}$ constitute a generalized cohomology theory on the category of chain complexes in $\AA$, see \cite[Theorem 1.1]{Weiss-I(1985)}.
\end{rem}

Now we proceed to explain how the above structures arise from
geometric examples.

\begin{construction}\label{constrn:symmetric construction}
\cite[Proposition 1.1,1.2]{Ranicki-II-(1980)} Let~$X$ be a topological space with the singular chain complex~$C(X)$. The Alexander-Whitney diagonal approximation gives a chain map
\[
  \varphi \colon C(X) \rightarrow \Ws{C(X)},
\]
called the \emph{symmetric construction on $X$}, such that for every $n$-dimensional cycle $[X] \in C(X)$, the component $\varphi([X])_0 \co C^{n-\ast} (X) \ra C(X)$ is the cap product with the cycle $[X]$.

There exists an equivariant version as follows. Let~$\tilde{X}$ be the universal cover of~$X$. The singular chain complex~$C(\tilde{X})$ is a chain complex over~$\Zpi$. The symmetric construction $\varphi_{\tilde{X}}$ on $\tilde{X}$ produces a chain map of $\Zpi$-modules. Applying $\ZZ \otimes_{\Zpi}$ we obtain a chain map of chain complexes of abelian groups
\[
  \varphi \colon C(X) \rightarrow \Ws{C(\tilde{X})} = \Hom_{\ZZ[\ZZ_2]} (W,C(\tilde{X}) \otimes_{\ZZ[\pi_1 (X)]}  C(\tilde{X})),
\]
still called the \emph{symmetric construction of $X$}, and such that for every cycle $[X]\in C(X)$, the component $\varphi([X])_0 \co C^{n-\ast} (\tilde{X}) \ra C(\tilde{X})$ is the cap product with the cycle $[X]$, but now we obtain a map of $\Zpi$-module chain complexes. There is also a version of it for pointed spaces where one works with reduced chain complexes $\tilde C (\tilde X)$.

If $X$ is an $n$-dimensional geometric Poincar\'{e} complex with the fundamental class~$[X]$, then $\varphi ([X])_0$ is the Poincar\'{e} duality chain equivalence. In this case we obtain an $n$-dimensional SAPC over $\ZZ[\pi_1 (X)]$
\[
(C(\tilde X),\varphi ([X])).
\]
\end{construction}

The symmetric construction is functorial with respect to maps of topological spaces and natural with respect to the suspension of chain complexes, as shown in \cite[Proposition 1.1, 1.2]{Ranicki-II-(1980)}. However, if we have a chain map $C(X) \ra C(Y)$ not necessarily induced by a map of spaces, it might not commute with the symmetric constructions of $X$ and $Y$. This is one motivation for the quadratic construction below.

\begin{construction}\label{constrn:quadconstrn}
\cite[Proposition 1.5]{Ranicki-II-(1980)} Let~$X,Y$ be pointed spaces and let $F \colon \Sigma^k X \rightarrow \Sigma^k Y$ be a map. Denote 
\[
f \colon C(X) \overset{\Sigma}{\rightarrow} \susp^{-k} C( \Sigma^k X)
\overset{F}{\rightarrow} \susp^{-k} C( \Sigma^k Y) \overset{\Sigma^{-1}}{\rightarrow} C(Y).
\]
where $\Sigma^{-1}$ is some homotopy inverse of $\Sigma$ from Notation \ref{notn:suspension}. The following diagram does not necessarily commute, since~$f$ does not come from a geometric map
\[
    \xymatrix{
      C (X)  \ar[r]^-{\varphi}  \ar[d]^{f_{\ast}}   &
      \Ws{C(X)}                  \ar[d]^{f^{\%}}     \\
      C (Y)  \ar[r]^-{\varphi}                     &
      \Ws{C(Y)}
    }
\]
There is a chain map, called the \emph{quadratic construction on
$F$},
\[
  \psi \colon C (X) \rightarrow \Wq{C(Y)} \quad \textup{such that}
  \quad (1+T) \cdot \psi \equiv f^{\%} \varphi - \varphi f_{\ast}.
\]
To show that such a map exists we look at the difference~$f^{\%}
\varphi - \varphi f_{\ast}$, and use Proposition
\ref{propn:LESQ} to obtain the commutative diagram
  \[
    \xymatrix@C=1.25cm{
      &  &  H_n(X)
                \ar@{-->}[dl]_{\Psi_F}
                \ar[d]|{f^{\%} \varphi - \varphi f_{\ast}}
                \ar[dr]^{\equiv 0}
      \\
      \ldots              \ar[r]  &
      \Qq{n}{C(Y)}        \ar[r]  &
      \Qs{n}{C(Y)}        \ar[r]  &
      \Qh{n}{C(Y)}        \ar[r]  &
      \ldots
    }
  \]
The map~$H_n(X) \rightarrow \Qh{n}{C(Y)}$ is the stabilization of the map~$f^{\%} \varphi - \varphi f_{\ast}$. But when we stabilize~$f$ we recover the map~$F \colon C(\Sigma^kX) \rightarrow C(\Sigma^kY)$, up to a preferred chain homotopy. This map comes from a geometric map, and so, by the naturality of the symmetric construction, the map~$H_n(X) \rightarrow \Qh{n}{C(Y)}$ is zero.

Then exactness tells us there is a lift. However, we are allowed to look on the chain level, and we observe that there is a preferred null-homotopy of the difference $S^k (f^{\%} \varphi - \varphi f_{\ast}) \simeq F^{\%} \varphi - \varphi F_{\ast}$ in the chain complex $\Sigma^{-k} \Ws{C(\Sigma^k Y)}$. By the remark following Proposition \ref{prop:sups-Q-groups} we obtain a preferred lift. This is describes the map~$\psi$, the full details can be found in \cite{Ranicki-II-(1980)}.

Similarly as in the symmetric construction there is an equivariant version, also called \emph{quadratic construction on $F$},
\[
  \psi \colon C (X) \rightarrow \Wq{C(\tilde{Y})} \quad \textup{such that}
  \quad (1+T) \cdot \psi \equiv f^{\%} \varphi - \varphi f_{\ast}.
\]
\end{construction}

\begin{construction} \label{con:quad-construction-on-degree-one-normal-map}
Let~$M$, $X$ be geometric Poincar\'{e} complexes with a degree one normal map~$(f,b) \colon M \rightarrow X$. Using $\pi_1(X)$-equivariant $S$-duality
(see section \ref{sec:normal-cplxs}) we obtain a stable equivariant map~$F \colon \Sigma^k \tilde{X}_+ \rightarrow \Sigma^k \tilde{M}_+$ for some~$k \in \mathbb{N}$.
Consider the Umkehr map
\[
f^{!} \colon C(\tilde X) \rightarrow \susp^{-k} C( \Sigma^k \tilde X_+)
\overset{F}{\rightarrow} \susp^{-k} C( \Sigma^k \tilde M_+) \rightarrow C(\tilde M)
\]
and its mapping cone $\sC (f^{!})$ with the inclusion map $e \co C(\tilde M) \ra \sC (f^{!})$. We obtain an $n$-dimensional QAPC over $\Zpi$
\[
(\sC(f^{!}),e_\% \psi ([X])).
\]
\end{construction}

An example of a hyperquadratic structure on a chain complex coming from geometry is relegated to section \ref{sec:normal-cplxs}. Now we present the relative versions of the above concepts.

\begin{definition}\label{defn:sympair}
An ($n$+1)-dimensional \emph{symmetric algebraic pair} over $\AA$ is
a chain map $f \colon C \rightarrow D$ in $\BB(\AA)$ together with
an $(n+1)$-dimensional cycle $(\delta \varphi,\varphi) \in \sC
(f^{\%})$. An ($n$+1)-dimensional \emph{quadratic algebraic pair}
over $\AA$ is a chain map $f \colon C \rightarrow D$ in $\BB(\AA)$
together with an $(n+1)$-dimensional cycle $(\delta \psi,\psi) \in
\sC (f_{\%})$.
\end{definition}

Notice that an $(n+1)$-dimensional symmetric pair contains an
$n$-dimensional symmetric complex $(C,\varphi)$ and similarly an
$(n+1)$-dimensional quadratic pair contains an $n$-dimensional
quadratic complex $(C,\psi)$. The cycle condition translates into the relation between $\delta \varphi$
and $\varphi$ via the equation $d(\delta\varphi) = (-1)^n f^{\%}
(\varphi)$. It is also helpful to define the evaluation map
\[
\textup{ev} \co \sC (f^{\%}) \ra \Hom_\AA (D^{n+1-\ast},\sC(f))
\quad \textup{ev} (\delta \varphi,\varphi) = \pairmap \colon
D^{n+1-\ast} \rightarrow \sC(f)
\]
and likewise in the quadratic case.

\begin{definition}\label{defn:SAPP-and-QAPP}
An $(n+1)$-dimensional \emph{symmetric algebraic \emph{Poincar\'{e}}
pair} (SAPP) in $\AA$ is a symmetric pair~$(f \colon C \rightarrow D
, (\delta\varphi,\varphi))$ such that
  \[
    \pairmap \colon D^{n+1-\ast} \rightarrow \sC (f)
  \]
is a chain equivalence.

An $(n+1)$-dimensional \emph{quadratic algebraic \emph{Poincar\'{e}}
pair} (QAPP) in $\AA$ is a quadratic pair~$(f \colon C \rightarrow D
, (\delta\psi,\psi))$ such that
  \[
    (1+T) \cdot \pairmapquad \colon D^{n+1-\ast} \rightarrow \sC (f)
  \]
is a chain equivalence.
\end{definition}


\begin{construction} \label{con:rel-sym} Let $(X,Y)$ be a pair of
topological spaces, and denote the inclusion $i \co Y \ra X$. By the
naturality of the symmetric construction we obtain a chain map
\[
 \varphi \co C(X,Y) \ra \sC (i^{\%})
\]
which is called the \emph{relative symmetric construction}.

If $(X,Y)$ is an $(n+1)$-dimensional Poincar\'e pair with the
fundamental class $[X] \in C_{n+1} (X,Y)$ then the evaluation
\[
\textup{ev} \circ \varphi ([X]) \co C^{n+1-\ast} (X) \ra
C(X,Y)
\]
is a chain homotopy equivalence. There also exists an equivariant version.
\end{construction}

\begin{construction} \label{con:rel-quad-htpy}
Let $(X,A)$ and $(Y,B)$ be pairs of pointed topological spaces and
let
\[
 \xymatrix{
 \Sigma^k A \ar[r]^{\del F} \ar[d]_{i} & \Sigma^k B \ar[d]^{j} \\
 \Sigma^k X \ar[r]_{F} & \Sigma^k Y
}
\]
be a commutative diagram. Let $\del f$ and $f$ be maps defined
analogous to the map $f$ in Construction \ref{constrn:quadconstrn}.
There is a chain map, the \emph{relative quadratic construction},
\[
 \psi \co C (X,A) \ra \sC (j_{\%})
\]
such that $(1+T) \cdot \psi = (f,\del f)^{\%} \varphi - \varphi (f,\del f)_\ast$. Again, there is also an equivariant version.
\end{construction}

\begin{construction} \label{con:rel-quad-mfds}
Let $((f,b),\del (f,b)) \co (M,N) \ra (X,Y)$ be a degree one normal
map of manifolds with boundary. Here we do not assume that the
restriction of $\del f$ on the boundary $N$ is a homotopy
equivalence. The $S$-duality yields in this case the 
commutative diagrams
\[
\xymatrix{
 T(\nu_M)/T(\nu_N) \ar[r] \ar[d] & T(\nu_X)/T(\nu_Y) \ar[d] &
 \leadsto & \Sigma^k Y_+ \ar[r]^{\del F} \ar[d]_{i} & \Sigma^k N_+
 \ar[d]^{j} \\
 \Sigma T(\nu_N) \ar[r] & \Sigma T(\nu_Y) & \leadsto &  \Sigma^k
 X_+ \ar[r]_{F} & \Sigma^k M_+
}
\]
We have two Umkehr maps $\del f^{!}$ and $f^{!}$ and a commutative
square
\[
 \xymatrix{
 C(N) \ar[r]^{\del e} \ar[d]_{j} & \sC(\del f^{!}) \ar[d]^{k} \\
 C(M) \ar[r]_{e} & \sC (f^{!})
}
\]
We obtain an $(n+1)$-dimensional QAPP
\[
\big( k \co \sC (\del f^{!}) \ra \sC (f^{!}),(e,\del e)_{\%} \psi ([X]) \big).
\]
\end{construction}

The notion of a pair allows us to define the notion of a cobordism
of structured chain complexes.

\begin{definition}\label{defn:symcobord}

A \emph{cobordism} of $n$-dimensional SAPCs~$(C,\varphi),(C',\varphi')$ in $\AA$ is
an~$(n+1)$-dimensional SAPP in $\AA$
\[
    ((f \, f') \colon C \oplus C' \rightarrow E ,
      (\delta\varphi , \varphi \oplus -\varphi'))
\]

A \emph{cobordism} of $n$-dimensional QAPCs~$(C,\psi),(C',\psi')$ in $\AA$ is
an~$(n+1)$-dimensional QAPP in $\AA$
\[
    ((f \, f') \colon C \oplus C' \rightarrow E ,
      (\delta\psi , \psi \oplus -\psi'))
\]
\end{definition}

There is a notion of a {\it union} of two adjoining cobordisms in $\AA$ is defined in \cite[section 3]{Ranicki-I-(1980)}. Using it one obtains transitivity for the cobordisms and hence an equivalence relation.

Geometrically, one obtains a symmetric cobordism from a geometric Poincar\'e triad and a quadratic cobordism from a degree one normal map of geometric Poincar\'e triads.

Recall the well-known fact that using Morse theory any geometric
cobordism can be decomposed into elementary cobordisms which are in
turn obtained via surgery. Although it has slightly different properties there exists an analogous notion of
algebraic surgery which we now recall. For simplicity we will only discuss it in the symmetric case, although there is an analogous
notion for quadratic complexes.

\begin{construction} \cite[Definition 1.12]{Ranicki(1992)}
Let $(C,\varphi)$ be an $n$-dimensional symmetric complex. The
\emph{data} for an algebraic surgery on $(C,\varphi)$ is an
~$(n+1)$-dimensional symmetric pair~$(f \colon C \rightarrow D ,
(\delta\varphi,\varphi))$ . The \emph{effect} of the algebraic
surgery on $(C,\varphi)$ using $(f \colon C \rightarrow D ,
(\delta\varphi,\varphi))$ is the $n$-dimensional symmetric complex
$(C',\varphi')$ defined by
\[
  C' = \susp ^{-1} \cone{\smallpairmap}, \quad \varphi' =
  \Sigma^{-1} (e')^{\%} (\delta \varphi / \varphi)
\]
Here the map $e'$ is defined by the diagram
\[
  \xymatrix@R=0,4cm{
    C  \ar[r]_{f}  &  D  \ar[rd]    &                     &  \susp C'     \\
               &                    &  \sC (f)  \ar[ur]_{e'}  &                \\
    C' \ar[r]  &  D^{n+1-\ast}   \ar[ur]_{\smallpairmap}
  }
\]
The symmetric structure on the pair $f \co C \ra D$ defines a
symmetric structure $\delta \varphi/\varphi$ on $\sC (f)$ by the
formula as in \cite[Proposition 1.15]{Ranicki(1992)}. It is pushed
forward by $e'$ to an $(n+1)$-cycle of~$\Ws{\susp C'}$ which turns
out to have a preferred desuspension and so we obtain an~$n$-cycle~$\varphi '$.
\end{construction}

A geometric analogue is obtained from a cobordism $W$ between closed
manifolds $M$ and $M'$. Then we have a diagram
\[
  \xymatrix@R=0.4cm{
    C(M)  \ar[r]  &  C(W,M')  \ar[rd]-<35pt,-5pt>    &                \\
                  &                                  &  C(W,M \cup M')  \\
    C(M') \ar[r]  &  C(W,M)   \ar[ur]-<35pt,5pt>
  }
\]
where the chain complexes~$C(W,M^\prime)$ and~$C(W,M)$ are Poincar\'e dual.

\begin{definition}\label{defn:symbdy}
Let~$(C,\varphi)$ be an $n$-dimensional SAC. The \emph{boundary}
of~$(C,\varphi)$ is the $(n-1)$-dimensional SAC obtained from surgery on the
symmetric pair $( 0 \rightarrow C , ( \varphi , 0 ) )$. The boundary
is denoted $\partial (C,\varphi) = (\partial C, \partial \varphi)$,
with $\del C = \Sigma^{-1} \sC (\varphi_0)$ and $\del \varphi = S^{-1}
e^{\%} (\varphi)$, where $e \co C \ra \sC (\varphi_0)$.
\end{definition}

Here the geometric analogue arises from considering an
$n$-dimensional manifold with boundary, say $(N,\del N)$. Consider
the chain complex $C(N,\del N)$ and its suspended dual $C^{n-\ast}
(N,\del N)$. There is a symmetric structure on $C(N,\del N)$, which
is not Poincar\'e. However, there is the Poincar\'e duality
$C^{n-\ast} (N,\del N) \simeq C(N)$. Thus the mapping cone of the
duality map $C^{n-\ast} (N,\del N) \ra C(N,\del N)$ becomes homotopy
equivalent to the mapping cone of the map $C(N) \ra C(N,\del N)$
which is $\Sigma C(\del N)$.

\begin{remark}
Notice that an $n$-dimensional SAC is Poincar\'{e} if and only if its boundary
is contractible.
\end{remark}

We also have the following proposition which is proven in \cite{Ranicki-I-(1980)} by writing out the formulas.

\begin{proposition}\textup{\cite[Proposition 4.1]{Ranicki-I-(1980)}} \label{prop:homotopy-type-of-boundary}
Algebraic surgery preserves the homotopy type of the boundary
of~$(C,\varphi)$. In particular we have that
\[
  (C,\varphi) \text{ is Poincar\'{e}} \Leftrightarrow
  (C',\varphi') \text{ is Poincar\'{e}}.
\]
\end{proposition}

An algebraic surgery on $(C,\varphi)$ using $(f \co C \ra D,(\delta \varphi,\varphi))$ gives rise to a symmetric pair $(f \; f'
\co C \oplus C' \ra D',(\delta \varphi',\varphi \oplus \varphi))$
with $D' = \sC (\varphi_0 f^\ast)$. If $(C,\varphi)$ is Poincar\'e
then, as noted above, $(C',\varphi')$ is also Poincar\'e, and in
addition the pair is a cobordism. We remark that the data for
algebraic surgery might not be a Poincar\'e pair, in fact this is a
typical case, since if it is a Poincar\'e pair, then it already
defines a null-cobordism of $(C,\varphi)$ and hence $C'$ is contractible.

The relationship between the algebraic cobordism and algebraic
surgery turns out to be as follows:

\begin{prop} \textup{\cite[Proposition 4.1]{Ranicki-I-(1980)}}
  The equivalence relation generated by surgery and homotopy equivalence is the same as the equivalence relation given by cobordism.
\end{prop}

\begin{definition}\label{defn:Lgps} \cite[Proposition 3.2]{Ranicki-I-(1980)}\

The \emph{symmetric $L$-groups} of an additive category with chain duality $\AA$ are
  \[
    L^n (\AA)  :=  \{ \text{cobordism classes of } n \text{-dimensional SAPCs in } \AA \}
  \]

The \emph{quadratic $L$-groups} an additive category with chain duality $\AA$ are
  \[
    L_n (\AA)  :=  \{ \text{cobordism classes of } n \text{-dimensional QAPCs in } \AA \}
  \]
The group operation is the direct sum of the structured chain complexes in both cases. The inverse of a SAPC $(C,\varphi)$ is given by $(C,-\varphi)$, and the inverse of a QAPC $(C,\psi)$ is given by $(C,-\psi)$.
\end{definition}

\begin{rem}
It is proven in \cite[sections 5,6,7]{Ranicki-I-(1980)} for $\AA = \AA (R)$, where $R$ is a ring with involution, that the groups $L_n (\AA(R))$ are isomorphic to the surgery obstruction groups $L_n (R)$ of Wall. Both symmetric and quadratic groups $L^n (\AA)$ and $L_n (\AA)$ are $4$-periodic for any $\AA$ \cite[Proposition 1.10]{Ranicki(1992)}.
\end{rem}

\begin{definition} \label{defn:sym-sign}
Let $X$ be an $n$-dimensional Poincar\'e complex. The cobordism class of the $n$-dimensional SAPC obtained from any choice of the fundamental class $[X] \in C_n (X)$ in Construction \ref{constrn:symmetric construction} does not depend on the choice of $[X]$ and hence defines an element
\[
 \ssign_{\Zpi} (X) = [(C(\tilde X),\varphi ([X]))] \in L^n (\Zpi).
\]
called the \emph{symmetric signature} of $X$.\footnote{The notation is somewhat premature, the symbol $\bL^\bullet$ denotes the symmetric $L$-spectrum and will be defined later in section \ref{sec:spectra}. Likewise in the quadratic case.} If $X$ is an oriented $n$-dimensional topological manifold, then the symmetric signature only depends on the oriented cobordism class of $X$, and so it provides us with a homomorphism\footnote{If $X$ is not a manifold we can still say that the symmetric signature only depends on the oriented cobordism class of $X$ in the Poincar\'e cobordism group $\Omega^{P}_n$, but we will not need this point of view later}
\[
 \ssign_{\Zpi} \co \Omega^{\STOP}_n (K(\pi_1 (X),1)) \ra L^n (\ZZ[\pi_1 (X)]).
\]
\end{definition}

\begin{definition} \label{defn:quad-sign}
Let $(f,b) \co M \ra X$ be a degree one normal map of Poincar\'e complexes. The cobordism class of the $n$-dimensional QAPC obtained from any choice of the fundamental class $[X] \in C_n (X)$ in Construction \ref{con:quad-construction-on-degree-one-normal-map} does not depend on the choice of $[X]$ and hence defines an element
\[
\qsign_{\Zpi} (f,b) = [(\sC(f^{!}),e_\% \psi ([X]))] \in L_n (\Zpi)
\]
called the \emph{quadratic signature} of the degree one normal map $(f,b)$. If $M$ is an $n$-dimensional oriented manifold then the quadratic signature only depends on the normal cobordism class of $(f,b)$ in the set of normal invariants $\sN (X)$ and provides us with a function\footnote{Recall that for $X$ an $n$-dimensional GPC the set of normal invariants if it is non-empty is a group with respect to the group structure given by the Whitney sum. The quadratic signature is NOT a homomorphism with respect to this group structure.}
\[
 \qsign_{\ZZ[\pi_1(X)]} \co \sN(X) \ra L_n (\ZZ[\pi_1 (X)]).
\]
\end{definition}

\begin{rem} \label{rem:symmetrization-of-surgery-obstruction} \cite[Proposition 2.2]{Ranicki-II-(1980)}
The symmetrization map (\ref{eqn:symmetrization-map}) carries over to the $L$-groups as
\[
 (1+T) \co L_n (\AA) \ra L^n (\AA)
\]
and for $\pi = \pi_1 (X)$ we have
\[ 
(1+T) \; \qsign_{\ZZ[\pi]} (f,b) = \ssign_{\ZZ[\pi]} (M) - \ssign_{\ZZ[\pi]} (X).
\]
\end{rem}

\begin{rem} \label{rem:quadratic-signature-is-surgery-obstruction}
If $M$ is a closed $n$-dimensional topological manifold then the quadratic signature from Definition \ref{defn:quad-sign} coincides with the classical surgery obstruction by the result of \cite[Proposition 7.1]{Ranicki-II-(1980)}.
\end{rem}

\begin{rem}
Notice that we did not define a hyperquadratic version of the $L$-groups. In fact, hyperquadratic structures are useful when we
have a fixed chain complex $C$ and we study the relationship between
the symmetric and quadratic structures on $C$ via the sequence in
Proposition \ref{propn:LESQ}. When comparing the symmetric and
quadratic $L$-groups, hence cobordism groups of complexes equipped
with a symmetric and quadratic structure a new concept of an
algebraic normal complex is needed. It is discussed in the next
section.
\end{rem}



\section{Normal complexes} \label{sec:normal-cplxs}


A geometric normal complex is a notion generalizing a geometric
Poincar\'e complex. It is motivated by the observation that although
a Poincar\'e complex is not necessarily locally Poincar\'e, it is
locally normal. On the other hand a manifold is also locally
Poincar\'e. Hence the question whether a Poincar\'e complex can be
modified within the homotopy type so that the locally normal
structure becomes Poincar\'e is central to our main problem.

In this section we will recall the definition of an algebraic normal
complex. In addition we recall that cobordism groups of algebraic
normal complexes, the so-called $NL$-groups, which measure the
difference between the symmetric and quadratic $L$-groups. Another
viewpoint on that same fact is that the quadratic $L$-groups measure
the difference between the symmetric $L$-groups and the $NL$-groups.
This will be crucially used in the proof of the Main Technical
Theorem.

The material from this section comes from \cite[section
2]{Ranicki(1992)}, \cite{Weiss-I(1985),Weiss-II(1985)} and
\cite[sections 7.3 and 7.4]{Ranicki(1981)}.

\begin{definition}\label{defn:GNC}
An $n$-dimensional \emph{geometric normal complex} (GNC) is a triple
$(X,\nu,\rho)$ consisting of a space~$X$ with a $k$-dimensional
oriented spherical fibration~$\nu$ and a map $\rho \colon S^{n+k}
\rightarrow \thom{\nu}$ to the Thom space of $\nu$.

The \emph{fundamental class} of~$(X,\nu,\rho)$ is the
$n$-dimensional homology class in $H_n (X)$ represented by the cycle
$[X] \in C_n (X)$ given by the formula $[X] := u(\nu) \cap h(\rho)$
where~$h$ is the Hurewicz homomorphism, and $u (\nu) \in C^k
(\thom{\nu})$ is some choice of the Thom class of $\nu$.
\end{definition}

Note that the dimension of a GNC is the dimension of the source
sphere of the map $\rho$ minus the dimension of the spherical
fibration. It does not necessarily have anything to do with a
geometric dimension of $X$. Also the cap product with the
fundamental class does not necessarily induce an isomorphism between
cohomology and homology of $X$.

\begin{example}\label{expl:nGPCtonGNC}
Let~$X$ be an $n$-dimensional geometric Poincar\'{e} complex (GPC)
with the fundamental class~$[X]$ in the sense of Poincar\'e duality.
Then, for $k$ large enough, the space~$X$ has the Spivak normal
fibration (SNF)~$\nu_X \colon X \rightarrow \BSG(k)$, which has the
property that there is a map~$\rho_X \colon S^{n+k} \rightarrow
\thom{\nu_X}$ such that
\[
  [X] = u(\nu_{X}) \cap h( \rho_X ) \in H_n (X).
\]
Thus we get an $n$-dimensional geometric normal complex~$(X,
\nu_X,\rho_X)$ with the fundamental class equal to the fundamental
class in the sense of Poincar\'e duality.
\end{example}

Some properties of normal complexes can be stated in terms of the
$S$-duality from stable homotopy theory. For pointed spaces $X$, $Y$
the symbol $[X,Y]$ denotes the abelian  group of stable homotopy
classes of stable pointed maps from $X$ to $Y$. Here, for
simplicity, we confine ourselves to a non-equivariant $S$-duality.
An equivariant version, which is indeed needed for our purposes is
presented in detail in \cite[section 3]{Ranicki-I-(1980)}.

\begin{definition}\label{defn:Sduality}
  Let~$X,Y$ be pointed spaces. A map~$\alpha \colon S^N \rightarrow X
  \wedge Y$ is an $N$-dimensional \emph{S-duality map} if the slant
  product maps
  \[
    \alpha_{\ast} ([S^N]) \backslash \underbar{ }
    \colon \tilde{C}(X)^{N-\ast} \rightarrow
    \tilde{C}(Y) \quad \textup{and} \quad \alpha_{\ast} ([S^N]) \backslash
    \underbar{ } \colon \tilde{C}(Y)^{N-\ast}
    \rightarrow \tilde{C}(X)
  \]
  are chain equivalences. We say the spaces~$X,Y$ are \emph{S-dual}.
\end{definition}

\begin{expl} \label{expl:SNF}
Let $X$ be an $n$-dimensional GPC with the $k$-dimensional SNF
$\nu_X \co X \ra \BSG (k)$. Then $\thom{\nu_X}$ is an
$(n+k)$-dimensional $S$-dual to $X_+$.
\end{expl}

\begin{prop} \label{S-duality-property}
The $S$-duality satisfies:
\begin{enumerate}
\item For every finite CW-complex $X$ there exists an $N$-dimensional S-dual, which we denote $X^\ast$, for some large $N \geq 1$.
\item If $X^\ast$ is an $N$-dimensional $S$-dual of $X$ then $\Sigma X^\ast$ is an $(N+1)$-dimensional $S$-dual of $X$.
\item For any space $Z$ we have isomorphisms
\begin{align*}
S \co [X,Z] \cong [S^N,Z \wedge Y] & \quad \gamma \mapsto S(\gamma)
=
(\gamma \wedge \id_Y) \circ \alpha, \\
S \co [Y,Z] \cong [S^N,X \wedge Z] & \quad \gamma \mapsto S(\gamma)
= (\id_X \wedge \gamma) \circ \alpha.
\end{align*}
\item A map $f \co X \ra Y$ induces a map $f^\ast \co Y^\ast \ra X^\ast$ for $N$ large enough via the isomorphism
    \[
    [X,Y] \cong [S^N,Y \wedge X^\ast] \cong [Y^\ast,X^\ast].
    \]
\item If $X \ra Y \ra Z$ is a cofibration sequence then $Z^\ast \ra Y^\ast \ra X^\ast$ is a cofibration sequence for $N$ large enough.
\end{enumerate}
\end{prop}

The $S$-dual is also unique in some sense. In fact the assignment $X
\mapsto X^\ast$ can be made into a functor in an appropriate stable
homotopy category. As this requires a certain amount of
technicalities and we do not really need it, we skip this aspect.
The reader can find the details for example in \cite{Adams(1974)}.

Now we present a generalization of Example \ref{expl:SNF}.

\begin{construction}  \label{con:S-duality-and-Thom-is-Poincare}
Let~$(X,\nu,\rho)$ be an $n$-dimensional GNC. Let~$V$ be the mapping
cylinder of the projection map of $\nu$ with~$\del V$ being the
total space of the spherical fibration $\nu$. Then we have the
generalized diagonal map
\[
  \tilde{\Delta} \colon
  \thom{\nu} \simeq \frac{V}{\partial V}
    \overset{\Delta}{\longrightarrow}
  \frac{V \times V} {V \times \partial V}
    \simeq
  \thom{\nu}  \wedge  X_+
\]
where~$\Delta$ is the actual diagonal map. Consider the composite
\[
  S^{n+k}
    \overset{\rho}{\longrightarrow}
  \thom{\nu}
    \overset{\tilde{\Delta}}{\longrightarrow}
  \thom{\nu} \wedge X_+.
\]
By Proposition \ref{S-duality-property} part (1) we have an
$S$-duality map $S^N \ra \Th (\nu) \wedge \Th (\nu)^\ast$ for $N$
large enough. Setting $p = N-(n+k)$ we obtain from part (3) the
one-to-one correspondence:
\begin{align*}
 S^{-1} \co [S^{n+k},\thom{\nu} \wedge X_+] & \cong [\thom{\nu}^\ast,\Sigma^p X_+]  \\
\tilde{\Delta} \circ \rho & \mapsto \Gamma_X := S^{-1}
(\tilde{\Delta} \circ \rho).
\end{align*}
Moreover, we obtain the following homotopy commutative diagram in
which $\gamma_X$ is the chain map induced by $\Gamma_X$:
\begin{equation} \label{dgrm:Thom-S-dual-Poincare}
\begin{split}
    \xymatrix@R=10pt@C=1cm{
    C(X)^{n-\ast} \ar[d]_{- \cap [X]} \ar[r]^-{-\cup u (\nu)}
    &
        \tilde{C}(\thom{\nu})^{n+k-\ast} \ar[r]^-{\textup{S-dual}}
    &
        C(\thom{\nu}^\ast)_{p + \ast} \ar[d]^{\gamma_X}
    \\
    C(X) = \tilde{C}(X_+) \ar[rr]_{\Sigma^p}
    & &
    \tilde{C} (\Sigma^p X_+)_{p + \ast} }
\end{split}
\end{equation}
If $X$ is Poincar\'e then $p$ can be chosen to be $0$ and the maps
$\Gamma_X$ and $\gamma_X$ the identity. Hence the Poincar\'e duality
is seen as the composition of the Thom isomorphism for the SNF and
the $S$-duality.
\end{construction}

Now we turn to algebraic normal complexes. As a first step we
discuss the following notion which is an algebraic analogue of a
spherical fibration.

\begin{definition}\label{defn:chain bundle}
Let~$C$ be a chain complex over an additive category with chain duality ~$\AA$. A \emph{chain bundle}
over~$C$ is a $0$-dimensional cycle~$\gamma$ in~$\Wh{TC}$.
\end{definition}

 \begin{construction}\label{constrn:hypquadconstrn}
Let~$X$ be a finite CW-complex and let~$\nu \co X \ra \BSG (k)$ be a
$k$-dimensional spherical fibration over~$X$. The Thom space
$\thom{\nu}$ is also a finite CW-complex and hence has an
$N$-dimensional S-dual $\thom{\nu}^\ast$ for some $N$. The
\emph{hyperquadratic construction} is the chain map given by the
following composition:
\[
  \xymatrix@R=0.1cm@!C=1.5cm{
        *+[l]{\gamma_\nu \co \tilde{C}^k(\thom{\nu})} \ar[r]^-{\text{S-duality}}
     &  *+[r]{\tilde{C}_{N-k}(\thom{\nu}^\ast)}
     \\
     \ar[r]^-{\varphi_{\thom{\nu}^\ast}}
     &   *+[r]{\Ws{\tilde{C}(\thom{\nu}^\ast)}_{N-k}}
     \\
     \ar[r]^-{\text{S-duality}}
     &   *+[r]{\Ws{\tilde{C}(\thom{\nu})^{N-\ast}}_{N-k}}
     \\
     \ar[r]^-{\text{Thom}}
     &   *+[r]{\Ws{C(X)^{N-k-\ast}}_{N-k}}
     \\
     \ar[r]^-{J}
     &   *+[r]{\Wh{C(X)^{N-k-\ast}}_{N-k}}
     \\
     \ar[r]^-{S^{-(N-k)}}
     &  *+[r]{\Wh{C(X)^{-\ast}}_0}
  }
\]
Given a choice of the Thom class $u (\nu) \in
\tilde{C}^k(\thom{\nu})$, the cycle $\gamma_\nu (u(\nu))$ becomes a
chain bundle over $C(X)$. An equivariant version produces a chain
bundle over $\Zpi$:
\[
 (C(\tilde X),\gamma_\nu (u(\nu)))
\]
\end{construction}

Now we can define an algebraic analogue of a geometric normal
complex.

\begin{definition}\label{defn:nNAC}
  An \emph{$n$-dimensional normal algebraic complex} (NAC) in $\AA$ is a pair $(C,\theta)$ where $\theta$ is a triple~$(\varphi,\gamma,\chi)$ such that
  \begin{itemize}
    \item $(C,\varphi)$ is an $n$-dimensional SAC
    \item $\gamma \in (\Wh{TC})_{0}$ is a chain bundle over~$C$
    \item $\chi \in (\Wh{C})_{n+1}$ satisfies
              $d \chi = J (\varphi) - \varphizp (S^n \gamma)$.
  \end{itemize}
\end{definition}

As we indicate below in the geometric example the third condition is
a consequence of the homotopy commutativity of the diagram
(\ref{dgrm:Thom-S-dual-Poincare}) and as such can be seen as a
generalization of the equation in Example \ref{expl:nGPCtonGNC}.
Notice that there is no requirement on $\varphi_0$ being a chain
equivalence, that means normal complexes are in no sense Poincar\'e.
Now we indicate the \emph{normal construction} which to an
$n$-dimensional GNC functorially associates an $n$-dimensional NAC.
The full details are somewhat complicated, the reader can find them
in \cite{Weiss-I(1985),Weiss-II(1985)}.

\begin{construction} \label{con:normal-construction}
Let~$(X,\nu,\rho)$ be an $n$-dimensional GNC with a choice of the
Thom class $u(\nu) \in \tilde C (\Th (\nu))$ whose associated
fundamental class is denoted~$[X]$. We would like to associate to it
an $n$-dimensional NAC over $\ZZ[\pi_1 X]$. We start with
\begin{itemize}
    \item $C = C(\tilde X)$
    \item $\varphi = \varphi ([X])$
    \item $\gamma = \gamma_{\nu} (u(\nu))$
\end{itemize}
Now we will only show that an element~$\chi$ with required
properties exists. In other words we show that $J (\varphi) =
\varphizp (S^n \gamma)$ in $\Qh{n}{C(X)}$. Consider in our case the
symmetric construction, the hyperquadratic construction and the
diagram (\ref{dgrm:Thom-S-dual-Poincare}). We obtain the following
commutative diagram:

{\footnotesize
\begin{center}
    \xymatrix@C=0.45cm{
    &
        H_n (X) \ar[r]^(0.45){\varphi} \ar[d]^{\Sigma^p}
    &
        Q^n (C(X)) \ar[r]^{J} \ar[d]^{\Sigma^p}
    &
        \Qh{n}{C(X)}
        \ar[d]_{\Sigma^p}^{\cong}
    &
        \Qh{n}{C^{n-\ast} (X)}
        \ar[l]_{\widehat\varphi_0^{\%}}
        \ar[d]_{\Sigma^p}^{\cong}
    \\
        H^k (\Th(\nu))
        \ar[ur]^{- \cap h (\rho)} \ar[r] \ar[dr]_{\labelstyle S-\text{dual}}
    &
        H_{n+p} (\Sigma^p X)
        \ar[r]^(0.45){\varphi}
    &
        Q^{n+p} (\tilde{C}(\Sigma^p_+ X)) \ar[r]^{J}
    &
        \Qh{n+p}{\tilde{C}(\Sigma^p_+ X)}
    &
        \Qh{n+p}{C^{n+p-\ast} (X)}
        \ar[l]_*!/u2pt/{\labelstyle \widehat{S^p\varphi}_0^{\%}}
    \\
    &
        H_{n+p} (\Th(\nu)^\ast)
        \ar[r]^(0.45){\varphi} \ar[u]_{\gamma_X}
    &
        Q^{n+p} (\tilde{C}(\Th(\nu)^\ast)) \ar[r]^{J} \ar[u]_{\gamma_X^{\%}}
    &
        \Qh{n+p}{\tilde{C}(\Th(\nu)^\ast)}
        \ar[r]^-*!/u3pt/{\labelstyle S-\text{dual}}
        \ar[u]^{\widehat\gamma_X^{\%}}
    &
        \Qh{n+p}{\tilde{C}^{N-\ast} (\Th(\nu))}
        \ar[u]_{\text{Thom}} \\
    }
\end{center}
}

The commutativity of the upper left part follows from the basic
properties of the symmetric construction. The commutativity of the
lower left part follows from the existence of the map $\Gamma_X$ and
naturality of the symmetric construction. The commutativity of the
right part follows from the commutativity of the diagram
(\ref{dgrm:Thom-S-dual-Poincare}).

As mentioned above, the construction can be made sufficiently
functorial, that means there is a preferred choice of $\chi$. We
obtain an $n$-dimensional NAC over $\Zpi$
\[
 (C(\tilde X),\theta (u(\nu))).
\]
\end{construction}


As in the previous section, we also need to discuss the relative
versions.

\begin{definition}\label{defn:GNP}
An $(n+1)$-dimensional \emph{geometric normal pair} (GNP) is a
triple $((X,Y),\nu,\rho)$ consisting of a pair of spaces~$(X,Y)$
with a $k$-dimensional spherical fibration~$\nu \co X \ra \BSG (k)$
and a map $\rho \colon (D^{n+1+k},S^{n+k}) \rightarrow
(\thom{\nu},\thom{\nu|_Y})$.

The \emph{fundamental class} of the normal pair~$((X,Y),\nu,\rho)$
is the $(n+1)$-dimensional homology class represented by the cycle
$[X,Y] \in C_n (X,Y)$ given by the formula $[X,Y] := u(\nu) \cap
h(\rho)$ where~$h$ is the Hurewicz homomorphism, and $u(\nu) \in
\tilde{C}^k (\thom{\nu})$ is some choice of the Thom class of $\nu$.

A \emph{geometric normal cobordism} between two $n$-dimensional GNCs
$(X,\nu,\rho)$ and $(X',\nu',\rho')$ is an $(n+1)$-dimensional
normal pair $((Z,X \sqcup X'),\nu'',\rho'')$ which restricts
accordingly over $X$ and $X'$.

The \emph{normal cobordism group} $\Omega^N_n (K)$ is defined as the abelian group of normal cobordism classes of $n$-dimensional GNCs with a reference map $r \co X \ra K$ and with the group operation given by the disjoint union operation.
\end{definition}

Notice that in the above setting, the triple $(Y,\nu|_Y,\rho|_{S^{n+k}})$ is an
$n$-dimensional GNC. The relative algebraic analogues come next.

\begin{defn}
A \emph{map of chain bundles} $(f,b) \co (C,\gamma) \ra
(C',\gamma')$ in $\AA$ is a map $f \co C \ra C'$ of chain complexes
in $\BB(\AA)$ together with a chain $b \in \Wh{TC}_1$ such that
\[
d(b) = \widehat f^{\%} (\gamma') - \gamma \in \Wh{TC}_0
\]
\end{defn}

\begin{defn}
An $(n+1)$-dimensional \emph{normal pair} $(f \co C \ra D,(\delta
\theta,\theta))$ in $\AA$ is an $(n+1)$-dimensional symmetric pair
$(f \co C \ra D,(\delta \varphi,\varphi))$ together with a map of
chain bundles $(f,b) \co (C,\gamma) \ra (D,\delta \gamma)$ and
chains $\chi \in \Wh C_{n+1}$ and $\delta \chi \in \Wh D_{n+2}$ such
that
\begin{align*}
J (\varphi) - \varphizp (S^n \gamma) & = d \chi \in \Wh{C}_n \\
J (\delta \varphi) - \widehat{\delta \varphi_0}^{\%} (S^{n+1} \delta
\gamma) + \widehat{f}^{\%} (\chi - \widehat{\varphi}^{\%}_0 (S^n b))
& = d ( \delta \chi) \in \Wh{D}_{n+1}
\end{align*}
where we abbreviate $(\delta \theta,\theta)$ for $((\delta
\varphi,\delta \gamma,\delta \chi),(\varphi,\gamma,\chi))$.
\end{defn}

Again notice that in the above setting $(C,\theta)$ is an $n$-dimensional NAC.

\begin{defn}
A \emph{normal cobordism} between normal complexes $(C,\theta)$ and
$(C',\theta')$ is a normal pair $((f \; f') \co C \oplus C' \ra
D,(\delta \theta,\theta \oplus -\theta'))$.
\end{defn}

The direct sum operation is defined analogously to the direct sum for the symmetric and quadratic complexes. Also there is a notion of a union of adjoining normal cobordisms and we obtain an equivalence relation. Again notice that a cobordism of normal complexes is in no sense a Poincar\'e pair.

There exists a relative normal construction. It associates to an
$(n+1)$-dimensional geometric normal pair an $(n+1)$-dimensional
algebraic normal pair in a functorial way. An $(n+1)$-dimensional
geometric normal cobordism induces an $(n+1)$-dimensional algebraic
normal cobordism in this way. These constructions are quite
complicated and therefore we again refer at this place to
\cite[section 7]{Weiss-II(1985)}.


Now we are ready to define the $NL$-groups, alias normal $L$-groups.

\begin{definition}\label{defn:NLgps} The \emph{normal $L$-groups} of an additive category with chain duality $\AA$ are
  \[
    NL^n (\AA)  :=  \{ \text{normal cobordism classes of } n \text{-dimensional NACs in } \AA \}.
  \]
\end{definition}

\begin{definition} \label{defn:normal-signature}
Let $(X,\nu,\rho)$ be an $n$-dimensional GNC. The cobordism class of
the $n$-dimensional NAC obtained from any choice of the Thom class
$u(\nu) \in \tilde{C}^k (\Th (\nu))$ in Construction
\ref{con:normal-construction} does not depend on the choice of
$u(\nu)$ and hence defines an element
\[
\nsign_{\Zpi} (X) = [(C(\tilde X),\theta (u(\nu)))] \in NL^n (\ZZ[\pi_1
(X)])
\]
called the \emph{normal signature} of $(X,\nu,\rho)$.

In fact the element $\nsign_{\ZZ[\pi_1 (X)]} (X)$ only depends on the normal cobordism class of $(X,\nu,\rho)$ and hence we obtain a homomorphism
\[
 \nsign_{\ZZ[\pi_1 (X)]} \co \Omega^N_n (K(\pi_1 (X),1) \ra NL^n (\ZZ[\pi_1 (X)]).
\]
\end{definition}

See also Remark \ref{rem:notation-NL-versus-L-hat} for a note on the notation.

Now we discuss the relation between the groups $L_n (\AA)$, $L^n
(\AA)$ and $NL^n (\AA)$. The details can be found in \cite[section
2]{Ranicki(1992)} and \cite{Weiss-II(1985),Weiss-II(1985)}. Here we
confine ourselves to the main ideas. We start with a lemma.

\begin{lemma} \textup{\cite[Proposition 2.6 (i)]{Ranicki(1992)}} \label{lem:normal-gives-quadratic-boundary}
Let~$(C,\varphi)$ be an $n$-dimensional SAC. Then~$(C,\varphi)$ can
be extended to a normal complex~$(C,\varphi,\gamma,\chi)$ if and
only if the boundary~$(\partial C, \partial \varphi)$ has a
quadratic refinement.
\end{lemma}

\begin{proof}
Consider the following long exact sequences
  \[
    \xymatrix@R=0,5cm{
      \ldots                    \ar[r]               &
      \Qq{n-1}{\partial C}      \ar[r]^-{1+T}        &
      \Qs{n-1}{\partial C}      \ar[r]^-{J}          &
      \Qh{n-1}{\partial C}      \ar[r]  \ar[d]_{S}^{\cong}   &
      \ldots
      \\
      \ldots                    \ar[r]               &
      \Qh{n}{C^{n-\ast}}         \ar[r]^-{\widehat{\varphi_0}^{\%}}    &
      \Qh{n}{C}                 \ar[r]^-{\widehat{e}^{\%}}               &
      \Qh{n}{\cone{(\varphi_0)}}     \ar[r]               &
      \ldots
    }
  \]
We have~$\partial \varphi = S^{-1} (e^{\%} (\varphi)) \in
\Qs{n-1}{\partial C}$. A diagram chase (using a slightly larger
diagram than the one above) gives the equation
\[
 \widehat{e}^{\%} (J(\varphi)) = S ( J (\del \varphi)) \in \Qh{n}{\cone{(\varphi_0)}}.
\]
It follows that $\del \varphi$ has a preimage in $\Qq{n-1}{\partial
C}$, that means a quadratic refinement, if and only if $J (\varphi)$
has a preimage in $\Qh{n}{C^{n-\ast}} \cong \Qh{0}{C^{-\ast}}$, that
means a chain bundle whose suspension maps to $J(\varphi)$ via
$(\widehat{\varphi_0})^{\%}$, in other words there is a normal
structure refining $\varphi$.
\end{proof}

The lemma can be improved so that one obtains a one-to-one
correspondence between the normal structures extending $(C,\varphi)$
and quadratic refinements of $(\del C,\del \varphi)$, the details
are to be found in \cite[sections 4,5]{Weiss-II(1985)}.

\begin{construction} \cite[Definition 2.9]{Ranicki(1992)}
The map
\[
 \del \co NL^n (\AA) \ra L_{n-1} (\AA) \quad \del
 (C,\varphi,\gamma,\chi) = (\del C,\del \psi)
\]
is defined so that $\del \psi$ is the quadratic refinement of $\del
\varphi$ described in Lemma
\ref{lem:normal-gives-quadratic-boundary}.
\end{construction}

\begin{lemma} \textup{\cite[Proposition 2.6 (ii)]{Ranicki(1992)}} \label{lem:symmetric-poincare-means-unique-normal}
There is a one-to-one correspondence between the homotopy classes of
$n$-dimensional SAPCs and the homotopy classes $n$-dimensional NACs
such that $\varphi_0$ is a chain homotopy equivalence.
\end{lemma}

\begin{proof}
Let $(C,\varphi)$ be an $n$-dimensional SAPC in $\AA$ so that $\varphi_0 \co \Sigma^n TC \ra C$ is a chain homotopy equivalence. One can associate a normal structure to $(C,\varphi)$ as follows. The chain bundle $\gamma \in \Wh{TC}_0$ is the image of $\varphi \in \Ws{C}_n$ under
\[
 \Ws{C}_n \xra{J} \Wh{C}_n \xra{(\widehat{\varphi_0}^{\%})^{-1}} \Wh{\Sigma^n TC}_n \xra{\;S^{-n}\;} \Wh{TC}_0.
\]
The chain $\chi \in \Wh{C}_{n+1}$ comes from the chain homotopy $(\widehat{\varphi_0}^{\%}) \circ (\widehat{\varphi_0}^{\%})^{-1} \simeq 1$. 
\end{proof}

\begin{construction} \cite[Proposition 2.6 (ii)]{Ranicki(1992)}
The map
\[
 J \co L^n (\AA) \ra NL^n (\AA) \quad J(C,\varphi) = (C,\varphi,\gamma,\chi)
\]
is constructed using the above Lemma
\ref{lem:symmetric-poincare-means-unique-normal}.
\end{construction}

The maps we just described in fact fit into a long exact sequence.

\begin{proposition} \textup{\cite[Definition 2.10, Proposition
2.8]{Ranicki(1992)}} \textup{\cite[Example 6.7]{Weiss-II(1985)}}
\label{propn:LESL} Let~$\AA$ be an additive category with chain
duality. Then there is a long exact sequence
  \[
    \xymatrix{
      \ldots         \ar[r]             &
      L_n (\AA)        \ar[r]^{1+T}       &
      L^n (\AA)        \ar[r]^J           &
      NL^n (\AA)       \ar[r]^{\partial}     &
      L_{n-1} (\AA)     \ar[r]             &
      \ldots
    }
  \]
\end{proposition}

\begin{proof}[Sketch of proof]
In \cite[chapter 2]{Ranicki(1981)} Ranicki defines the concept of a
triad of structured chain complexes and shows that one can define a
cobordism group of pairs of structured chain complexes, where the
structure on the boundary is some refinement of the structure
inherited from the pair. Such cobordism groups then fit into a
corresponding long exact sequence. The whole setup is analogous to
the definition of relative cobordism groups for a pair of spaces and
the associated long exact sequence.

In our special case we consider the map $J \co L^n (\AA) \ra NL^n
(\AA)$. So the $n$-th relative group is the cobordism group of
$n$-dimensional (normal, symmetric Poincar\'e) pairs, that means we
have a normal pair $f \co C \ra D$ such that the symmetric structure
on $C$ is Poincar\'e. This together with the following lemma
establish the proposition.
\end{proof}

\begin{lemma} \label{lem:normal-sym-pair-gives-quad-poincare-cplx}
\textup{\cite[Proposition 2.8 (ii)]{Ranicki(1992)}} Let~$\AA$ be an
additive category with chain duality. There is a one-to-one
correspondence between the cobordism classes of $n$-dimensional (normal, symmetric
Poincar\'e) pairs in $\AA$ and the cobordism classes of $(n-1)$-dimensional QAPCs in $\AA$.
\end{lemma}

\begin{proof}[Sketch of proof]
Let $(f \co C \ra D,(\delta \theta,\theta))$ be an $n$-dimensional
(normal, symmetric Poincar\'e) pair in $\AA$. In particular we have
an $n$-dimensional symmetric pair $(f \co C \ra D,(\delta
\varphi,\varphi))$, which we can use as data for an algebraic
surgery on the $(n-1)$-dimensional SAPC $(C,\varphi)$. The effect
$(C',\varphi')$ is again an $(n-1)$-dimensional SAPC. It turns out
to have a quadratic refinement, by a generalization of the proof of
Lemma \ref{lem:normal-gives-quadratic-boundary} (the lemma is a
special case when $f \co 0 \ra C$). The assignment $(f \co C \ra
D,(\delta \theta,\theta)) \mapsto (C',\del \psi')$ turns out to
induce a one-to-one correspondence on cobordism classes.
\end{proof}

\begin{rem} \label{rem:notation-NL-versus-L-hat}
Proposition \ref{propn:LESL} provides us with an isomorphism between the groups $NL^n (\AA)$ and the groups $\widehat{L}^n (\AA)$ defined in \cite{Ranicki(1981)} and used in \cite{Ranicki(1979)}.
\end{rem}


\subsection{The quadratic boundary of a GNC}
\label{subsec:quadratic-boundary-of-gnc}


In this subsection we study in more detail the passage from a GNC to
the boundary of its associated NAC. This means that from an
$n$-dimensional GNC we pass to an $(n-1)$-dimensional QAPC. The
construction was described in \cite[section 7.4]{Ranicki(1981)} even
before the invention of NAC in \cite{Weiss-I(1985),Weiss-II(1985)}.
It will be useful for geometric applications in later sections.

Before we start we need more basic technology. First we describe the
spectral quadratic construction:

\begin{construction}\label{con:SpectralQuadraticConstruction}
Let~$F \co X \lra \Sigma^{p} Y$ be a map between pointed spaces (a
map of this shape is called a {\it semi-stable map}) inducing the
chain map
\[
 f \co  \tilde{C}(X)_{p+\ast} \lra \tilde{C}(\Sigma^p Y)_{p+\ast} \simeq \tilde{C}(Y)
\]
The \emph{spectral quadratic construction} on~$F$ is a chain map
\[
\Psi \co \tilde{C}(X)_{p+\ast} \lra \Wq{\sC (f)}
\]
such that
\[
 (1+T) \circ \Psi \equiv e^\% \circ \varphi \circ f
\]
where $\varphi \co C(Y) \ra \Ws{C(Y)}$ is the symmetric construction
on $Y$ and $e \co C(Y) \ra \sC (f)$ is the inclusion map. The
existence of $\Psi$ can be read off the following commutative
diagram in which the lower horizontal sequence is exact by Remark \ref{rem:Qh-is-cohomology} and the right vertical sequence is exact by Proposition \ref{propn:LESQ}
\[
 \xymatrix{
  & \tilde{H}_{n+p} (X)  \ar[dl]_{\cong}  \ar[dr]_-{f}  \ar@{-->}[drr]^{\Psi}
  & 
  & 
  \\ 
    \tilde{H}_{n+p} (X)      \ar[d]_{\varphi_{X}}     \ar[r]^-{F}
  & \tilde{H}_{n+p} (\Sigma^{p} Y)  \ar[d]_{\varphi_{\Sigma^{p} Y}}
  & \tilde{H}_n (Y)     \ar[l]_-{\cong}  \ar[d]^{\varphi_Y}
  & \Qq{n}{\cone(f)}                  \ar[d]^{1+T}
  \\ 
    \Qs{n+p}{\tilde{C}(X)}      \ar[dr]^J  \ar[r]^-{F^\%}
  & \Qs{n+p}{\tilde{C}(\Sigma^{p} Y)}  \ar[dr]^J
  & \Qs{n}{\tilde{C}(Y)}  \ar[l]_-{S^{p}} \ar[d]^J   \ar[r]^-{e^\%}
  & \Qs{n}{\cone(f)}                     \ar[d]^J
  \\ 
  & \Qh{n}{\Sigma^{-p} \tilde{C}(X)}   \ar[r]^{\widehat{f}^\%}
  & \Qh{n}{\tilde{C}(Y)}         \ar[r]^-{\widehat{e}^\%}
  & \Qh{n}{\cone(f)}
 }
\]
The spectral quadratic construction $\Psi$ on $F$ has the property
that if $X = \Sigma^p X_0$ for some $X_0$ then it coincides with the
quadratic construction on $F$ as presented in Construction
\ref{constrn:quadconstrn} composed with $e_{\%}$.
\end{construction}

Recall that we have already encountered the semi-stable map
$\Gamma_Y$ coming from an $n$-GNC $(Y,\nu_Y,\rho_Y)$ in Construction
\ref{con:S-duality-and-Thom-is-Poincare}. The spectral quadratic
construction on $\Gamma_Y$ is identified below.

\begin{construction}
\label{con:normal-con-via-spectral-quadratic-con} See
\cite[Proposition 7.4.1]{Ranicki(1981)} and \cite[Theorem
7.1]{Weiss-II(1985)}. Let~$\Gamma_Y \co \Th(\nu_Y)^\ast \ra \Sigma^p
Y_+$ be the semi-stable map obtained in Construction
\ref{con:S-duality-and-Thom-is-Poincare} and let $\gamma_Y \co
C(\Th(\nu_Y)^\ast)_{\ast+p} \ra C(Y)$ denote the induced map. Recall
diagram (\ref{dgrm:Thom-S-dual-Poincare}) in Construction
\ref{con:S-duality-and-Thom-is-Poincare} which identifies
\[
\sC (\varphi_0) \simeq \sC (\gamma_Y)
\]
via the Thom isomorphism and $S$-duality. The spectral quadratic
construction on the map $\Gamma_Y$ produces a quadratic structure
\[
\Psi (u(\nu_Y)^\ast) \in W_{\%} \sC (\gamma_Y)_n
\]
where $u(\nu_Y)^\ast$ denotes the $S$-dual of the Thom class of
$\nu_Y$. We also have
\[
  (1+T) \circ \Psi ( u (\nu_{Y})^\ast) \equiv e^{\%} (\varphi ([Y])) \stackrel{\text{def}}{=} S(\del\varphi([Y]))
\]
From the cofibration sequence of chain complexes (with $C' = \del
C(Y)$):
\[
\Sigma \Wq{C'} \xra{\left(\begin{smallmatrix} 1+T  \\ S
\end{smallmatrix}\right)} \Sigma \Ws{C'} \oplus \Wq{\Sigma C'} \xra{S
- (1+T)} \Ws{\Sigma C'}
\]
we see that there exists a~$\psi (Y) \in (\Wq{\del C(Y)})_{n-1}$,
unique up to equivalence, such that~$(1+T)\psi (Y)\simeq
\del\varphi([Y])$. Hence we obtain an $(n-1)$-dimensional QAPC over
$\ZZ$ giving an element
\begin{equation*}
[(\del C(Y),\psi(Y))] \in L_{n-1} (\ZZ).
\end{equation*}

Recall from Construction \ref{con:normal-construction} that for any
geometric normal complex~$(Y,\nu_Y,\rho_Y)$ there is defined an
$n$-dimensional NAC $\nsign (Y)$ over $\ZZ$, which, as such, has a
quadratic boundary
\[
\del \nsign (Y) = [(C',\psi')] \in L_{n-1} (\ZZ)
\]
defined via Lemma \ref{lem:normal-gives-quadratic-boundary}.
Inspecting the definitions we see that $C' \simeq \del C(Y)$ and
further inspection of commutative diagrams defining the respective
quadratic structures shows that $\psi (Y)$ and $\psi'$ are
equivalent.
\end{construction}

\begin{example}
\label{expl:normal-symm-poincare-pair-gives-quadratic} See
\cite[Remark 2.16]{Ranicki(1992)}, \cite[Proposition
7.4.1]{Ranicki(1981)} and \cite[Theorem 7.1]{Weiss-II(1985)}. Recall
from the sketch proof of Lemma
\ref{lem:normal-sym-pair-gives-quad-poincare-cplx} that there is an
equivalence between cobordism classes of $n$-dimensional algebraic
(normal,~symmetric Poincar\'e) pairs and cobordism classes of
$(n-1)$-dimensional QAPCs, and that, in the special case that the
boundary in the pair we start with is $0$ the construction giving
the equivalence specializes to the construction of the quadratic
boundary of a normal complex.

In Construction \ref{con:normal-con-via-spectral-quadratic-con} it
is shown how the spectral quadratic construction can be used to
construct the quadratic boundary when we have a geometric normal
complex as input. In this example it is shown how the equivalence of
Lemma \ref{lem:normal-sym-pair-gives-quad-poincare-cplx} can be
realized using the spectral quadratic construction when we have a
degree one normal map of Poincar\'e complexes as input. In that case
the mapping cylinder of the map gives a normal pair, with Poincar\'e
boundary. Furthermore, it is shown that the quadratic complex
obtained in this way coincides with the surgery obstruction
associated to the degree one normal map. This is crucially used in
the proof of part (i) of the Main technical theorem (see proof of
Theorem \ref{thm:lifts-vs-orientations}).

Let $(f,b) \co M \ra X$ be a degree one normal map of $n$-GPC.
Denote by $\nu_M$, $\nu_X$ the respective SNFs. We form the
$(n+1)$-dimensional geometric (normal,~Poincar\'e) pair
\begin{equation*}
\big( (W,M \sqcup X), (\nu_W,\nu_{M \sqcup X}), (\rho_W,\rho_{M
\sqcup X}) \big)
\end{equation*}
with $W = \cyl (f)$. The symbol $\nu_W$ denotes the $k$-spherical
fibration over $W$ induced by $b$ and
\[
(\rho_W,\rho_{M \sqcup X}) \co (D^{n+1+k},S^{n+k}) \ra (\Th (\nu_W),
\Th (\nu_M \sqcup \nu_X))
\]
is the map induced by $\rho_M$ and $\rho_X$. Denote $j \co M \sqcup
X \hookrightarrow W$, $j_M \co M \hookrightarrow W$, and $j_X \co X
\hookrightarrow W$ the inclusions, by $\pr_X \co W \ra X$ the
projection which is also a homotopy inverse to $j_X$ and observe
that $f = \pr_X \circ j_M$.

Now we describe the passage
\begin{equation}
\label{eqn:norm-poincare-sign-of-deg-one-map-is-surgery-obstruction}
\textup{Lemma } \ref{lem:normal-sym-pair-gives-quad-poincare-cplx}
\co (\nsign (W),\ssign (M) - \ssign (X)) \mapsto [(C',\psi')].
\end{equation}

According to the proof of Lemma
\ref{lem:normal-sym-pair-gives-quad-poincare-cplx} the underlying
chain complex $C'$ is obtained by algebraic surgery on the
$(n+1)$-dimensional symmetric pair
\[
(j_\ast \co C(M) \oplus C(X) \ra C(W),(\delta \varphi,\varphi)).
\]
This is just the desuspension of the mapping cone of the 'want to
be' Poincar\'e duality map
\begin{equation*}
C' = S^{-1} \sC \big( C^{n+1-\ast} (W) \xra{\smallpairmapforw}
C(W,M\sqcup X) \big)
\end{equation*}

If we want to use the spectral quadratic construction we need a
semi-stable map inducing the map in the above display. Consider the
map
\[
S^N \xra{\rho_W/\rho_{M\sqcup X}} \Th (\nu_W) / \Th (\nu_{M\sqcup
X}) \xra{\Delta} \Sigma^p (W/(M \sqcup X)) \wedge \Th (\nu_W)
\]
which has an $S$-dual
\[
\Gamma_W \co \Th (\nu_W)^\ast \ra \Sigma^p (W/(M \sqcup X))
\]
which in turn induces a map of chain complexes
\[
\gamma_W \co C_{\ast+p}(\Th (\nu_W)^\ast) \ra C_\ast (W/(M \sqcup
X))
\]
The map $\gamma_W$ coincides with the map $\smallpairmapforw$ under
Thom isomorphism and $S$-duality (by a relative version of Diagram
\ref{dgrm:Thom-S-dual-Poincare}, see also section
\ref{sec:proof-part-2}).

The spectral quadratic construction on $\Gamma_W$
\[
\Psi \co C_{n+1+p} (\Th (\nu_W)^\ast) \ra \Wq{\sC (\gamma_W)}_{n+1}
\]
produces from the dual of the Thom class $u(\nu_W)^\ast \in
C_{n+1+p} (\Th (\nu_W)^\ast)$ an $(n+1)$-dimensional quadratic
structure on $\sC (\gamma_W)$ which has a desuspension unique up to
equivalence and that is our desired $\psi'$ such that
\[
\Psi (u(\nu_W)^\ast) = S (\psi').
\]

The construction just described comes from \cite[Proposition
7.4.1]{Ranicki(1981)}. By \cite[Proof of Theorem
7.1]{Weiss-II(1985)} we obtain that
(\ref{eqn:norm-poincare-sign-of-deg-one-map-is-surgery-obstruction})
holds.

Now recall from Definition \ref{defn:quad-sign} that we have an
another way of assigning $n$-dimensional quadratic Poincar\'e
complex to $(f,b)$, namely the surgery obstruction $\qsign (f,b) \in
L_n (\ZZ)$.

We claim that
\begin{equation*}
[(C',\psi')] = \qsign (f,b) \in L_n (\ZZ)
\end{equation*}

The following commutative diagram identifies $C' \simeq \sC (f^!)$:
\begin{equation} \label{dgrm:want-to-be-duality-in-pair-vs-umkehr}
\begin{split}
 \xymatrix{
 C^{n+1-\ast} (W) \ar[r]^{\smallpairmapforw} \ar[d]^{j_M^\ast} &
 C(W,M\sqcup X) \ar[d]_{\simeq} \\
 C^{n+1-\ast} (M) \ar[r]^{\varphi_0|_M}_{\simeq} & C(\Sigma M) \\
 C^{n+1-\ast} (X) \ar[u]_{S(f^\ast)}
 \ar@/^5pc/[uu]^{\pr_X^\ast}_{\simeq}
 \ar[r]^{\varphi_0|_X}_{\simeq} & C (\Sigma X) \ar[u]_{S(f^!)}
 }
\end{split}
\end{equation}

To identify the quadratic structures recall first that the spectral
quadratic construction $\Psi$ on a semi-stable map $F$ is the same
as the quadratic construction $\psi$ composed with $e_{\%}$ if the
semi-stable map $F \co X \ra \Sigma^p Y$ is in fact a stable map $F
= \Sigma^p X_0 = X \ra \Sigma^p Y$. Furthermore the homotopy
equivalence $j_X$ and the $S$-duality are used to show that Diagram
\ref{dgrm:want-to-be-duality-in-pair-vs-umkehr} is induced by the
commutative diagram of maps of spaces as follows:
\begin{equation}
\begin{split}
\xymatrix{ \Th(\nu_W)^\ast \ar[r]^(0.4){\Gamma_W}
\ar[d]^{T(j_M)^\ast} &
\Sigma^p (W/(M \sqcup X)) \ar[d]^{\simeq} \\
\Th(\nu_M)^\ast \ar[r]^{\gamma_M}_{\simeq} & \Sigma^{p+1} M_+ \\
\Th(\nu_X)^\ast \ar[u]_{T(b)^\ast}
\ar@/^5pc/[uu]^{T(\pr_X)^\ast}_{\simeq}
 \ar[r]^{\gamma_X}_{\simeq} & \Sigma^{p+1} X_+ \ar[u]_{F}
}
\end{split}
\end{equation}
which identifies $F$ and $\Gamma_W$.

The Thom class $u(\nu_W)$ restricts to $u(\nu_X)$ and hence the
duals $u(\nu_W)^\ast$ and $u(\nu_X)^\ast = \Sigma [X]$ are also
identified. The uniqueness of desuspensions gives the identification
of the equivalence classes of the quadratic structures
\[
e_{\%} \psi ([X]) \sim \psi'.
\]
\end{example}

\section{Algebraic bordism categories and exact sequences}
\label{sec:alg-bord-cat}

In previous sections we recalled the notions of certain structured
chain complexes over an additive category with chain duality $\AA$
and corresponding $L$-groups. In this section we review a
generalization where the category we work with is an algebraic
bordism category. This eventually allows us to vary $\AA$ and we also obtain
certain localization sequences.

\begin{defn}
An algebraic bordism category $\Lambda =(\AA,\BB,\CC,(T,e))$
consists of an additive category with chain duality $(\AA,(T,e))$, a
full subcategory $\BB \subseteq \BB(\AA)$ and another full
subcategory $\CC \subseteq \BB$ closed under taking cones.
\end{defn}

\begin{defn}
Let $\Lambda=(\AA,\BB,\CC,(T,e))$ be an algebraic bordism category.

An \emph{$n$-dimensional symmetric algebraic complex in $\Lambda$}
is a pair $(C,\varphi)$ where $C\in\BB$ and $\varphi\in
(W^{\%}C)_{n}$ is an $n$-cycle such that $\partial C =
\dsc(\varphi_{0}:\susp^{n}TC\ra C) \in\CC$.

An \emph{$(n+1)$-dimensional symmetric algebraic pair in $\Lambda$}
is a pair\\ $(f:C\ra D, (\delta\varphi,\varphi))$ in $\Lambda$ where
$f:C\ra D$ is a chain map with $C,D\in \BB$, the pair
$(\delta\varphi,\varphi)\in \sC(f^{\%})$ is an $(n+1)$-cycle and
$\sC(\delta\varphi_{0},\varphi_{0}f^{*})\in\CC$.

A \emph{cobordism} between two $n$-dimensional symmetric algebraic
complexes $(C,\varphi)$ and $(C',\varphi')$ in $\Lambda$ is an
$(n+1)$-dimensional symmetric pair $(C\oplus C' \ra
D,(\delta\psi,\varphi\oplus-\varphi')$ in $\Lambda$.
\end{defn}

So informally these are complexes, pairs and cobordisms of chain
complexes which are in $\BB$ and which are Poincar\'e modulo $\CC$.
There are analogous definitions in quadratic and normal case. The $L$-groups are generalized to this setting as follows.

\begin{defn} \label{defn:L-groups-over-Lambda}
The symmetric, quadratic, and normal \emph{L-groups}
\[
  L^{n}(\Lambda), \quad L_{n}(\Lambda) \quad \textup{and} \quad NL^{n}(\Lambda)
\]
are defined as the cobordism groups of $n$-dimensional symmetric,
quadratic, and normal algebraic complexes in $\Lambda$ respectively.
\end{defn}

\begin{expl} \label{expl:alg-bord-cat-ring}
Let $R$ be a ring with involution. By
\[
\Lambda(R)=(\AA(R),\BB(R),\CC(R),(T,e))
\]
is denoted the algebraic bordism category with
\begin{itemize}
\item[$\AA(R)$] the category of $R$-modules from Example \ref{expl:R-duality},
\item[$\BB(R)$] the bounded chain complexes in $\AA(R)$,
\item[$\CC(R)$] the contractible chain complexes of $\BB(R)$.
\end{itemize} 
We also consider the algebraic bordism category
\[
\widehat \Lambda(R)=(\AA(R),\BB(R),\BB(R),(T,e)).
\]
The $L$-groups of section \ref{sec:algebraic-cplxs} and the $L$-groups of Definition \ref{defn:L-groups-over-Lambda} are related by
\[
 L^n (R) \cong L^n (\Lambda (R)) \quad \textup{and} \quad L_n (R) \cong L_n (\Lambda (R)).
\]
For the $NL$-groups of section \ref{sec:normal-cplxs} we have:
\[
 NL^n (R) \cong NL^n (\widehat \Lambda (R)) \quad \textup{and} \quad L^n (R) \cong NL^n (\Lambda (R)).
\]
The second isomorphism is due to Lemma \ref{lem:symmetric-poincare-means-unique-normal}.
\end{expl}

The notion of a functor of algebraic bordism categories
    \[
    F:\Lambda=(\AA,\BB,\CC)\ra \Lambda'=(\AA',\BB',\CC')
    \]
is defined in \cite[Definition 3.7]{Ranicki(1992)}. Any such functor
induces a map of $L$-groups.

\begin{prop} \textup{\cite[Prop. 3.8]{Ranicki(1992)}} \label{prop:relativ-L-groups}
For a functor $F:\Lambda\ra \Lambda'$ of algebraic bordism
categories there are relative L-groups $L_{n}(F)$, $L^{n}(F)$ and
$NL^{n}(F)$ which fit into the long exact sequences
\[\les[4]{L_{\n}(\Lambda)}{L_{\n}(\Lambda')}{L_{\n}(F)},\]
\[\les[4]{L^{\n}(\Lambda)}{L^{\n}(\Lambda')}{L^{\n}(F)},\]
\[\les[4]{NL^{\n}(\Lambda)}{NL^{\n}(\Lambda')}{NL^{\n}(F)}.\]
\end{prop}
These exact sequences are produced by the technology of \cite[chapter 2]{Ranicki(1981)} already mentioned in the previous section. An element in $L_{n}(F)$ is an $(n-1)$-dimensional quadratic complex $(C,\psi)$ in $\Lambda$ together with an $n$-dimensional quadratic
pair $(F(C)\ra D, (\delta\psi,F(\psi)))$ in $\Lambda'$. There is a notion of a cobordism of such pairs and the group $L_n (F)$ is
defined as such a cobordism group. Analogously in the symmetric and normal case.

The following proposition improves the above statement in the sense that the relative terms are given as cobordism groups of complexes rather than pairs.

\begin{prop} \textup{\cite[Prop. 3.9]{Ranicki(1992)}} \label{prop:inclusion-les}
Let $\AA$ be an additive category with chain duality and let
$\DD\subset\CC\subset\BB\subset \BB (\AA)$ be subcategories closed
under taking cones. The relative symmetric L-groups for the
inclusion $F:(\AA,\BB,\DD)\ra(\AA,\BB,\CC)$ are given by
    \begin{itemize}
        \item[(i)] $L^{n}(F) \cong L^{n-1}(\AA,\CC,\DD)$
    \end{itemize}
    and in the quadratic and normal case by
    \begin{itemize}
        \item[(ii)] $L_{n}(F) \cong L_{n-1}(\AA,\CC,\DD) \cong NL^{n}(F).$
    \end{itemize}
\end{prop}

Part (ii) of the proposition allows us to produce interesting
relations between the long exact sequences for various inclusions
combining the quadratic $L_n$-groups and the normal $NL^n$-groups.
In the following commutative braid we have $4$ such sequences.
Sequence (1) is given by the inclusion $(\AA,\BB,\DD) \ra
(\AA,\BB,\CC)$ in the quadratic theory, sequence (2) by the
inclusion $(\AA,\BB,\DD) \ra (\AA,\BB,\CC)$, sequence (3) by the
inclusion $(\AA,\BB,\CC) \ra (\AA,\BB,\BB)$, and sequence (4) by the
inclusion $(\AA,\BB,\DD) \ra (\AA,\BB,\BB)$, all last three in the
normal theory:

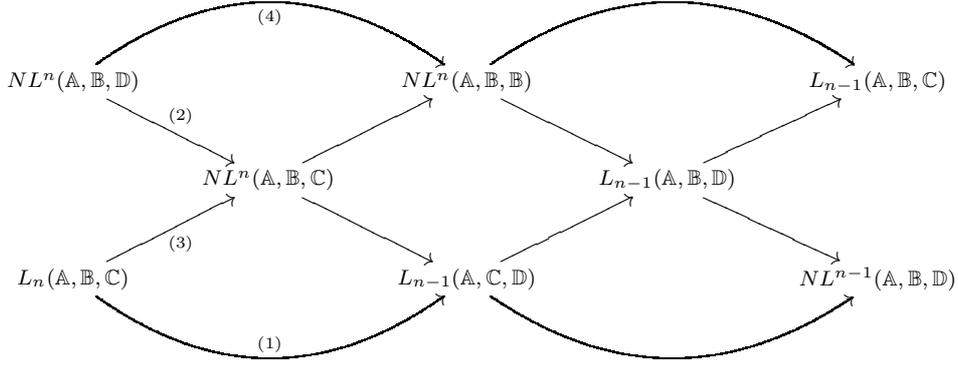
\newcommand{\braideightlabel}[8]{\xymatrix@C=0.65cm{
    & & & & \\
    {#1}\ar@/^2.5pc/_{(4)}[rr]\ar^{(2)}[rd] & & {#2}\ar@/^2.5pc/[rr]\ar[rd] & & #3 \\
    & {#4}\ar[ru]\ar[rd] & & {#5}\ar[rd]\ar[ru] & \\
    {#6}\ar@/_2.5pc/^{(1)}[rr]\ar_{(3)}[ru] & & {#7}\ar@/_2.5pc/[rr]\ar[ru] & & #8 \\
 & & & &
}}
{\footnotesize
\[
\braideightlabel{NL^{n}(\AA,\BB,\DD)}{NL^{n}(\AA,\BB,\BB)}{L_{n-1}(\AA,\BB,\CC)}{NL^{n}(\AA,\BB,\CC)}{L_{n-1}(\AA,\BB,\DD)}{L_{n}(\AA,\BB,\CC)}{L_{n-1}(\AA,\CC,\DD)}{NL^{n-1}(\AA,\BB,\DD)}
\]
}

\begin{proof}[Comments on the proof Proposition \ref{prop:inclusion-les}]
Recall that an element in $L_{n}(F)$ is an $(n-1)$-dimensional
quadratic complex $(C,\psi)$ in $(\AA,\BB,\DD)$ together with an
$n$-dimensional quadratic pair $(C \ra D, (\delta\psi,\psi))$
in $(\AA,\BB,\CC)$. The isomorphism $L_{n}(F)\cong
L_{n-1}(\AA,\CC,\DD)$ is given by
\[
\left( (C,\psi),C \ra D, (\delta\psi,\psi)\right)\mapsto (C',\psi')
\]
where $(C',\psi')$ is the effect of algebraic surgery on $(C,\psi)$
using as data the pair $(C\ra D, (\delta\psi,\psi))$. We have $C' \in \CC$
since $C \ra D$ is Poincar\'e modulo $\CC$. Furthermore, the
observation that $(C',\psi')$ is Poincar\'e modulo $\DD$ follows
from the assumption that $(C,\psi)$ is Poincar\'e modulo $\DD$ and
from Proposition \ref{prop:homotopy-type-of-boundary} which says that the homotopy type of the boundary is
preserved by algebraic surgery.

The inverse map is given by
\[
 (C,\psi)\mapsto \left( (C,\psi), C\ra 0, (0,\psi)\right).
\]

Similarly for $NL^{n}(F)\cong L_{n-1}(\CC,\DD)$. Consider $\left((C,\theta), C\ra D, (\delta \theta,\theta)\right)\in NL^{n}(F)$ and perform
algebraic surgery on $(C,\theta)$ with data $(C\ra D, (\delta \theta,\theta))$. We obtain an $(n-1)$-dimensional symmetric complex in $\CC$ which
is Poincar\'e modulo $\DD$. Using \cite[2.8(ii)]{Ranicki(1992)} we see that the symmetric structure has a quadratic refinement.
\end{proof}

\begin{expl}
Let $R$ be a ring with involution and consider the inclusion of the algebraic bordism categories $\Lambda (R) \ra \widehat \Lambda (R)$ from Example \ref{expl:alg-bord-cat-ring}. Then the long exact sequence of the  associated $NL$-groups (sequence (3) in the diagram above) becomes the long exact sequence of Proposition \ref{propn:LESL}, thanks to Lemma \ref{lem:symmetric-poincare-means-unique-normal}.
\end{expl}


\section{Categories over complexes}  \label{sec:cat-over-cplxs}


In this section we recall the setup for studying local Poincar\'e duality over a locally finite simplicial complex $K$. For a simplex $\sigma \in K$ we will use the notion of a dual cell $D(\sigma,K)$ which is a certain subcomplex of the barycentric subdivision $K'$, see \cite[Remark 4.10]{Ranicki(1992)} for the definition if needed.\footnote{Note that in general the dual cell $D(\sigma,K)$ is not a ``cell'' in the sense that it is not homeomorphic to $D^l$ for any $l$. Nevertheless the terminology is used in \cite{Ranicki(1992)} and we keep it.}

Observe first that there are two types of such a local duality for a
triangulated $n$-manifold $K$.
\begin{enumerate}
\item Each simplex $\sigma$ of $K$ is a $|\sigma|$-dimensional manifold with boundary and so there is a duality between $C_{*}(\sigma,\partial\sigma)$ and $C^{|\sigma|-*}(\sigma)$
\end{enumerate}
\begin{minipage}{\linewidth-4.5cm}
	\begin{enumerate}
		\item[(2)]
Each dual cell $D(\sigma,K)$ is an $(n-|\sigma|)$-dimensional manifold with boundary and so there is a duality between the chain complexes $C_{\ast}(D(\sigma, K),\partial D(\sigma, K))$ and $C^{n-|\sigma|-\ast}(D(\sigma, K))$
\end{enumerate} 
\end{minipage}
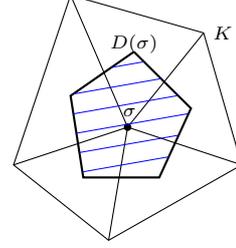
    \begin{minipage}{3.5cm}
    \begin{flushright}
    \dualpicture
    \end{flushright}
\end{minipage}\\[0.2cm]

This observation leads to two notions of additive categories with
chain duality over $K$.

\begin{defn} Let $\AA$ be an additive category with chain duality and $K$ as above.
The additive categories of $K$-based objects $\AA^{*}(K)$ and
$\AA_{*}(K)$ are defined by
\begin{enumerate}
\item[] $\text{Obj}(\AA^{*}(K))=\text{Obj}(\AA_{*}(K))=\{\sum_{\sigma\in K}M_{\sigma}\;|\; M_{\sigma}\in\AA\}$,\\[0.3em]
\item $\text{Mor}(\AA^{*}(K)) =$\\ \indent $\{ \sum\limits_{\sigma\geq\tau}f_{\tau,\sigma}:\sum\limits_{\sigma\in K}M_{\sigma}\ra\sum\limits_{\tau\in K}N_{\tau}\;|\;(f_{\tau,\sigma}:M_{\sigma}\ra N_{\tau})\in\text{Mor}(\AA)\}$\\[0.2em]
\item $\text{Mor}(\AA_{*}(K)) =$ \\ \indent $\{ \sum\limits_{\sigma\leq\tau}f_{\tau,\sigma}:\sum\limits_{\sigma\in K}M_{\sigma}\ra\sum\limits_{\tau\in K}N_{\tau}\;|\; (f_{\tau,\sigma}:M_{\sigma}\ra N_{\tau})\in\text{Mor}(\AA)\}$
\end{enumerate}
\end{defn}

A chain complex $(C, d)$ over $\AA^{*}(K)$, respectively $\AA_\ast
(K)$, consists of chain complexes $(C(\sigma),d(\sigma))$ for each
$\sigma \in K$ and additional boundary maps
$d(\sigma,\tau) \co C(\sigma)_\ast \ra C(\tau)_{\ast-1}$ for each
$\tau \leq \sigma$, respectively $\sigma \leq \tau$.

\begin{example} \label{expl:chain-cplxs-over-K}
The simplicial chain complex $C = \Delta (K)$ is a chain complex in
$\BB (\AA^\ast (K))$, by defining $C(\sigma)\coloneqq \Delta (\sigma,\del
\sigma) = S^{|\sigma|} \ZZ$.

The simplicial chain complex $C = \Delta (K')$ of the barycentric
subdivision $K'$ is a chain complex in $\BB (\AA_\ast (K))$ by $C
(\sigma) = \Delta (D(\sigma),\del D(\sigma))$.
\end{example}

The picture depicts the simple case of the simplicial chain complex
$\Delta_{*}(\Delta^{1})$ as a chain complex in
$\AA(\ZZ)^{*}(\Delta^{1})$:

\begin{center}
	\[
\xymatrix@R=0.3cm{
  *=0{\bullet}\save[]+<0cm,0.3cm>*{_{\sigma_{0}}}\restore\ar@{-}^-{\tau}[rr]
    & &
    *=0{\bullet}\save[]+<0cm,0.3cm>*{_{\sigma_{1}}}\restore
    \\
    C(\sigma_{0})=\Delta_{*}(\sigma_{0},\partial\sigma_{0})
    &
    C(\tau)=\Delta_{*}(\tau,\partial\tau)
    &
    C(\sigma_{1})=\Delta_{*}(\sigma_{1},\partial\sigma_{1})
    \\
		\save[]+<-1.9cm,0cm>*{C_{2}:}\restore
    0\ar[dd]
    &
    0\ar[ddl]\ar[dd]\ar[ddr]
    &
    0\ar[dd]
    \\
		&&&
		\\
		\save[]+<-1.9cm,0cm>*{C_{1}:}\restore
    0\ar[dd]
    &
    \ZZ\ar[ldd]_{\partial_{0}}\ar[rdd]^-{\partial_{1}}\ar[dd]
    &
    0\ar[dd]
    \\
		&&&
		\\
		\save[]+<-1.9cm,0cm>*{C_{0}:}\restore
    \ZZ
    &
    0
    &
    \ZZ
}
\]
\end{center}

Now we recall the extension of the chain duality from $\AA$ to the
two new categories.

\begin{defn}
 \[
T^{*}:\AA^{*}(K)\ra\BB(\AA^{*}(K)),\; T^{*}(\sum_{\sigma\in K}
M_{\sigma}))_{r}(\tau) =
(T(\bigoplus_{\tau\geq\tilde{\tau}}M_{\tilde{\tau}}))_{r-|\tau|}.
\]
\[
T_{*}:\AA_{*}(K)\ra\BB(\AA_{*}(K)),\; T_{*}(\sum_{\sigma\in K}
M_{\sigma}))_{r}(\tau) =
(T(\bigoplus_{\tau\leq\tilde{\tau}}M_{\tilde{\tau}}))_{r+|\tau|}.
\]
\end{defn}

\begin{example} \label{expl:duality-in-chain-cplxs-over-K}
The dual $T^\ast (C)$ of the simplicial chain complex $C = \Delta
(K)$ is a chain complex in $\BB (\AA^\ast (K))$ given by $(T^\ast
C)(\sigma) = \Delta^{|\sigma|-\ast} (\sigma)$.

The dual $T_\ast C$ of the simplicial chain complex $C = \Delta (K')$
of the barycentric subdivision $K'$ is a chain complex in $\BB
(\AA_\ast (K))$ given by $(T_\ast C)(\sigma) =
\Delta^{-|\sigma|-\ast} (D(\sigma))$.
\end{example}

In the example when $K$ is a triangulated manifold recall that the
chain duality functor $T^{*}$ on $\AA^{*}(K)$ is supposed to encode
the local Poincar\'e duality of all simplices of $K$. But the
dimensions of these local Poincar\'e dualities vary with the
dimension of the simplices and we have to deal with the boundaries.
So the dimension shift in the above formula comes from the varying
dimensions and the direct sum comes from ``dealing with the
boundary''. In the example $C = \Delta_* (\Delta^1)$ we obtain the
following picture

{\footnotesize
\[
\xymatrix@C=0.5cm@R=0.3cm{
            {\Delta_{*}(\sigma_{0},\partial\sigma_{0})}
            &
            {\Delta_{*}(\tau,\partial\tau)}
            &
            {\Delta_{*}(\sigma_{1},\partial\sigma_{1})}
            &
            {\Delta^{*}(\sigma_{0})}
            &
            {\Delta^{*}(\tau)}
            &
            {\Delta^{*}(\sigma_{1})}
\\
        \save[]+<-0.7cm,0cm>*{C_{1}:}\restore
            0\ar[dd]
            &
            \ZZ\ar[ldd]_{\partial_{0}}\ar[rdd]^-{\partial_{1}}\ar[dd]
            &
            0\ar[dd]\save[]+<0.6cm,0cm>*{}\ar@{<--}^{\varphi_{0}}[r]+<-0.6cm,0cm>\restore
            & &
            (\ZZ\oplus\ZZ)^{*}\ar[dd]^(0.6)*!/l2pt/{{\partial_{0}^{*}}\choose{\partial_{1}^{*}}}
            \ar[ddl]_{i^{*}_{0}}\ar[ddr]^{i^{*}_{1}}
            &
            \save[]+<1.3cm,0cm>*{:(T^{*}C)_{1}}\restore
        \\
					&&&&&
				\\
        \save[]+<-0.7cm,0cm>*{C_{0}:}\restore
            \ZZ
            &
            0
            &
            \ZZ
            \save[]+<0.6cm,0cm>*{}\ar@{<--}^{\varphi_{0}}[r]+<-0.6cm,0cm>\restore
            &
            \ZZ^{*}\ar[dd]
            &
            \ZZ^{*}\ar[ddr]\ar[ddl]
            &
            \ZZ^{*}\ar[dd]
            \save[]+<1.3cm,0cm>*{:(T^{*}C)_{0}}\restore
                \\
								&&&&&
							\\
        &&&0&&0\save[]+<1.4cm,0cm>*{:(T^{*}C)_{-1}}\restore
        \\
    }
\]
}
In $\AA_{*}(K)$ the role of simplices is replaced by the dual cells
and so the formulas are changed accordingly.

The additive categories with chain duality $\AA^\ast (K)$ and $\AA_\ast (K)$ can be made into algebraic bordism categories in various ways yielding chain complexes with various types of Poincar\'e duality. Now we introduce the local duality, in the next section we will have the global duality.

\begin{prop} \label{defn:Lambda-star-categories}
Let $\Lambda=(\AA,\BB,\CC)$ be an algebraic bordism category and $K$
a locally finite simplicial complex. Then the triples
\[
\Lambda^{*}(K) = (\AA^\ast (K),\BB^\ast (K),\CC^\ast (K)) \qquad
\Lambda_\ast (K) = (\AA_\ast (K),\BB_\ast (K),\CC_\ast (K))
\]
where $\BB^{*} (K)$, $\BB_\ast (K)$, $(\CC^\ast (K)$, $\CC_\ast
(K))$ are the full subcategories of $\BB (\AA^\ast (K))$,
respectively $\BB (\AA_\ast (K))$, consisting of the chain complexes
$C$ such that $C(\sigma) \in \BB$ $(C(\sigma) \in \CC)$ for all
$\sigma \in K$ are algebraic bordism categories.
\end{prop}

See \cite[Proposition 5.1]{Ranicki(1992)} for the proof. We remark that other useful algebraic bordism categories associated to $\Lambda$ and $K$ will be defined in Definitions \ref{defn:Lambda-K-category} and \ref{defn:Lambda-hat-category}.

\begin{prop}
Let $\Lambda=(\AA,\BB,\CC)$ be an algebraic bordism category and let
$f \co J \ra K$ be a simplicial map. Then $f$ induces
contravariantly (covariantly) a covariant functor of algebraic
bordism categories
\[
f^\ast \co \Lambda^\ast (K) \ra \Lambda^\ast (J) \qquad (f_\ast \co
\Lambda_\ast (J) \ra \Lambda_\ast (K)).
\]
\end{prop}

See \cite[Proposition 5.6]{Ranicki(1992)}. A consequence is that we
obtain induced maps on the $L$-groups as well, which we do not
write down explicitly at this stage.




Now we present constructions over the category $\AA(\ZZ)^\ast (K)$ analogous to the symmetric and quadratic construction in section \ref{sec:algebraic-cplxs}. Examples of chain complexes over $\AA(\ZZ)^{\ast} (K)$ were already presented in Examples \ref{expl:chain-cplxs-over-K} and \ref{expl:duality-in-chain-cplxs-over-K}. The underlying chain complexes below are generalizations of those. We will write $\ZZ^\ast(K)$ as short for $\AA(\ZZ)^\ast(K)$ and $\ZZ_\ast(K)$ as short for $\AA(\ZZ)_\ast(K)$.

\begin{construction} \label{con:sym-construction-over-cplxs-upper-star-cplx}
Consider a topological $k$-ad $(X,(\del_\sigma X)_{\sigma \in \Delta^k})$ and the subcomplex of the singular chain complex $C (X)$ consisting of simplices which respect the $k$-ad structure in a sense that each singular simplex is contained in $\del_\sigma X$ for some $\sigma \in \Delta^k$. By a Mayer-Vietoris type argument this chain complex is chain homotopy equivalent to $C(X)$ and by abuse of notation we still denote it $C(X)$. It becomes a chain complex over $\ZZ^\ast (\Delta^k)$ by $C(X) (\sigma) = C(\del_\sigma X,\del (\del_\sigma X))$. Its dual is a chain complex $T^\ast C(X) $ given by $(T^\ast C(X) )(\sigma) = C^{|\sigma|-\ast} (\del_\sigma X)$ for $\sigma \in \Delta^k$. A generalization of the relative symmetric construction \ref{con:rel-sym} gives a chain map
\[
 \varphi_{\Delta^k} \co \Sigma^{-k} C(X,\del X) \ra \Ws{C(X)} \; \textup{over} \; \ZZ^\ast (\Delta^k)
\]
called the \emph{symmetric construction over} $\Delta^k$ which evaluated on a cycle $[X] \in C_{n+k} (X, \del X)$ gives an $n$-dimensional symmetric algebraic complex $(C (X),\varphi_{\Delta^k} [X])$ in $\ZZ^\ast (\Delta^k)$ whose component
\[
\varphi_{\Delta^k} ([X]) (\sigma)_0 \co C^{n+|\sigma|-\ast} (\del_\sigma X) \ra C(\del_\sigma X,\del (\del_\sigma X))
\]
is the cap product with the cycle $[\del_\sigma X] \in C_{n+|\sigma|} (\del_\sigma X,\del (\del_\sigma X))$. Here $\del_\sigma \co C (\Delta^k) \ra C(\sigma)$ is the map defined as in \cite[Definition 8.2]{Ranicki(1992)}.
\end{construction}

\begin{construction} \label{con:sym-construction-over-cplxs-upper-star-mfd}
Consider now the special case when we have an $(n+k)$-dimensional manifold $k$-ad $(M,(\del_\sigma M)_{\sigma \in \Delta^k})$. Let $\Lambda (\ZZ)$ be the algebraic bordism category from Example \ref{expl:alg-bord-cat-ring}. Construction \ref{con:sym-construction-over-cplxs-upper-star-cplx} applied to the fundamental class $[M] \in C_{n+k} (M,\del M)$ produces an $n$-dimensional symmetric algebraic complex $(C (M),\varphi_{\Delta^k} ([M]))$ in the category $\Lambda (\ZZ)^\ast (\Delta^k)$ since the maps
\[
\varphi_{\Delta^k} ([M]) (\sigma)_0 \co C^{(n-k)+|\sigma|-\ast} (\del_\sigma M) \ra C(\del_\sigma M,\del (\del_\sigma M))
\]
are the cap products with the fundamental classes $[\del_\sigma M] \in C_{n-k+|\sigma|} (\del_\sigma M,\del (\del_\sigma M))$ and hence chain homotopy equivalences and hence their mapping cones are contractible.
\end{construction}

\begin{construction} \label{con:quad-construction-over-cplxs-upper-star}
Analogously, when we have a degree one normal map of manifold $k$-ads
\[
 ((f,b),(f_\sigma,b_\sigma)) \co (M,\del_\sigma M) \ra (X,\del_\sigma X)
\]
with $\sigma \in \Delta^k$, the stable Umkehr map $F \co \Sigma^p X_+ \ra \Sigma^p M_+$ for some $p$ induces by a generalization of the relative quadratic construction \ref{con:rel-quad-htpy} a chain map
\[
 \psi_{\Delta^k} \co \co \Sigma^{-k} C(X,\del X) \ra \Wq{C(M)} \; \textup{over} \; \ZZ^\ast (\Delta^k)
\]
called the \emph{quadratic construction over} $\Delta^k$. Evaluated on the fundamental class $[X] \in C_{n+k} (X, \del X)$ it produces an $n$-dimensional quadratic algebraic complex in the category $\Lambda(\ZZ)^\ast (\Delta^k)$. The mapping cone $\sC (f^!)$ becomes a complex over $\ZZ^\ast (\Delta^k)$ by $\sC (f^{!}) (\sigma) = \sC (f_\sigma^{!},\del f_\sigma^{!})$. The chain map $e \co C(M) \ra \sC (f^!)$ in $\ZZ^\ast (\Delta^k)$ produces an $n$-dimensional quadratic complex in $\Lambda (\ZZ)^\ast (\Delta^k)$ 
\[
 \big( C(f^!), e_\% \psi_{\Delta^k} [X] \big) 
\]
\end{construction}

Now we move to the constructions in the category $\ZZ_\ast (K)$.

\begin{construction} \label{con:sym-construction-over-cplxs-lower-star-cplx}
Let $r \co X \ra K$ be a map of simplicial complexes. Denoting for $\sigma \in K$
\[
X[\sigma] = r^{-1} (D(\sigma,K)) \subset X' \; \textup{we obtain} \;
X = \bigcup_{\sigma \in K} X[\sigma].
\]
This decomposition is called a $K$-dissection of $X$. Consider the subcomplex of the singular chain complex of $C(X)$ consisting of the singular chains which respect the dissection in the sense that each singular simplex is contained in some $X[\sigma]$. This chain complex is chain homotopy equivalent to $C(X)$ and by abuse of notation we still denote it $C(X)$. It becomes a chain complex in $\BB (\ZZ_\ast (K))$ by 
\[
 C(X)(\sigma) = C(X[\sigma],\del X[\sigma]) \; \textup{for} \; \sigma \in K
\]
with $n$-dual
\[
 \Sigma^n T_\ast C(X) (\sigma) = C^{n-|\sigma|-\ast}(X[\sigma]).
\]
There is a chain map $\del_\sigma \co C(X) \ra S^{|\sigma|}
C(X[\sigma],\del X [\sigma])$, defined in \cite[Definition
8.2]{Ranicki(1992)}, the image of a chain $[X] \in C(X)_n$ is
denoted $[X[\sigma]] \in C(X[\sigma],\del X [\sigma])_{n-|\sigma|}$.
A generalization of the relative symmetric construction \ref{con:rel-sym} gives a chain map 
\[
 \varphi_K \co C(X) \ra \Ws{C(X)} \; \textup{over} \; \ZZ_\ast (K)
\]
called the \emph{symmetric construction over } $K$, which evaluated on a cycle $[X] \in C(X)_n$ produces an $n$-dimensional symmetric complex $(C(X),\varphi_K [X])$ over $\ZZ_\ast (K)$ whose component
\[
\varphi_K ([X]) (\sigma)_0 \co C^{n-|\sigma|-\ast} (X[\sigma]) \ra
C(X[\sigma],\del X[\sigma])
\]
is the cap product with the class $[X[\sigma]]$.
\end{construction}

\begin{construction} \label{con:sym-construction-over-cplxs-lower-star-mfd}
More generally, let $X$ be an $n$-dimensional topological manifold and let $r \co X \ra K$ be a map, transverse to the dual cells $D(\sigma,K)$ for all $\sigma \in K$. Any map can be so deformed by topological transversality. In this situation we obtain an analogous $K$-dissection. The resulting complex $(C (X),\varphi_K [X])$ is now an $n$-dimensional symmetric algebraic complex in $\Lambda (\ZZ)_\ast (K)$ since the maps
\[
\varphi (\sigma)_0 \co C^{(n-|\sigma|-\ast} (X[\sigma]) \ra C(X[\sigma],\del X[\sigma])
\]
are the cap products with the fundamental classes $[X[\sigma]] \in C_{n-|\sigma|} (X[\sigma],\del X[\sigma])$ and hence chain homotopy equivalences and hence their mapping cones are contractible. Here we are using the fact that each $X[\sigma]$ is an $n-|\sigma|$-dimensional manifold with boundary and hence satisfies Poincar\'e duality.
\end{construction}

\begin{construction} \label{con:quad-construction-over-cplxs-lower-star}
Analogously let $(f,b) \co M \ra X$ be a degree one normal map of closed $n$-dimensional topological manifolds. We can make $f$ transverse to the $K$-dissection of $X$ in a sense that
each preimage
\[
 (M[\sigma],\del M[\sigma]) := f^{-1} (X[\sigma],\del X[\sigma])
\]
is an $(n-|\sigma|)$-dimensional manifold with boundary and each restriction
\[
 (f[\sigma],f[\del \sigma]) \co  (M[\sigma],\del M[\sigma]) \ra (X[\sigma],\del X[\sigma])
\]
is a degree one normal map. The stable Umkehr map $F \co \Sigma^p X_+ \ra \Sigma^p M_+$ for some $p$ induces by a generalization of the relative quadratic construction \ref{con:rel-quad-htpy} a chain map
\[
 \psi_{K} \co \Sigma^{-k} C(X) \ra \Wq{C(M)} \; \textup{over} \; \ZZ_\ast (K)
\]
called the \emph{quadratic construction over} $K$. Evaluated on the fundamental class $[X] \in C_{n} (X, \del X)$ produces an $n$-dimensional quadratic algebraic complex in the category $\Lambda(\ZZ)_\ast (K)$. The mapping cone $\sC (f^!)$ becomes a complex over $\ZZ_\ast (K)$ by $\sC (f^{!}) (\sigma) = \sC (f (\sigma)^{!},f (\del \sigma)^{!})$. The chain map $e \co C(M) \ra \sC (f^!)$ in $\ZZ_\ast (K)$ produces an $n$-dimensional quadratic complex in $\Lambda (\ZZ)_\ast (K)$ 
\[
 \big( C(f^!), e_\% \psi_{K} [X] \big). 
\]
\end{construction}


\section{Assembly} \label{sec:assembly}


Assembly is a map that allows us to compare the concepts of the local Poincar\'e duality introduced in section \ref{sec:cat-over-cplxs} and the global Poincar\'e duality in section \ref{sec:algebraic-cplxs}. It is formulated as a functor of algebraic bordism categories. 

\begin{prop}
The functor of additive categories $A:\ZZ_{*}(K)\ra \ZZ[\pi_{1}(K)]$
defined by
\[
 M\mapsto \sum_{\tilde{\sigma}\in\tilde{K}}M(p(\tilde{\sigma}))
\]
defines a functor of algebraic bordism categories.
\[
 A \co \Lambda (\ZZ)_\ast (K) \ra \Lambda (\ZZ[\pi_1 (K)])
\]
and hence homomorphisms
\[
 A \co L^n (\Lambda (\ZZ)_\ast (K)) \ra L^n (\Lambda (\ZZ[\pi_1 (K)])) \qquad A \co L_n (\Lambda (\ZZ)_\ast (K)) \ra L_n (\Lambda (\ZZ[\pi_1 (K)]))
\]

\end{prop}

\begin{expl} \label{expl:assembly-of-symmetric-signature-over-K} \cite[Example 9.6]{Ranicki(1992)}
Let $X$ be an $n$-dimensional topological manifold with a map $r \co X \ra K$. In Construction \ref{con:sym-construction-over-cplxs-lower-star-cplx} there is described how to associate to $X$ an $n$-dimensional SAC $(C,\varphi)$ in $\Lambda (\ZZ)_\ast (K)$. The assembly $A(C,\varphi)$ is then the $n$-dimensional SAPC $\ssign (X) = (C(\widetilde{X}),\varphi ([X]))$ in $\Lambda (\ZZ[\pi_1 (K)])$ described in Construction \ref{constrn:symmetric construction}.
\end{expl}

\begin{expl} \label{expl:assembly-of-quadratic-signature-over-K} \cite[Example 9.6]{Ranicki(1992)}
Let $(f,b) \co M \ra X$ be a degree one normal map of closed $n$-dimensional topological manifolds.  In Construction \ref{con:quad-construction-over-cplxs-lower-star} there is described how to associate to $(f,b)$ an $n$-dimensional QAC $(C,\psi)$ in $\Lambda (\ZZ)_\ast (K)$. The assembly $A(C,\varphi)$ is then the $n$-dimensional QAPC $\qsign (f,b) = (C(f^{!}),e_\% \psi [X])$ in $\Lambda (\ZZ[\pi_1 (K)])$ described in Construction \ref{constrn:quadconstrn}.
\end{expl}

It is convenient to factor the assembly map into two maps. The reason is that we have nice localization sequences for a functor of algebraic bordism categories when the underlying category with chain duality is fixed and the functor is an inclusion. Hence we define

\begin{defn} \label{defn:Lambda-K-category}
Let $\Lambda (\ZZ)$ be the algebraic bordism category of Example \ref{expl:alg-bord-cat-ring} and $K$ a locally finite simplicial complex. Then the triple
\[
\Lambda (\ZZ) (K) = (\AA_\ast (K),\BB_\ast (K),\CC (K))
\]
where the subcategory $\CC (K)$ consists of the chain complexes $C \in \BB_\ast (K)$ such that $A(C) \in \CC (\ZZ[\pi_1 (K)])$.
\end{defn}

Hence, for example, an $n$-dimensional symmetric complex $(C,\varphi)$ in $\Lambda (\ZZ) (K)$ will be a complex over $\ZZ_\ast (K)$, which will only be globally Poincar\'e in the sense that $A(C,\varphi)$ will be an $n$-dimensional SAPC over $\ZZ[\pi_1 K]$, but the duality maps $\varphi (\sigma) \co \Sigma^n TC (\sigma) \ra C(\sigma)$ do not have to be chain homotopy equivalences for a particular simplex $\sigma \in K$.  

\begin{prop}
The assembly functor factors as
\[
A \co \Lambda (\ZZ)_\ast (K) \ra \Lambda (\ZZ) (K) \ra \Lambda (\ZZ[\pi_1 (K)])
\]
\end{prop}

Furthermore Ranicki proves the following algebraic $\pi-\pi$-theorem\footnote{The name is explained at the beginning of \cite[chapter 10]{Ranicki(1992)}}:

\begin{prop} \textup{\cite[chapter 10]{Ranicki(1992)}} \label{prop:algebraic-pi-pi-theorem}
The functor $\Lambda(\ZZ) (K) \ra \Lambda (\ZZ[\pi_1 (K)])$ induces an isomorphism on quadratic $L$-groups
\[
 L_n (\Lambda (\ZZ) (K)) \cong L_n (\ZZ [\pi_1 (K)])
\]
\end{prop}

It follows that when we want to compare local and global Poincar\'e duality it is enough to study the map
\begin{equation}
A \co L_n (\Lambda (\ZZ)_\ast (K)) \ra L_n (\Lambda (\ZZ) (K)). 
\end{equation}


\section{$\bL$-Spectra}  \label{sec:spectra}


The technology of the previous sections also allows us to construct
$L$-theory spectra whose homotopy groups are the already defined
$L$-groups. Spectra give rise to generalized (co-)homology theories
via the standard technology of stable homotopy theory. That is also
the main reason for their introduction in $L$-theory.

These spectra are constructed as spectra of $\Delta$-sets, alias
simplicial sets without degeneracies. We refer the reader to
\cite[chapter 11]{Ranicki(1992)} for the detailed definition as well
as for the notions of Kan $\Delta$-sets, the geometric product $K
\otimes L$, the smash product $K \wedge L$, the function
$\Delta$-sets $L^K$, the fiber and the cofiber of a map of
$\Delta$-sets, the loop $\Delta$-set $\Omega K$ and the suspension
$\Sigma K$ as well as the notion of an $\Omega$-spectrum of
$\Delta$-sets.

Below, $\Delta^n$ is the standard $n$-simplex, $\Lambda$ is an
algebraic bordism category and $K$ is a finite $\Delta$-set.

\begin{defn} 
Let $\bL_{n}(\Lambda)$, $\bL^{n}(\Lambda)$ and $\bNL^{n}(\Lambda)$
be pointed $\Delta$-sets defined by
    \begin{align*}
        \bL^{n}(\Lambda)^{(k)}   & = \{n\text{-dim. symmetric complexes in }\Lambda^{*}(\Delta^{k})\}, \\
        \bL_{n}(\Lambda)^{(k)}   & = \{n\text{-dim. quadratic complexes in }\Lambda^{*}(\Delta^{k})\, \\
        \bNL^{n}(\Lambda)^{(k)}  &= \{n\text{-dim. normal complexes in }\Lambda^{*}(\Delta^{k+n})\}.
    \end{align*}
    The face maps are induced by the face inclusions $\partial_{i}:\Delta^{k-1}\ra\Delta^{k}$
    and the base point is the 0-chain complex.
\end{defn}


\begin{prop}
We have $\Omega$-spectra of pointed Kan $\Delta$-sets
\[
\bL^{\bullet}(\Lambda) \!\coloneqq \{\bL^{n}(\Lambda)\;|\;n\in\ZZ\}
\quad \! \bL_{\bullet}(\Lambda) \!\coloneqq
\{\bL_{n}(\Lambda)\;|\;n\in\ZZ\} \quad \! \bNL^{\bullet}(\Lambda)
\!\coloneqq \{\bNL^{n}(\Lambda)\;|\;n\in\ZZ\}
\]
with homotopy groups
\[
\pi_{n}(\bL^{\bullet}(\Lambda)) \cong L^{n}(\Lambda) \quad
\pi_{n}(\bL_{\bullet}(\Lambda)) \cong L_{n}(\Lambda) \quad
\pi_{n}(\bNL^{\bullet}(\Lambda)) \cong NL^{n}(\Lambda)
\]
\end{prop}

\begin{rem}
The indexing of the $L$-spectra above is the opposite of the usual
indexing in stable homotopy theory. Namely, if $\bE$ is any of the
spectra above we have $\bE_{n+1} \simeq \Omega \bE_n$.
\end{rem}

\begin{notation}
To save space we will abbreviate
\[
\bL^{\bullet} = \bL^{n}(\Lambda (\ZZ)) \quad \bL_{\bullet} =
\bL_{n}(\Lambda (\ZZ)) \quad \bNL^{\bullet} = \bNL^{n}(\widehat \Lambda
(\ZZ)).
\]
\end{notation}

We note that the exact sequences from Propositions \ref{propn:LESL},
\ref{prop:relativ-L-groups}, and \ref{prop:inclusion-les} can be
seen as the long exact sequences of the homotopy groups of fibration
sequences of spectra. We are mostly interested in the following
special case.

\begin{prop}  \label{prop:fib-seq-of-quad-sym-norm}
Let $R$ be a ring with involution. Then we have a fibration sequence
of spectra
\[
\bL_\bullet (\Lambda(R)) \ra \bL^\bullet (\Lambda(R)) \ra \bNL^\bullet (\Lambda(R)).
\]
\end{prop}

\begin{proof}
 Consider the fiber of the map of spectra $\bL^\bullet (\Lambda(R)) \ra \bNL^\bullet (\Lambda(R))$. Use algebraic surgery to identify it with $\bL_\bullet (\Lambda(R))$ just as in the proof of Proposition \ref{propn:LESL}.
\end{proof}

In fact the $L$-theory spectra are modeled on some geometric spectra. We will use the notion of a $(k+2)$-ad (of spaces) and manifold $(k+2)$-ads as defined in \cite[\S 0]{Wall(1999)}.

\begin{defn}
Let $n\in\ZZ$ and $\bbOmega[n]{STOP}$ and $\bbOmega[n]{N}$ be pointed
$\Delta$-sets defined by
\begin{align*}
    (\bbOmega[n]{STOP})^{(k)} = & \{ (M,\partial_0M,\ldots,\partial_kM) \; | \; (n+k)\textup{-dimensional manifold} \\ 
    & \textup{$(k+2)$-ad such that } \partial_0 M \cap \ldots \cap \partial_k M = \emptyset\} \\
    (\bbOmega[n]{N})^{(k)} = & \{ (X,\nu,\rho) \; | \; (n+k)\textup{-dimensional normal space $(k+2)$-ad} \\
    & X = (X, \partial_0 X,\ldots,\partial_k X)  \; \textup{such that} \; \partial_0 X \cap \ldots\cap \partial_k X = \emptyset, \\ & \nu \co X \ra \BSG (r) \textup{ and } \rho:\Delta^{n+k+r}\ra \Th(\nu_X) \; \textup{such that} \\ & \rho (\partial_i\Delta^{n+k+r}) \subset \Th(\nu_{\partial_{i} X}) \}
\end{align*}
Face maps $\partial_i:(\bbOmega[n]{?})^{(k)} \ra (\bbOmega[n]{?})^{(k-1)}$, $0\leq i\leq k$ are given in both cases by
\[
    \partial_i(X) = (\partial_iX, \partial_0X\cap \partial_iX,\ldots,  \partial_{i-1}X\cap \partial_iX,\partial_{i+1}X\cap \partial_iX,\ldots, \partial_nX\cap \partial_iX, ).
\]
Here a convention is used that an empty space is a manifold (normal space) of any dimension $n \in \ZZ$ and it is a base point in all the dimensions.
\end{defn}

\begin{prop}
We have $\Omega$-spectra of pointed Kan $\Delta$-sets
\[
\bbOmega[\bullet]{STOP} \coloneqq \{\bbOmega[n]{STOP}\;|\;n\in\ZZ\}
\quad \bbOmega[\bullet]{N} \coloneqq\{\bbOmega[n]{N}\;|\;n\in\ZZ\}
\]
with homotopy groups
\[
\pi_n(\bbOmega{STOP}) = \Omega^{\STOP}_n \quad \pi_n(\bbOmega{N}) =
\Omega^N_n.
\]
\end{prop}

\begin{defn}
For $n \in \ZZ$ let $\Sigma^{-1} \bbOmega[n]{N,\STOP}$ be the
pointed $\Delta$-set defined as the fiber of the map of
$\Delta$-sets
\[
 \Sigma^{-1} \bbOmega[n]{N,\STOP} = \textup{Fiber} (\bbOmega[n]{STOP} \ra \bbOmega[n]{N} )
\]
The collection $\Sigma^{-1} \bbOmega[n]{N,\STOP}$ becomes an
$\Omega$-spectrum of $\Delta$-sets.
\end{defn}

\begin{rem}
Again, the indexing of the above spectra is the opposite of the
usual indexing in stable homotopy theory. To see that the spectra
are indeed $\Omega$-spectra observe that an $(n+1+k-1)$-dimensional
$(k-1+2)$-ad is the same as an $(n+k)$-dimensional $(k+2)$-ad whose faces
$\del_0$ and $\del_1 \ldots \del_k$ are empty. Similar observation
is used in the algebraic situation.
\end{rem}

Hence we have a homotopy fibration sequence of spectra
\begin{equation} \label{eqn:htpy-fib-seq-of-bordism-spectra}
 \Sigma^{-1} \bbOmega[\bullet]{N,\STOP} \ra \bbOmega[\bullet]{STOP} \ra \bbOmega[\bullet]{N}
\end{equation}

The fibration sequences from Proposition
\ref{prop:fib-seq-of-quad-sym-norm} and of
(\ref{eqn:htpy-fib-seq-of-bordism-spectra}) are related by the
signature maps as follows.

\begin{prop} \label{prop:signatures-on-spectra-level}
The relative symmetric construction produces
\[
\textup{(1)}\quad \ssign \co \bbOmega[n]{STOP} \ra \bL^n(\Lambda(\ZZ)) \; \leadsto \; \ssign \co
\bbOmega[\bullet]{STOP} \ra \bL^\bullet
\]
The relative normal construction produces
\[
\textup{(2)}\quad \nsign \co \bbOmega[n]{N} \ra \bNL^n(\widehat \Lambda(\ZZ))
\;  \leadsto \; \nsign \co \bbOmega[\bullet]{N} \ra \bNL^\bullet.
\]
The relative normal construction together with the fibration
sequence from Proposition \ref{prop:fib-seq-of-quad-sym-norm}
produces
\[
\textup{(3)}\quad \qsign \co \Sigma^{-1} \bbOmega[n]{N,\STOP} \ra
\bL_n(\Lambda(\ZZ)) \;  \leadsto \; \qsign \co \Sigma^{-1}
\bbOmega[\bullet]{N,\STOP} \ra \bL_\bullet.
\]
\end{prop}

\begin{proof}
For (1) use Construction
\ref{con:sym-construction-over-cplxs-upper-star-cplx} which is just
a generalization of the relative symmetric construction. For (2) the
relative normal construction can be used. The full details are
complicated, they can be found in \cite[section 7]{Weiss-II(1985)}.
For (3) observe that the relative normal construction provides us
with a map to the fiber of the map $\bL^{n} (\Lambda(\ZZ)) \ra
\bNL^{n} (\Lambda(\ZZ))$. The identification of this fiber from Proposition
\ref{prop:fib-seq-of-quad-sym-norm} produces the desired map.
\end{proof}


\section{Generalized homology theories} \label{sec:gen-hlgy-thies}


Now we come to the use of the spectra just defined to produce (co-)homology. Definition \ref{defn:co-hlgy-via-spectra} below contains the formulas. In addition the $S$-duality gives an opportunity to express homology as cohomology and vice versa. In our application it turns out that the input we obtain is of cohomological nature, but we would like to think of it in terms of homology. Therefore the strategy is adopted which comes under the slogan: ``homology is the cohomology of the $S$-dual''. Here in fact a simplicial model for the $S$-duality will be useful when we work with particular cycles. For $L$-theory spectra a relation to the $L$-groups of algebraic bordism categories from section \ref{sec:cat-over-cplxs} will be established.

The following definitions are standard.

\begin{defn} \label{defn:co-hlgy-via-spectra}
Let $\bE$ be an $\Omega$-spectrum of Kan $\Delta$-sets and let $K$
be locally finite $\Delta$-set.

(1) The \emph{cohomology with $\bE$-coefficients} is defined by
\[
    H^n(K;\bE) = \pi_{-n}(\bE^{K_{+}}) = [K_{+}, \bE_{-n}]
\]
where $\bE^{K_+}$ is the mapping $\Delta$-set given by
\[
    (\bE^{K_{+}}_{-n})^{(p)} =  \{ K_{+}\otimes \Delta^p\ra \bE_{-n} \}
\]

(2) The \emph{homology with $\bE$-coefficients} is defined by
\[
    H_n(K;\bE) = \pi_{n}(K_+ \wedge \bE) = \textup{colim} \; \pi_{n+j} (K_+ \wedge \bE_{-j} )
\]
where $K_+ \wedge \bE$ is the $\Omega$-spectrum of $\Delta$-sets
given by
\[
    (K_+ \wedge \bE) =  \{ \textup{colim} \; \Omega^j (K_+ \wedge \bE_{n-j}) \; | \; n \in \ZZ \}.
\]
\end{defn}

What follows is a combinatorial description of $S$-duality from \cite{Whitehead(1962)} and \cite{Ranicki(1992)}.

\begin{defn}
Let $K \subset L$ be an inclusion of a simplicial subcomplex. The
\emph{supplement of} $K$ \emph{in} $L$ is the subcomplex of the
barycentric subdivision $L'$ defined by
\[
\overline K = \{ \sigma' \in L' \; | \; \textup{no face of} \;
\sigma' \; \textup{is in} \; K' \} = \bigcup_{\sigma \in L, \sigma
\notin K} D(\sigma,L) \subset L'
\]
\end{defn}

Next we come to the special case when $L = \del \Delta^{m+1}$. In
this case the dual cell decomposition of $\del \Delta^{m+1}$ can in
fact be considered as a simplicial complex, which turns out to be
convenient. First a definition.

\begin{defn}\label{defn:sigma-m}
Define the simplicial complex $\Sigma^m$ by
\begin{align*}
 (\Sigma^m)^{(k)} &= \{ \sigma^\ast \; | \; \sigma \in (\del \Delta^{m+1})^{(m-k)} \} \\
 \del_i \co (\Sigma^m)^{(k)} & \ra (\Sigma^m)^{(k-1)} \; \textup{for} \; 0 \leq i \leq k \; \textup{is} \; \del_i \co \sigma^\ast \mapsto (\delta_i \sigma)^\ast
\end{align*}
with $\delta_i \co (\del \Delta^{m+1})^{(m-k)} \ra (\del
\Delta^{m+1})^{(m-k+1)}$ given by
\[
\delta_i \co \sigma = \{0,\ldots,m+1\} \setminus \{j_0,\ldots,j_k\}
\mapsto \sigma \cup \{j_i\}, \quad (j_0 < j_1 < \cdots < j_k).
\]
\end{defn}

So $\Sigma^m$ has one $k$-simplex $\sigma^\ast$ for each
$(m-k)$-simplex $\sigma$ of $\del \Delta^{m+1}$ and $\sigma^\ast
\leq \tau^\ast$ if and only if $\sigma \geq \tau$.
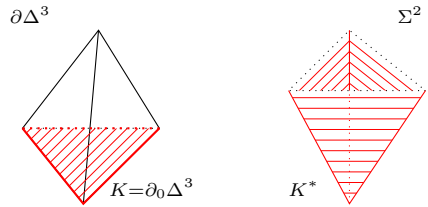
\[
\deltapicture
\]

The usefulness of this definition is apparent form the following
proposition, namely that each dual cell in $\partial\Delta^{m+1}$
appears as a simplex in $\Sigma^m$.

\begin{prop}
There is an isomorphism of simplicial complexes
\[
\Phi \co (\Sigma^m)' \xra{\cong} (\del \Delta^{m+1})'
\]
such that for each $\sigma \in K \subset \del \Delta^{m+1}$ we have
\[
\Phi(\sigma^\ast) = D(\sigma,\partial\Delta^{m+1}) \quad
\textup{and} \quad \Phi(\sigma^\ast) \cap K' = D(\sigma,K)
\]
\end{prop}

Notice that since $\del \Delta^{m+1}$ is an $m$-dimensional manifold
the dual cell $D(\sigma,\del \Delta^{m+1})$ is a submanifold with
boundary of dimension $(m-|\sigma|)$ which coincides with the
dimension of $\sigma^\ast$.

\begin{proof}
The isomorphism $\Phi$ is given by the formula
\begin{align*}
\xymatrix{
    (\sigma^\ast)'
     =
    \{ \hat{\sigma}^\ast_0\hat{\sigma}^\ast_1\ldots\hat{\sigma}^\ast_p \;|\; \sigma^\ast_p < \ldots <\sigma^\ast_1< \sigma^\ast_0\leq\sigma^\ast\} \ar[d]^{\Phi},
    &
    \hat{\sigma}^\ast_0\hat{\sigma}^\ast_1\ldots\hat{\sigma}^\ast_p \ar@{|->}[d]
    \\
    D(\sigma,\partial\Delta^{m+1})
    =
     \{ \hat{\sigma}_0\hat{\sigma}_1\ldots\hat{\sigma}_p \;|\; \sigma \leq  \sigma_0< \sigma_1<\ldots<\sigma_p \},
    &
    \hat{\sigma}_0\hat{\sigma}_1\ldots\hat{\sigma}_p
    }
\end{align*}
\end{proof}
The isomorphism of course induces a homeomorphism of geometric
realizations. For $m=2$ it looks like this:
\[
    \flippedsigmapicture
\]

\begin{prop} \textup{\cite[Prop. 12.4]{Ranicki(1992)}} \label{prop:hom=cohom+S}
Let $\bE$ be a $\Omega$-spectrum of Kan $\Delta$-sets and $K$ a
finite simplicial complex. Then for $m\in\NN$ large enough we have
\[
    H_n(K;\bE) \cong H^{m-n}(\Sigma^m,\splK; \bE)
\]
\end{prop}

\begin{proof}
The above proposition allows us to think of $K$ as being embedded in
$\Sigma^m$ and the complex $\Sigma^m/\splK$ is the quotient of
$\Sigma^m$ by the complement of a neighborhood of $K$. This is a
well known construction of an $m$-dimensional $S$-dual of $K$, which
is proved in detail for example in \cite[p. 265]{Whitehead(1962)}.
The construction there provides an explicit simplicial construction
of a map $\Delta' \co \Sigma^m\ra K_{+}\wedge
(\Sigma^m/\overline{K})$ which turns out to be such an  $S$-duality.
\end{proof}

We remark that if $K$ is an $n$-dimensional Poincar\'e complex with
the SNF $\nu_K \co K \ra \BSG (m-n)$ then $\Sigma^m / \splK \simeq
Th(\nu_K)$.

Now we come to the promised alternative definition of the homology
of $K$.

\begin{defn} \label{defn:E-cycles}
Let $\bE$ be an $\Omega$-spectrum of $\Delta$-sets. An
$n$-dimensional $\bE$-cycle in $K$ is a collection
\[
 x = \{x(\sigma) \in \bE_{n-m}^{(m-|\sigma|)} \; | \; \sigma\in K \}
\]
such that $\partial_ix(\sigma) =
\begin{cases}
 x(\delta_i\sigma) & \text{if } \delta_i\sigma\in K\\
 \emptyset & \text{if }\delta_i\sigma\notin K
\end{cases} (0\leq i\leq m-|\sigma|)$

A cobordism of $n$-dimensional $\bE$-cycles $x_0$, $x_1$ in $K$ is a
$\Delta$-map
\[
 y: (\Sigma^m,\overline{K})\otimes \Delta^1 \ra (\bE_{n-m},\emptyset)
\]
such that $y(\sigma\otimes i)= x_i(\sigma)\in \bE_{n-m}^{m-|\sigma}$
for $\sigma\in K$ and $i=0,1$.

\end{defn}


\begin{prop}[{\!\cite[Prop. 2.8]{Ranicki(1992)}}]
There is a bijection between the set of cobordism equivalence classes of $n$-dimensional $\bE$-cycles in $K$ and the $n$-dimensional $\bE$-homology group $H_n(K,\bE)$. 
\end{prop}
\begin{proof}
A $n$-dimensional $\bE$-cycle $x$ defines a $\Delta$-map
\[
(\Sigma^m,\splK)\ra \bE_{n-m}, \sigma^\ast\mapsto \begin{cases}
x(\sigma) &  \sigma\in K \\ \emptyset & \sigma\notin K\end{cases}
\]
and cobordism relation of cycles corresponds to the homotopy
relation of $\Delta$-maps.
\end{proof}

\begin{prop}[{\!\cite[Prop. 13.7]{Ranicki(1992)},\cite[Remark 14.2]{Laures-McClure2009}}]  \label{prop:L-thy-of-star-cat-is-co-hlgy}
\renewcommand{\labelenumi}{(\roman{enumi})}
Let $K$ be a finite simplicial complex and $\Lambda$ an algebraic
bordism category. Then
    \begin{enumerate}
        \item $\bL_{\bullet}(\Lambda)^{K_{+}}\simeq\bL_{\bullet}(\Lambda^{*}(K))$ and $\bL^{\bullet}(\Lambda)^{K_{+}}\simeq\bL^{\bullet}(\Lambda^{*}(K))$
        \item $K_{+}\wedge\bL_{\bullet}(\Lambda)\simeq\bL_{\bullet}(\Lambda_{*}(K))$ and $K_{+}\wedge\bL^{\bullet}(\Lambda)\simeq\bL^{\bullet}(\Lambda_{*}(K))$
    \end{enumerate}
\end{prop}

\begin{cor}
For the algebraic bordism category $\Lambda = \Lambda (\ZZ)$ we have
    \[
    L_{n}(\Lambda (\ZZ)_{*}(K)) \cong H_{n}(K, \bL_{\bullet}) \textup{  and  } \\
 L^{n}(\Lambda (\ZZ)_{*}(K)) \cong H_{n}(K, \bL^{\bullet}).
    \]
\end{cor}

\begin{proof}[Proof of Corollary]
For any $\Lambda$ we have
\[
    L_{n}(\Lambda_\ast (K))  \cong \pi_{n}(\bL_{\bullet} (\Lambda_{*}(K))) \cong \pi_{n}(K_{+}\wedge \bL_{\bullet}(\Lambda)) \cong H_{n} (K,\bL_{\bullet}(\Lambda))
\]
and similarly in the symmetric case.  
\end{proof}

\begin{proof}[Proof of (i)]
Since the morphisms in the category $\Lambda^\ast (K)$ only go from bigger to smaller simplices we can split an $n$-dimensional QAC $(C,\varphi)\in\Lambda^{*}(K)$ over $K$ into a collection of $n$-dimensional QAC $\{(C_{\sigma},\varphi_{\sigma}) \in \Lambda^{*}(\Delta^{|\sigma|})\}$ over standard simplices such that the $(C_{\sigma},\varphi_{\sigma})$ are related to each other in the same way the corresponding simplices are related to each other in $K$, i.e. $C_\sigma(\partial_{i}\sigma)=C_{\partial_{i}\sigma}(\partial_{i}\sigma)$ for all $\sigma\in K$. The complex $(C_{\sigma},\varphi_{\sigma})$ is a $|\sigma|$-simplex in $\bL_{n}(\Lambda)$ and the compatibility conditions are contained in the notion of $\Delta$-maps.

Hence we get
   \begin{eqnarray*}
            (C,\varphi) & = & \{n\text{-dim. QAC }(C_\sigma,\varphi_\sigma)\in\Lambda^{*}(\Delta^{|\sigma|})\;|\;
            \\
            & &\qquad\qquad \sigma\in K\text{ and } C_\sigma(\partial_{i}\sigma)=C_{\partial_{i}\sigma}(\partial_{i}\sigma)\}\\
            & = & \Delta\text{-map } f_{C}: K_{+}\ra \bL_{n}(\Lambda)\text{ with } f(\sigma) = (C_\sigma,\varphi_{\sigma}) \text{ for } \sigma\in K_{+}
    \end{eqnarray*}
Thus
    \begin{eqnarray*}
            \bL_{n}(\Lambda^{*}(K))^{(k)} & = & \{n\text{-dim. QAC } (C,\varphi)\in \Lambda^{*}(K)^{*}(\Delta^{k}) \simeq \Lambda^{*}(K\otimes \Delta^{k})\}\\
            &   = & \{ f:(K\otimes\Delta^{k})_{+}\ra\bL_{n}(\Lambda)\;|\; f \text{ is a pointed } \Delta\text{-map}\}\\
            & = &(\bL_{n}(\Lambda)^{K_{+}})^{(k)}
    \end{eqnarray*}
\end{proof}

\begin{proof}[Proof of (ii)]
For $m\in\NN$ large enough consider an embedding
$i:K\ra\partial\Delta^{m+1}$, the complex $\Sigma^{m}$ and the
supplement $\overline{K}$ in $\Sigma^m$ as in
Definition~\ref{defn:sigma-m}. The first observation is that there
is an isomorphism of algebraic bordism categories
\[
\Lambda_{*}(K) \cong \Lambda^{*}(\Sigma^m,\bar{K})
\]
This follows from the existence of the one-to-one correspondence
$\sigma \leftrightarrow \sigma^\ast$ between $k$-simplices of $K$
and $(m-k)$-simplices of $\Sigma^m \smallsetminus \overline{K}$ which have
the property $\sigma \leq \tau$ if and only if $\sigma^\ast \geq
\tau^\ast$ and the symmetry in the definition of the dualities
$T_\ast$ and $T^\ast$.

The observation leads to
\[
\bL_\bullet (\Lambda_\ast (K)) \cong \bL_{\bullet}(\Lambda^{*}
(\Sigma^m,\overline{K})) \simeq
\bL_{\bullet}(\Lambda)^{(\Sigma^{m},\overline{K})} \simeq K_{+} \wedge
\bL_{\bullet}(\Lambda).
\]
where the last homotopy equivalence is a spectrum version of the
isomorphism in Proposition \ref{prop:hom=cohom+S}.
\end{proof}

\begin{remark}
Recall that in section \ref{sec:cat-over-cplxs} we have defined various structured algebraic complexes over $X$. By theorems of this section some of them represent homology classes with coefficients in the $L$-theory spectra. Alternatively to the explicit construction above a different approach in \cite{Weiss(1992)} proves that these homology groups $H_n(K,\bE)$ are induced by homotopy invariant and excisive functors $K\ra\bE(\Lambda_*(K))$ and hence this construction is natural in $K$.
\end{remark}

\begin{definition} \label{defn:sym-sign-over-X}
Let $X$ be an $n$-dimensional closed topological manifold with a map $r \co X \ra K$ to a simplicial complex. The cobordism class of the $n$-dimensional SAC in $\Lambda (\ZZ)_\ast (K)$ obtained from any choice of the fundamental class $[X] \in C_n (X)$ in Construction \ref{con:sym-construction-over-cplxs-lower-star-mfd} does not depend on the choice of $[X]$ and hence defines an element
\[
 \ssign_{K} (X) = (C(X),\varphi_K ([X])) \in H_n (K;\bL^\bullet)
\]
called the \emph{symmetric signature} of $X$ over $K$.
\end{definition}

\begin{definition} \label{defn:stop-sign}
Let $X$ be an $n$-dimensional closed topological manifold with a map $r \co X \ra K$ to a simplicial complex. Recall the spectrum $\bOmega^{\STOP}_\bullet$ from section \ref{sec:spectra}. Note that the $K$-dissection of $X$ obtained by making $r$ transverse to the dual cells gives a compatible collection of manifolds with boundary so that the assignment $\sigma \ra X[\sigma]$ is precisely an $n$-dimensional $\bOmega^{\STOP}_\bullet$-cycle. We call it the $\STOP$-\emph{signature} of $X$ over $K$ and denote
\[
 \stopsign_{K} (X) \in H_n (K;\bOmega^{\STOP}_\bullet).
\]
\end{definition}

\begin{rem}
The symmetric signature $\ssign_K (X)$ can be seen as obtained from the STOP-signature $\stopsign_K (X)$ by applying the symmetric signature map on the level of spectra, that means the map $\ssign$ from Proposition \ref{prop:signatures-on-spectra-level}. In fact the $\STOP$-signature and hence the symmetric signature only depend on the oriented cobordism class of $X$, and so we obtain a homomorphism
\[
 \ssign_{K} \co \Omega^{\STOP}_n (K) \ra H_n (K;\bL^\bullet).
\]
\end{rem}

\begin{definition} \label{defn:quad-sign-over-X}
Let $(f,b) \co M \ra X$ be a degree one normal map of $n$-dimensional closed topological manifolds and let $r \co X \ra K$ be a map to a simplicial complex. The cobordism class of the $n$-dimensional QAC in $\Lambda (\ZZ)_\ast (K)$ obtained from any choice of the fundamental class $[X] \in C_n (X)$ in Construction \ref{con:quad-construction-over-cplxs-lower-star} does not depend on the choice of $[X]$ and hence defines an element
\[
\qsign_{K} (f,b) = (\sC(f^{!}),e_\% \psi_K ([X]) \in H_n (K;\bL_\bullet)
\]
called the \emph{quadratic signature} of the degree one normal map $(f,b)$ over $K$. In fact the quadratic signature only depends on the normal cobordism class of $(f,b)$ in the set of normal invariants $\sN (X)$ and provides us with a function
\[
 \qsign_{K} \co \sN(X) \ra H_n (K;\bL_\bullet).
\]
\end{definition}

In order to obtain an analogue of Proposition \ref{prop:L-thy-of-star-cat-is-co-hlgy} for $\bNL^\bullet$ spectra we need to introduce yet another algebraic bordism category associated to $\Lambda$ and $K$.

\begin{defn} \label{defn:Lambda-hat-category}
Let $\Lambda=(\AA,\BB,\CC)$ be an algebraic bordism category and $K$ a locally finite simplicial complex. Define the algebraic bordism category
\[
\widehat \Lambda (K) = (\AA_\ast (K),\BB_\ast (K),\BB_\ast (K))
\]
where $\AA_\ast (K)$ and $\BB_\ast (K)$ are as in section \ref{sec:cat-over-cplxs}.
\end{defn}

\begin{prop} \textup{\cite[Proposition 14.5]{Ranicki(1992)}}
Let $K$ be a finite simplicial complex and $\Lambda$ an algebraic bordism category. Then
\[
 K_{+}\wedge\bNL^{\bullet}(\Lambda) \simeq \bNL^{\bullet}(\widehat \Lambda (K)).
\]
\end{prop}

To complete the picture we present the following proposition which follows from Lemma \ref{lem:symmetric-poincare-means-unique-normal} and Proposition \ref{prop:L-thy-of-star-cat-is-co-hlgy}

\begin{prop} \label{prop:NL-spectrum-of-lower-star-K}
We have
\[
 \bNL^{\bullet} (\Lambda_\ast (K)) \simeq \bL^{\bullet} (\Lambda_\ast (K)) \simeq K_{+}\wedge\bL^{\bullet}(\Lambda)
\]
\end{prop}

\begin{rem} \label{rem:hlgical-assembly}
Recall the idea of the assembly map from section \ref{sec:assembly}. Via Proposition \ref{prop:L-thy-of-star-cat-is-co-hlgy} it induces a map
\[
 A \co H_n (K;\bL_\bullet) = \pi_n (K_+ \wedge \bL_\bullet) \ra L_n (\ZZ[\pi_1 (K)])) = \pi_n (\bL_\bullet (\Lambda (\ZZ[\pi_1 (K)])))
\]
If $\pi_1 (K) = 0$ then this map can be thought of as an induced map on homology by the collapse map $K \mapsto \ast$. Similarly for spectra $\bL^\bullet$ and $\bNL^\bullet$. However, this is not a phenomenon special to these spectra. In fact in \cite[chapter 12]{Ranicki(1992)} an assembly map
\[
 A \co H_n (K;\bE) \ra \pi_n (\bE)
\]
is discussed for any spectrum $\bE$, hence any homology theory. On the level of chains this map can be described via certain ``gluing'' procedure. For the spectra $\bOmega^N_\bullet$ and $\bOmega^{\STOP}_\bullet$ this procedure coincides with the geometric gluing. 
\end{rem}


\section{Connective versions} \label{sec:conn-versions}


An important technical aspect of the theory is the use of connective
versions of the $L$-theory spectra. This is related to the
difference between topological manifolds and ANR-homology manifolds.
In principle there are two ways how to impose connectivity
restrictions. One is to fix the algebraic bordism category and
modify the definition of the $L$-groups and $L$-spectra. The other
is to modify the algebraic bordism category and keep the definition
of the $L$-groups and $L$-spectra. Both ways are convenient at some
stages.

\begin{prop}
Let $\Lambda$ be an algebraic bordism category and let $q \in \ZZ$.
Then there are $\Omega$-spectra of Kan $\Delta$-sets
$\bL_\bullet \langle q \rangle (\Lambda)$, $\bL^\bullet \langle q
\rangle (\Lambda)$, $\bNL^\bullet \langle q \rangle (\Lambda)$ with
homotopy groups
\begin{align*}
\pi_n \bL_\bullet \langle q \rangle (\Lambda) & = \bL_n (\Lambda) \; \textup{for} \; n \geq q, \; 0 \; \textup{for} \; n < q  \\
\pi_n \bL^\bullet \langle q \rangle (\Lambda) & = \bL^n (\Lambda) \; \textup{for} \; n \geq q, \; 0 \; \textup{for} \; n < q\\
\pi_n \bNL^\bullet \langle q \rangle (\Lambda) & = \bNL^n (\Lambda)
\; \textup{for} \; n \geq q, \; 0 \; \textup{for} \; n < q.
\end{align*}
\end{prop}

\begin{defn} \label{defn:abc-connective-modification}
Let $\Lambda = (\AA,\BB,\CC)$ be an algebraic bordism category, and let $q \in \ZZ$. Denote $\BB
\langle q \rangle \subset \BB$ the category of chain complexes in $\BB$ which are homotopy equivalent to $q$-connected chain complexes and $\CC \langle q \rangle = \BB \langle q \rangle \cap \CC$. The algebraic bordism categories $\Lambda \langle q \rangle$ and $\Lambda \langle 1/2 \rangle$ are defined by 
\[
\Lambda \langle q \rangle = (\AA,\BB \langle q \rangle,\CC \langle q \rangle) \quad \textup{and} \quad \Lambda \langle 1/2 \rangle = (\AA,\BB\langle 0 \rangle,\CC \langle 1 \rangle)
\]
\end{defn}

\begin{notation}
 In accordance with \cite{Ranicki(1992)} we will now use the notation $\Lambda \langle q \rangle (R)$ for $\Lambda (R) \langle q \rangle$ where $R$ is a ring with involution, $q \in \ZZ$ or $q = 1/2$ and $\Lambda (R)$ is the algebraic bordism category from Example \ref{expl:alg-bord-cat-ring}. Similarly $\widehat \Lambda \langle q \rangle (R)$ stands for $\widehat \Lambda (R) \langle q \rangle$.
\end{notation}
                                                                                                           In general the groups $L^n (\Lambda \langle q \rangle (R)) = \pi_n \bL^\bullet (\Lambda \langle q \rangle (R))$ need not be isomorphic to $\pi_n \bL^\bullet \langle q \rangle (\Lambda (R))$, likewise for quadratic and normal $L$-groups. However, in certain special cases this holds, for example we have (\cite[Example 15.8]{Ranicki(1992)}):
\begin{align*}
\pi_n \bL_\bullet \langle 0 \rangle (\Lambda (\ZZ)) & \cong L_n (\Lambda \langle 0 \rangle (\ZZ)) \\
\pi_n \bL^\bullet \langle 0 \rangle (\Lambda (\ZZ)) & \cong L^n
(\Lambda \langle 0 \rangle (\ZZ))
\end{align*}

\begin{notation}
Again, to save space we abbreviate for $q \in \ZZ$ or $q = 1/2$:
\[
\bL^{\bullet} \langle q \rangle = \bL^{n}(\Lambda \langle q \rangle (\ZZ)) \quad \bL_{\bullet} \langle q \rangle = \bL_{n}(\Lambda \langle q \rangle (\ZZ)) \quad \bNL^{\bullet} \langle q \rangle = \bNL^{n}(\widehat \Lambda \langle q \rangle (\ZZ)).
\]
\end{notation}

We also obtain a connective version of Proposition \ref{prop:fib-seq-of-quad-sym-norm}.

\begin{prop} \textup{\cite[Proposition 15.16]{Ranicki(1992)}} \label{prop:fib-seq-of-quad-sym-norm-connective-version}
We have a homotopy fibration sequence
\begin{equation} 
\bL_\bullet \langle 1 \rangle \ra \bL^\bullet \langle 0 \rangle \ra
\bNL^\bullet \langle 1/2 \rangle.
\end{equation}
As a consequence we have
\[
\pi_0 \bL^\bullet \langle 0 \rangle \cong \pi_0 \bNL^\bullet \langle
1/2 \rangle \cong \ZZ.
\]
\end{prop}

Let $K$ be a simplicial complex. Now we consider the connective versions of the $L$-theory groups/spectra of algebraic bordism categories associated to $K$ in sections \ref{sec:cat-over-cplxs} and \ref{sec:gen-hlgy-thies}. Specifically we are interested in the algebraic bordism category $\Lambda \langle q \rangle (\ZZ)_\ast (K)$ from Definition \ref{defn:Lambda-star-categories} and $\widehat \Lambda \langle q \rangle (\ZZ) (K)$ from Definition \ref{defn:Lambda-hat-category}. The following proposition is an improvement on Propositions \ref{prop:L-thy-of-star-cat-is-co-hlgy}, \ref{prop:NL-spectrum-of-lower-star-K} and \ref{prop:algebraic-pi-pi-theorem}.

\begin{prop} \textup{\cite[Proposition 15.9,15.11]{Ranicki(1992)}} \label{prop:L-theory-over-K-are-homology-theories-conn-versions}
There are isomorphisms
\begin{align*}
\pi_n \bL^\bullet (\Lambda \langle q \rangle (\ZZ)_\ast (K)) & \cong H_n (K;\bL^\bullet  \langle q \rangle (\ZZ)) \\
\pi_n \bL_\bullet (\Lambda \langle q \rangle (\ZZ)_\ast (K)) & \cong H_n (K;\bL_\bullet \langle q \rangle (\ZZ)) \\
\pi_n \bNL^\bullet (\widehat \Lambda \langle q \rangle (\ZZ) (K)) & \cong H_n
(K;\bNL^\bullet \langle q \rangle (\ZZ)) \\
\pi_n \bL_\bullet (\Lambda \langle q \rangle (K)) & \cong L_n (\ZZ [\pi_1 [K]) \quad \textup{for} \; n \geq 2q.
\end{align*}
\end{prop}

In addition there is an improved version of Proposition \ref{prop:signatures-on-spectra-level} as follows:

\begin{prop} \label{prop:connective-signatures-on-spectra-level}
The relative symmetric construction produces
\[
\textup{(1)}\quad \ssign \co \bbOmega[n]{STOP} \ra
\bL^n(\Lambda \langle 0 \rangle (\ZZ)) \; \leadsto \; \ssign \co
\bbOmega[\bullet]{STOP} \ra \bL^\bullet \langle 0 \rangle
\]
The relative normal construction produces
\[
\textup{(2)}\quad \nsign \co \bbOmega[n]{N} \ra
\bNL^n(\widehat \Lambda \langle 1/2 \rangle (\ZZ)) \;  \leadsto \; \nsign \co
\bbOmega[\bullet]{N} \ra \bNL^\bullet \langle 1/2 \rangle.
\]
The relative normal construction produces
\[
\textup{(3)}\quad \qsign \co \Sigma\invers\bbOmega[n]{N,\STOP} \ra
\bL_n(\Lambda \langle 1 \rangle (\ZZ)) \;  \leadsto \; \qsign \co
\Sigma\invers\bbOmega[\bullet]{N,\STOP} \ra \bL_\bullet \langle 1 \rangle.
\]
\end{prop}

Part (1) is obvious, since a geometric situation provides only chain complexes concentrated in non-negative dimensions. Part (2) is shown in \cite[page 178]{Ranicki(1992)}. Part (3) follows from part (2) and the fibration sequence from Proposition \ref{prop:fib-seq-of-quad-sym-norm-connective-version}.

\section{Surgery sequences and the structure groups $\SS_n (X)$}  \label{sec:surgery-sequences}


Now we have assembled all the tools needed to define, for a finite simplicial complex $X$, the group $\SS_n (X)$, which is the home of the total surgery obstruction if $X$ is an $n$-dimensional Poincar\'e complex. It is important and useful to define not only the group $\SS_n (X)$ itself, but also to relate it to other groups which we might understand better. So the group $\SS_n (X)$ is placed into a commutative braid, which is obtained from the braid in section \ref{sec:alg-bord-cat} by plugging in suitable algebraic bordism categories. We recall these now.

In fact the categories are, as indicated below, various connective versions of the following categories. The underlying additive category with chain duality is
\begin{itemize}
 \item $\AA=\ZZ_{*}(X)$ is the additive category of finitely generated free $\ZZ$-modules over $X$..
\end{itemize}
Now we specify the subcategories of $\BB (\AA)$ needed to construct the braid. 
\begin{itemize}
\item $\BB = \BB(\ZZ_{*}(X))$ are the bounded chain complexes in $\AA$, 
\item $\CC = \{C\in\BB\;|\;A(C)\simeq *\}$ are the \emph{globally contractible} chain complexes in $\BB(\AA)$,
\item $\DD = \{C\in\BB\;|\;C(\sigma) \simeq * \; \forall \sigma\in K\}$ are the \emph{locally contractible} chain complexes in $\BB(\AA)$.
\end{itemize}

The precise connective versions used are indicated in the braid diagram below, which is taken from \cite[Proposition 15.18]{Ranicki(1992)}. Due to the lack of space we have omitted the underlying category $\AA$ from the notation as it is the same everywhere. We also note that obviously $\DD \langle 0 \rangle = \DD \langle 1 \rangle$: 

\vspace{1cm}

{\footnotesize
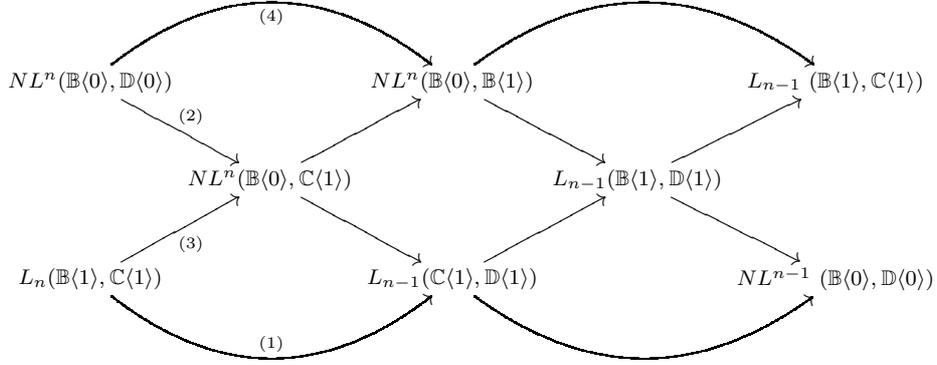
\[
\xymatrix@C=0.05em{ NL^n (\BB \langle 0 \rangle ,\DD \langle 0
\rangle) \ar^{(2)}[dr] \ar@/^2.5pc/_{(4)}[rr] & & NL^n (\BB \langle 0
\rangle,\BB \langle 1 \rangle) \ar[dr] \ar@/^2.5pc/[rr] & &
L_{n-1}\ (\BB \langle 1 \rangle,\CC \langle 1 \rangle) \\
& NL^n (\BB \langle 0 \rangle,\CC \langle 1\rangle) \ar[ur] \ar[dr]
& & L_{n-1} (\BB \langle 1 \rangle,\DD \langle 1 \rangle) \ar[ur] \ar[dr] \\
L_n (\BB \langle 1 \rangle,\CC \langle 1 \rangle) \ar_{(3)}[ur]
\ar@/_2.5pc/^{(1)}[rr] & & L_{n-1} (\CC \langle 1 \rangle,\DD \langle 1
\rangle) \ar[ur] \ar@/_2.5pc/[rr] & & NL^{n-1}\ (\BB \langle 0
\rangle,\DD \langle 0 \rangle) }
\]
}

\vspace{1cm}

Notice that the exact sequence labeled (1) is induced by the assembly functor $A \co \Lambda (\ZZ) \langle 1 \rangle_\ast (X) \ra \Lambda (\ZZ) \langle 1 \rangle (X)$ from section \ref{sec:assembly}. The other sequences are induced by analogous functors, the precise statements are left to the reader. It is more interesting at this stage that we have already identified various groups in the braid. We recapitulate using Propositions \ref{prop:L-theory-over-K-are-homology-theories-conn-versions} and \ref{prop:NL-spectrum-of-lower-star-K}:
\begin{align*}
L_n (\BB \langle 1 \rangle,\CC \langle 1 \rangle) & = L_n (\ZZ[\pi_1
(X)])  \\
L_n (\BB \langle 1 \rangle,\DD \langle 1 \rangle) & = H_n
(X;\bL_\bullet \langle 1 \rangle)  \\
NL^n (\BB \langle 0 \rangle,\DD \langle 0 \rangle) & = H_n
(X;\bL^\bullet \langle 0 \rangle))  \\
NL^n (\BB \langle 0 \rangle,\BB \langle 1 \rangle) & = H_n
(X;\bNL^\bullet \langle 1/2 \rangle)) \\
\end{align*}

The sequence containing only homology theories can be thought of as induced by the cofibration sequence from Proposition \ref{prop:fib-seq-of-quad-sym-norm-connective-version}. In addition we have new terms as follows. We keep the notation from the beginning of this section. 

\begin{defn} \cite[chapter 17]{Ranicki(1992)}
Let $X$ be a finite simplicial complex. Define the $n$-dimensional \emph{structure group} of $X$ to be 
\[
\SS_n (X) = L_{n-1} (\AA,\CC \langle 1 \rangle,\DD \langle 1 \rangle) 
\]
\end{defn}

So an element in $\SS_n (X)$ is represented by an $(n-1)$-dimensional $1$-connective QAC in $\ZZ_\ast (X)$ which is globally contractible and locally Poincar\'e. We will see in the following section how to obtain such a chain complex from a geometric situation. 

\begin{defn} \cite[chapter 17]{Ranicki(1992)}
Let $X$ be a finite simplicial complex. Define the $n$-dimensional \emph{visible symmetric $L$-group} of $X$ to be 
\[
 VL^n(X) = NL^n (\AA,\BB \langle 0 \rangle, \CC \langle 1 \rangle) 
\]
\end{defn}

The visible symmetric $L$-groups were defined by Weiss in \cite{Weiss(1992)} to clarify certain relations between the symmetric and quadratic $L$-groups. We will not need this aspect, what is important is that an element in $VL^n (X)$ is represented by an $n$-dimensional $0$-connective NAC in $\ZZ_\ast (X)$ whose underlying symmetric structure is locally $0$-connective and globally Poincar\'e. We will see in the following section a geometric situation which yields such a chain complex.

The sequence labeled (1) in the braid is known as the \emph{algebraic surgery exact sequence}:
\begin{equation}
\cdots \ra H_n (X,\bL_\bullet \langle 1 \rangle) \xra{A} L_n (\ZZ[\pi_1
(X)]) \xra{\del} \SS_n (X) \ra H_{n-1} (X;\bL_\bullet \langle 1 \rangle)
\ra \cdots
\end{equation}

Summarizing the above identification we obtain the commutative braid which will be our playground in the rest of the paper:    

\vspace{1cm}

\[
\xymatrix@C=0.5em{ H_n (X;\bL^\bullet \langle 0 \rangle) \ar[dr]
\ar@/^2.5pc/[rr] & & H_n (X;\bNL^\bullet \langle 1/2 \rangle)
\ar[dr] \ar@/^2.5pc/[rr] & &
L_{n-1}\ (\ZZ[\pi]) \\
& VL^n (X) \ar[ur] \ar[dr]^{\del} & & H_{n-1} (X;\bL_\bullet \langle 1 \rangle) \ar[ur] \ar[dr] \\
L_n (\ZZ[\pi]) \ar[ur] \ar@/_2.5pc/[rr]^{\del} & & \SS_{n} (X) \ar[ur]
\ar@/_2.5pc/[rr] & & H_{n-1}\ (X;\bL^\bullet \langle 0 \rangle) }
\]

\vspace{1cm}

\section{Normal signatures over $X$} \label{sec:normal-signatures-over-X}

As indicated in the introduction in order to define the total surgery obstruction of $X$ we need to discuss the normal and the visible signature of $X$. In this section we define the visible signature $\vsign_X (X) \in VL^n (X)$ as a refinement of the normal signature $\nsign_X (X) \in H_n (X;\bNL^\bullet \langle 1/2 \rangle)$. In fact, it should be expected that any $n$-dimensional GNC $(X,\nu,\rho)$ has an associated normal signature over $X$ which assembles to the normal signature over $\ZZ[\pi_1 X]$ defined in section \ref{sec:normal-cplxs}. However, we are not able to show such a general statement. We need to assume that the normal complex comes from a Poincar\'e complex with its SNF.\footnote{In \cite[Errata for page 103]{Errata-Ranicki(1992)} a construction for the normal signature over $X$ for a normal complex $X$ is actually given. However, since in the proof of the subsequent sections we directly use the specific properties of the construction presented in this section in the case when $X$ is Poincar\'e, we only discuss this special case.} Before we start we still need some technical preliminaries.

\begin{construction}
Recall some more ideas from \cite{Whitehead(1962)} surrounding the concept of supplement described in section
\ref{sec:gen-hlgy-thies}. Let $K \subseteq L$ be a simplicial subcomplex. The supplement is a subcomplex $\overline{K} \subseteq
L'$. As explained in \cite{Whitehead(1962)} there is an embedding
$|L| \subset |K'| \ast |\overline K|$ into the join of the two
realizations. A point in $|K'| \ast |\overline K|$ can be described
as $t \cdot x + (1-t) \cdot y$ for $x \in |K'|$, $y \in |\overline
K|$, and $t \in [0,1]$. The space $|L|$ can be decomposed as the
union of two subspaces
\begin{align*}
 N = N(K') = & \{ t \cdot x + (1-t) \cdot y \; | \; t \geq 1/2 \} \cap L \\ 
 \overline N = N(\overline K) = & \{ t \cdot x + (1-t) \cdot y \; | \; t \leq 1/2 \} \cap L.
\end{align*}
These come with obvious deformation retractions $r \co N \ra |K|$
and $\overline r \co \overline N \ra |\overline K|$.

Next denote $N(\sigma) = N \cap (|D(\sigma,K)| \ast |\overline K|)$
for $\sigma \in K$. Then we have the dissection $N = \cup_{\sigma
\in K} N(\sigma)$ and the retraction $r$ respects the dissections
of $N$ and $|K|$
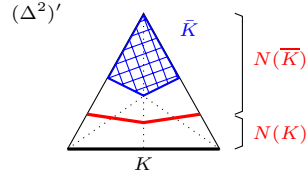
\[
  r|_{N(\sigma)} = r(\sigma) \co N(\sigma) \ra |D(\sigma,K)|.
\]

\[
    \jointriangle
\]
\end{construction}

\begin{construction} \label{con:simpl-model-for-the-thom-space-of-snf}
Consider now the case when $X$ is a finite simplicial Poincar\'e
complex of dimension $n$ which we embed into $\del \Delta^{m+1}$, that means $K = X$ and $L = \del \Delta^{m+1}$ in the above notation. For $m$ large enough the homotopy fiber of the projection map $\del r \co \del N = N \cap \overline N \ra X$ is homotopy equivalent to
$S^{m-n-1}$ and the associated spherical fibration is the SNF
$\nu_X$. In more detail, there is a $(D^{m-n},S^{m-n-1})$-fibration
$p \co (D(\nu_X),S(\nu_X)) \ra X$ and a homotopy equivalence of
pairs $i \co (N,\del N) \ra (D(\nu_X),S(\nu_X))$ such that the
following diagram commutes
\[
 \xymatrix{
(N,\del N) \ar[rr]^{i} \ar[dr]_{r} & & (D(\nu_X),S(\nu_X)) \ar[dl]^{p} \\
& X & }
\]
The map $p$ is now an honest fibration. Recall from Definition \ref{defn:sigma-m} the complex $\Sigma^m$ and that we have an embedding $\overline X \subset \Sigma^m$. It follows that
\[
 |\Sigma^m / \overline X| \simeq N/ \del N \simeq \Th(\nu_X)
\]
\end{construction}

\begin{construction} \label{con:geom-normal-signature-over-X}
Now we would like to present an analogue of Construction \ref{con:sym-construction-over-cplxs-lower-star-cplx} for normal complexes. What we are aiming for is an assignment
\[
 \gnsign_X (X) \co \sigma \; \mapsto \; (X(\sigma),\nu(\sigma),\rho(\sigma)) \in \big(\bOmega^N_{n-m}\big)^{(m-|\sigma|)} \; \textup{for each} \; \sigma \in X.
\]

The first two entries are defined as follows:
\[
 X(\sigma) = |D(\sigma,X)| \quad \nu(\sigma) = \nu_X \circ \textup{incl} \co X(\sigma) \subset X \ra \BSG (m-n-1)
\]

To define $\rho (\sigma)$ consider the following commutative diagram
\[
 \xymatrix{
 N(\sigma) \ar[r] \ar[ddr] \ar@{-->}[dr]^{\rho(N(\sigma))} & N \ar[dr]^{i} & \\
 & D(\nu(\sigma)) \ar[r] \ar[d] & D(\nu_X) \ar[d] \\
 & |D(\sigma,X)| \ar[r] & |X|
}
\]

In general the homotopy fibers of the projections $r(\sigma) \co
\del N(\sigma) \ra |D(\sigma,X)|$ are not $S^{k-1}$. On the other
hand the pullback of $\nu_X$ along the inclusion $D(\sigma,X)
\subset X$ yields an $S^{m-n-1}$-fibration $\nu(\sigma)$. The
associated disc fibration is a pullback as indicated by the diagram.
Since the two compositions $N(\sigma) \ra |X|$ commute we obtain the
dashed map.

Recall the $(m-|\sigma|)$-dimensional simplex $\sigma^\ast \in
\Sigma^m$. Observe that we have
\[
 \Delta^{m-|\sigma|} \cong |\sigma^\ast| = N(\sigma) \cup \overline N(\sigma)
\]
where $\overline N (\sigma) = \overline N \cap (|D(\sigma,X)| \ast |\overline X|)$.
Define the map
\[
\rho (\sigma) = \rho(N(\sigma)) \cup \rho (\overline{N} (\sigma))
\co \Delta^{m-|\sigma|} \cong N(\sigma) \cup \overline N(\sigma) \ra
D(\nu(\sigma)) \cup \{ \ast \} \cong \thom{\nu(\sigma)}
\]
where the map $\rho (\overline{N} (\sigma))$ is the collapse map.
\end{construction}


\begin{defn} \label{defn:geom-normal-signature-over-X}
Let $X$ be an $n$-dimensional finite Poincar\'e simplicial complex with the associated $n$-dimensional GNC $(X,\nu_X,\rho_X)$. Then the assignment from Construction \ref{con:geom-normal-signature-over-X} defines an element
\[
 \gnsign_X (X) \in H_n (X;\bOmega^N_\bullet)
\]
and is called the \emph{geometric normal signature of $X$ over $X$}.
\end{defn}

\begin{rem}
The assembly of Remark \ref{rem:hlgical-assembly} satisfies
\[
 A(\gnsign_X (X)) = (X,\nu_X,\rho_X) \in \Omega^N_n.
\]
\end{rem}

\begin{defn} \label{lem:normal-signature-over-X}
Let $X$ be an $n$-dimensional finite Poincar\'e simplicial complex with the associated $n$-dimensional GNC $(X,\nu_X,\rho_X)$. The composition of the geometric normal signature over $X$ from Definition \ref{defn:geom-normal-signature-over-X} with the normal signature map on the level of spectra from Proposition \ref{prop:connective-signatures-on-spectra-level} produces a well-defined element
\[
 \nsign_X (X) = \nsign \circ \gnsign_X (X) \in H_n (X;\bNL^\bullet \langle 1/2 \rangle) \quad
\]
called the \emph{normal signature of $X$ over $X$}.
\end{defn}

\begin{rem}
We have
\[
 A(\nsign_X (X)) = \nsign_{\Zpi} (X) \in NL^n (\Zpi)
\]
Recall that an $n$-dimensional NAC has an underlying symmetric complex. In the case of $\nsign_X (X)$ this is the complex obtained in Construction \ref{con:sym-construction-over-cplxs-lower-star-cplx}. It is not locally Poincar\'e, that means, it does not give a symmetric complex in $\Lambda (\ZZ)_\ast (X)$. Its assembly is $\ssign_{\Zpi} (X) \in L^n (\Zpi)$.
\end{rem}

\begin{defn} \label{prop:visible-signature-over-X}
Let $X$ be an $n$-dimensional finite Poincar\'e simplicial complex with the associated $n$-dimensional GNC $(X,\nu_X,\rho_X)$. The assembly of the normal complex over $X$ which defines the normal signature over $X$ is Poincar\'e and hence produces an $n$-dimensional NAC in the algebraic bordism category $\Lambda (\ZZ)(X)$ from Definition \ref{defn:Lambda-K-category} and as such a well-defined element
\[
 \vsign_X (X) \in VL^n (X)
\]
called the \emph{visible symmetric signature of $X$ over $X$}.
\end{defn}


Now we present another related construction. It will not be needed for the definition of the total surgery obstruction, but it will be used in the proof of the main theorem. (See Theorem \ref{thm:lifts-vs-orientations}.)

Let $(f,b) \co M \ra X$ be a degree one normal map of $n$-dimensional topological manifolds such that $X$ is triangulated. As discussed in Example \ref{expl:normal-symm-poincare-pair-gives-quadratic} the pair $(W,M \sqcup X)$, where $W$ is the mapping cylinder of $f$ possesses a structure of an $(n+1)$-dimensional geometric (normal, topological manifold) pair: the spherical fibration denoted by $\nu (b)$ is obtained as the mapping cylinder of $b \co \nu_M \ra \nu_X$ and the required map as the composition $\rho (b) \co D^{n+k+1} \ra S^{n+k} \times [0,1] \ra \Th (\nu(b))$. Another way of looking at the pair $(W,M \sqcup X)$ is to say that it is an $n$-simplex in the space $\Sigma^{-1} \bOmega^{N,\STOP}_0$. Hence via the relative normal construction we can associate to $(f,b)$ an $(n+1)$-dimensional (normal,symmetric Poincar\'e) algebraic pair
\[
 \nssign (f,b) = (\nsign (W),\ssign (M) - \ssign (X)) \in \pi_n (\bF)
\]
where $\bF := \textup{Fiber} \; \bL^\bullet \langle 0 \rangle \ra \bNL^\bullet \langle 1/2 \rangle$.

We would like to associate to $(f,b)$, respectively $(W,M \sqcup X)$ an $(n+1)$-dimensional (normal,symmetric Poincar\'e) algebraic pair over $\ZZ_\ast (X)$. This is not exactly a relative version of the previous definitions since the pair is not Poincar\'e. Nevertheless, in this special case we are able to obtain what we want.

\begin{con} \label{con:normal-symmetric-signature-over-X-for-a-deg-one-normal-map}
Let $(f,b) \co M \ra X$ be a degree one normal map of $n$-dimensional topological manifolds such that $X$ is triangulated. We can assume that $f$ is transverse to the dual cell decomposition of $X$. Consider the dissection
\[
 X = \bigcup_{\sigma \in X} X(\sigma) \quad (f,b) = \bigcup_{\sigma \in X} (f(\sigma),b(\sigma)) \co M (\sigma) \ra X (\sigma)
\]
where each $(f(\sigma),b(\sigma))$ is a degree one normal map of $(n-|\sigma|)$-dimensional manifolds $(m-|\sigma|)$-ads. We obtain an assignment which to each $\sigma \in X$ associates an $(n+1-|\sigma|)$-dimensional pair of normal $(m-|\sigma|)$-ads
\[
 \sigma \mapsto ((W(\sigma),\nu(b(\sigma)),\rho(b(\sigma))),M(\sigma) \sqcup X(\sigma)).
\]
These fit together to produce an $\bOmega^N_\bullet$-cobordism of $\bOmega^{\STOP}_\bullet$-cycles in the sense of Definition \ref{defn:E-cycles}, or equivalently a $\Sigma^{-1} \bOmega^{N,\STOP}_\bullet$-cycle, providing us with an element
\[
\sign_X^{\G/\TOP} (f,b) \in H_n (X ; \Sigma^{-1} \bOmega^{N,\STOP}_\bullet) 
\]
Composing with the normal signature map $\nsign \co \bOmega^N_\bullet \ra \bNL^\bullet \langle 1/2 \rangle$ then produces a $\bNL^\bullet \langle 1/2 \rangle$-cobordism, which can be seen as an $(n+1)$-dimensional (normal,symmetric Poincar\'e) pair over $\ZZ_\ast (X)$
\[
 \nssign_X (f,b) = \nssign (\sign_X^{\G/\TOP} (f,b)) \in H_n (X ; \bF).
\]
By applying the homological assembly of Remark \ref{rem:hlgical-assembly} we obtain the $(n+1)$-dimensional $($normal,symmetric Poincar\'e$)$ pair
\[
 \nssign (f,b) \in \pi_n (\bF).
\]
\end{con}

\begin{rem}
Recall from Example \ref{expl:normal-symm-poincare-pair-gives-quadratic} the correspondence
\[
 (\nsign (W),\ssign (M) - \ssign (X)) \longleftrightarrow \qsign (f,b)
\]
where $\qsign (f,b) \in L_n (\ZZ)$ is the quadratic signature (=surgery obstruction) of the degree one normal map $(f,b)$. Using the relative version of Example \ref{expl:normal-symm-poincare-pair-gives-quadratic} we obtain in this situation an identification of $\nssign_X (f,b)$ with the a quadratic signature of Construction \ref{con:quad-construction-over-cplxs-lower-star}
\[
 \nssign _X (f,b) = \qsign_X (f,b) \in H_n (X ; \bL_\bullet \langle 1 \rangle).
\]
\end{rem}

\section{Definition of $s(X)$} \label{sec:defn-of-s(X)}

\begin{defn} \cite[17.1]{Ranicki(1992)} Define
\begin{equation}
 s(X) : = \partial \big( \vsign_X (X) \big) \in \SS_n (X).
\end{equation}
\end{defn}

Let us have a close look at the $(n-1)$-dimensional QAC $(C,\psi)$ in the category $\Lambda (\ZZ) \langle 1 \rangle_\ast (X)$ representing $s(X)$. By definition of $\vsign_X (X)$ and of the map $\del \co VL^n (X) \ra \SS_n (X)$ the subcomplex $C(\sigma)$ is the mapping cone of the duality map
\[
\varphi [X(\sigma)] \co \Sigma^n TC(\sigma) =
C^{n-|\sigma|} (D(\sigma)) \ra C(\sigma) = C(D(\sigma),\del
D(\sigma)).
\]
The quadratic structure $\psi (\sigma)$ is more subtle to describe. It corresponds to the normal structure on $X(\sigma)$ via Lemma \ref{lem:normal-gives-quadratic-boundary}.

We clearly see that if $X$ is a manifold then the mapping cones of the maps $\varphi [X(\sigma)]$ above are contractible and the total surgery obstruction equals $0$. If $X$ is homotopy equivalent to a manifold then the mapping cylinder of the homotopy equivalence provides via the constructions in Construction \ref{con:normal-symmetric-signature-over-X-for-a-deg-one-normal-map} a cobordism from $(C,\psi)$ to $0$.

\section{Proof of the Main Technical Theorem (I)} \label{sec:proof-part-1}


Recall the statement. For an $n$-dimensional finite Poincar\'e complex $X$ with $n \geq 5$ let $t(X)$ be the image of $s(X)$ under the map $\SS_n (X) \ra H_{n-1} (X;\bL_\bullet \langle 1 \rangle)$. Then $t(x) = 0$ if and only if there exists a topological block bundle reduction of the SNF $\nu_X$. The main idea of the proof is to translate the statement about the reduction of $\nu_X$ into a statement about orientations with respect to $L$-theory spectra. The principal references for this section are \cite[pages 280-292]{Ranicki(1979)} and \cite[section 16]{Ranicki(1992)}.

\subsection{Topological surgery theory} \label{subsec:top-surgery}

\

Before we start we offer some comments about the topological surgery and about the bundle theories used. The topological surgery is a modification of the surgery in the smooth and PL-category, due to Browder-Novikov-Sullivan-Wall as presented in \cite{Browder(1971)} and \cite{Wall(1999)}, by the work of Kirby and Siebenmann as presented in \cite{Kirby-Siebenmann(1977)}. This book also discusses various bundle theories and transversality theorems for topological manifolds. From our point of view the notion of a ``stable normal bundle'' for topological manifolds is of prominent importance. As explained in Essay III, \S 1, the notion of a stable microbundle is appropriate and there exists a corresponding transversality theorem, whose dimension and codimension restrictions are removed by \cite[Chapter 9]{Freedman-Quinn(1990)}. It is also explained that when enough triangulations are in sight, one can use block bundles and the stable microbundle transversality can be replaced by block transversality. This is thanks to the fact that for the classifying spaces we have $\BSTOP \simeq \BSbTOP$. Since for our problem we can suppose that the Poincar\'e complex $X$ is in fact a simplicial complex we can ask about the reduction of the SNF to a stable topological block bundle. When we talk about the degree one normal maps $(f,b) \co M \ra X$ we mean the stable microbundle normal data, since we need to work in full generality.

\subsection{Orientations} \label{subsec:orientations}

\

Let $\bE$ be a ring spectrum. An \emph{$\bE$-orientation} of a $\ZZ$-oriented spherical fibration $\nu \co X \ra \BSG(k)$ is an element of $u_{\bE} (\nu) \in H^k (\Th (\nu) ; \bE)$ that means a homotopy class of maps $u_{\bE} (\nu) \co \bT(\nu) \ra \bE$, where $\bT(\nu)$ denotes the Thom spectrum of $\nu$, such that for each $x \in X$, the restriction $u_{\bE} (\nu)_x \co \bT(\nu_x) \ra \bE$
to the fiber $\nu_x$ of $\nu$ over $x$ represents a generator of $\bE^\ast (\bT(\nu_x)) \cong \bE^\ast (S^k)$ which under the Hurewicz homomorphism $\bE^\ast (\bT(\nu_x)) \ra H^\ast (\bT(\nu_X);\ZZ)$ maps to the chosen $\ZZ$-orientation.


\subsection{Canonical orientations} \label{subsec:canonical-orientations}

\

Denote by $\bMSG$ the Thom spectrum of the universal stable
$\ZZ$-oriented spherical fibrations over the classifying space
$\BSG$. Its $k$-th space is the Thom space $\bMSG (k) =
\thom{\gamma_{\SG} (k)}$ of the canonical $k$-dimensional spherical
fibration $\gamma_{\SG} (k)$ over $\BSG (k)$. Similarly denote by
$\bMSTOP$ the Thom spectrum of the universal stable $\ZZ$-oriented
topological block bundles over the classifying space $\BSTOP \simeq
\BSbTOP$. Its $k$-th space is the Thom space $\bMSTOP (k) =
\thom{\gamma_{\SbTOP} (k)}$ of the canonical $k$-dimensional block
bundle $\gamma_{\SbTOP} (k)$ over $\BSbTOP (k)$. There is a map $J \co \bMSTOP \ra \bMSG$ defined by viewing the canonical block bundle $\gamma_{\SbTOP} (k)$ as a spherical fibration. 

Both $\bMSG$ and $\bMSTOP$ are ring spectra. The multiplication on $\bMSTOP$ is given by the Cartesian product of block bundles. The multiplication on $\bMSG$ is given by the sequence of the operations: take the associated disk fibrations, form the product disk fibration and take the associated spherical fibration. Upon precomposition with the diagonal map the multiplication on $\bMSTOP$ becomes the Whitney sum and the multiplication on $\bMSG$ becomes fiberwise join. The map $J \co \bMSTOP \ra \bMSG$ is a map of ring spectra.


\begin{prop} \textup{\cite[pages 280-283]{Ranicki(1979)}} \label{canonical-geom-orientations} \
\begin{enumerate}
 \item Any $k$-dimensional $\ZZ$-oriented spherical fibration $\alpha \co X \ra \BSG (k)$ has a canonical orientation $u_{\bMSG} (\alpha) \in H^k (\Th(\alpha);\bMSG)$.
 \item Any $k$-dimensional  $\ZZ$-oriented topological block bundle $\beta \co X \ra \BSbTOP (k)$ has a canonical orientation $u_{\bMSTOP} (\beta) \in H^k (\Th(\beta);\bMSTOP)$.
\end{enumerate}
Moreover $J (u_{\bMSTOP} (\beta)) = u_{\bMSG} (J (\beta))$.
\end{prop}

This follows since any spherical fibration (or a topological block
bundle) is a pullback of the universal via the classifying map.

\subsection{Transversality} \label{subsec:transversality}

\

By transversality one often describes statements which assert that a
map from a manifold to some space with a closed subspace can be
deformed by a small homotopy to a map such that the inverse image of
the closed subspace is a submanifold. Such notion of transversality
can then be used to prove various versions of the Pontrjagin-Thom
isomorphism. For example topological transversality of
Kirby-Siebenmann \cite[Essay III]{Kirby-Siebenmann(1977)} and Freedman-Quinn \cite[chapter 9]{Freedman-Quinn(1990)} implies that the classifying map induces
\begin{equation} \label{htpy-eq-top-transversality}
 c \co \bOmega_\bullet^{STOP} \simeq \bMSTOP.
\end{equation}

On the other hand normal transversality used here has a different
meaning, no statement invoking preimages is
required.\footnote{Although there are some such statements
\cite{Hausmann-Vogel(1993)}, we will not need them.} It just means
that there is the homotopy equivalence
(\ref{htpy-eq-normal-transversality}) below inducing a
Pontrjagin-Thom isomorphism. To arrive at it one can use the ideas
described in \cite[Errata]{Ranicki(1992)}. Recall the spectrum
$\bOmega_\bullet^N$ from section \ref{sec:spectra}. Further recall
for a space $X$ with a $k$-dimensional spherical fibration $\nu \co
X \ra \BSG (k)$ the space $\bOmega_n^N (X,\nu)$ of normal spaces
with a degree one normal map to $(X,\nu)$. The normal transversality
described in \cite[Errata]{Ranicki(1992)} says that the classifying map induces
\[
 c \co \bOmega_0^N (X,\nu) \simeq \thom{\nu}
\]
We have the classifying space $\BSG (k)$ with the canonical
$k$-dimensional spherical fibration $\gamma_\SG (k)$. The spectrum
$\bOmega_\bullet^N$ can be seen as the colimit of spectra
$\bOmega_\bullet^N (\BSG(k),\gamma_\SG (k))$. The normal
transversality from \cite[Errata]{Ranicki(1992)} translates into
homotopy equivalence
\begin{equation} \label{htpy-eq-normal-transversality}
 \bOmega_\bullet^N \simeq \bMSG.
\end{equation}

There are multiplication operations on the spectra $\bOmega^\STOP_\bullet$ and $\bOmega^N_\bullet$, which make the above Pontrjagin-Thom maps to ring spectra homotopy equivalences. These operations are given by Cartesian products. However, we will not use this point of view later.

To complete the picture we denote
\begin{equation} \label{eqn:defn-of-MSGTOP}
 \bMSGTOP := \textup{Fiber} \; (\bMSTOP \ra \bMSG)
\end{equation}
and observe that the above classifying maps induce yet another Pontrjagin-Thom isomorphism 
\begin{equation} \label{eqn:normal-topological-transversality}
 \Sigma^{-1} \bOmega^{N,\STOP}_\bullet \simeq \bMSGTOP. 
\end{equation}
Furthermore we have that $\bMSGTOP$ is a module spectrum over $\bMSTOP$ and similarly $\Sigma^{-1} \bOmega^{N,\STOP}_\bullet$ is a module spectrum over $\bOmega^{\STOP}_\bullet$.


\subsection{$L$-theory orientations} \label{subsec:L-theory-orientations}


Here we use the signature maps between the spectra from section \ref{sec:conn-versions} to construct orientations with respect to the $L$-theory spectra. We recall that $\bNL^\bullet \langle 1/2\rangle $ and $\bL^\bullet \langle 0 \rangle$ are ring spectra with the multiplication given by the products of \cite[section 8]{Ranicki-I-(1980)} and \cite[Appendix B]{Ranicki(1992)}. The spectrum $\bL_\bullet \langle 1 \rangle$ is a module over $\bL^\bullet \langle 0 \rangle$ again by the products of \cite[section 8]{Ranicki-I-(1980)} and \cite[Appendix B]{Ranicki(1992)}.

\begin{prop} \textup{\cite[pages 284-289]{Ranicki(1979)}} \label{canonical-L-orientations} \
\begin{enumerate}
\item Any $k$-dimensional $\ZZ$-oriented spherical fibration $\alpha \co X \ra \BSG (k)$ has a canonical orientation $u_{\bNL^\bullet}(\alpha) \in H^k (\Th(\alpha);\bNL^\bullet \langle 1/2 \rangle)$.
\item Any $k$-dimensional $\ZZ$-oriented topological block bundle $\beta \co X \ra \textup{BS}\bTOP (k)$ has a canonical orientation $u_{\bL^\bullet}( \beta) \in H^k (\Th(\beta);\bL^\bullet \langle 0 \rangle)$.
\end{enumerate}
Moreover $J (u_{\bL^\bullet} (\beta)) = u_{\bNL^\bullet} (J
(\beta))$.
\end{prop}

\begin{proof}
These orientations are obtained from maps between spectra using the following up to homotopy commutative diagram of spectra:
\[
 \xymatrix{
\bMSTOP \ar[r] \ar[d] & \bOmega_\bullet^{STOP} \ar[r]^-{\ssign} \ar[d] & \bL^\bullet \langle 0 \rangle \ar[d] \\
\bMSG \ar[r] & \bOmega_\bullet^N \ar[r]_-{\nsign} & \bNL^\bullet
\langle 1/2 \rangle }
\]
where the maps in the left hand part of the diagram are the homotopy
inverses of the transversality homotopy equivalences.
\end{proof}


\subsection{$S$-duality} \label{subsec:S-duality}


\

If $X$ is a Poincar\'e complex with the SNF $\nu_X \co X \ra \BSG
(k)$ then we have the $S$-duality $\thom{\nu_X}^\ast \simeq X_+$
producing isomorphisms
\begin{align*}
 S \co H^k (\Th(\nu_X);\bNL^\bullet \langle 1/2 \rangle) & \cong H_n (X;\bNL^\bullet \langle 1/2 \rangle) \\
 S \co H^k (\Th(\nu_X);\bL^\bullet \langle 0 \rangle) & \cong H_n (X;\bL^\bullet \langle 0 \rangle).
\end{align*}

The following proposition describes a relation between the signatures over $X$ in homology, obtained in sections \ref{sec:gen-hlgy-thies} and \ref{sec:normal-signatures-over-X} and orientations in cohomology from this section.

\begin{prop} \textup{\cite[Proposition 16.1.]{Ranicki(1992)}} \label{prop:S-duals-of-orientations-are-signatures}
If $X$ is an $n$-dimensional geometric Poincar\'e complex with the
Spivak normal fibration $\nu_X \co X \ra \BSG(k)$ then we have
\[
S (u_{\bNL^\bullet} (\nu_X)) = \nsign_X (X) \in H_n (X;\bNL^\bullet
\langle 1/2 \rangle).
\]
If $\bar \nu_X$ is a topological block bundle reduction of the SNF
of $X$ and $(f,b) \co M \ra X$ is the associated degree one normal
map, then we have
\[
S (u_{\bL^\bullet} (\bar \nu_X)) = \ssign_X (M) \in H_n (X;\bL^\bullet
\langle 0 \rangle).
\]
\end{prop}

\begin{proof}
The identification of the normal signature $\nsign_X (X)$ as the canonical orientation follows from the commutative diagram of simplicial sets
\[
\xymatrix{
\Sigma^m/\overline X \ar[d]_{i} \ar[rr]^{\gnsign_X (X)} & & \Omega_{-k}^N \ar[d]^{c} \\
\textup{Sing} \; \Th(\nu_X) \ar[rr]_-{u_{\bMSG}(\nu_X)} & & \textup{Sing} \; \bMSG (k) }
\]
This in turn is seen by inspecting the definitions of the maps in the diagram. The upper horizontal map comes from \ref{con:geom-normal-signature-over-X}, the map $i$ from \ref{con:simpl-model-for-the-thom-space-of-snf} and the other two maps were defined in this section. Note that the classifying map $c$ is characterized by the property that the classified spherical fibration is obtained as the pullback of the canonical $\gamma_{\SG}$ along $c$. But this is also the characterization of the canonical orientation $u_{\bMSG} (\nu_X)$. The desired statement is obtained by composing with $\nsign \co \bOmega^N_\bullet \ra \bNL^\bullet \langle 1/2 \rangle$.

For the second part recall how the degree one normal map $(f,b)$ associated to $\bar \nu_X$ is constructed. Consider the composition $\Sigma^m \ra \Sigma^m / \overline X \ra \Th (\bar \nu_X)$. Since $\bar \nu_X$ is a stable topological block bundle this map can be made transverse to $X$ and $M$ is the preimage, $f$ is the restriction of the map to $M$ and it is covered by a map of stable microbundles $\nu_M \ra \bar \nu_X$, where $\nu_M$ is the stable normal microbundle of $M$. In addition this can be made in such a way that $f$ is transverse to the dual cells of $X$. Hence we obtain a dissection of $M$ which gives rise to the symmetric signature $\ssign_X (M)$ over $X$ as in Construction \ref{con:sym-construction-over-cplxs-lower-star-mfd}. It fits into the following diagram
\[
\xymatrix{
\Sigma^m/\overline X \ar[d]_{i} \ar[rr]^{\stopsign_X (M)} & & \Omega_{-k}^{\STOP} \ar[d]^{c} \\
\textup{Sing} \; \Th(\bar \nu_X) \ar[rr]_-{u_{\bMSTOP}(\bar \nu_X)} & & \textup{Sing} \; \bMSTOP (k) }
\]
The desired statement is obtained by composing with $\ssign \co \bOmega^{\STOP}_\bullet \ra \bL^\bullet \langle 0 \rangle$.
\end{proof}

Suppose now that we are given a degree one normal map $(f,b) \co M \ra X$ between $n$-dimensional topological manifolds with $X$ triangulated. In Construction \ref{con:normal-symmetric-signature-over-X-for-a-deg-one-normal-map} we defined the (normal,symmetric Poincar\'e) signature $\nssign_X (f,b)$ over $X$ associated to $(f,b)$. In analogy with the previous proposition we would like to interpret this signature as an orientation via the $S$-duality. For this recall first that specifying the degree one normal map $(f,b)$ is equivalent to specifying a pair $(\nu,h)$, with $\nu \co X \ra \BSTOP$ and $h \co J(\nu) \simeq \nu_X$, where in our situation the SNF $\nu_X$ has a preferred topological block bundle lift, also denoted $\nu_X$, coming from the stable normal bundle of $X$ (see subsection \ref{subsec:normal-invariants-revisited} if needed). The homotopy $h$ gives us a spherical fibration over $X \times I$ with the canonical orientation $u^{\bMSG} (h)$ which we view as a homotopy between the orientations $J(u^{\bMSTOP} (\nu))$ and $J(u^{\bMSTOP} (\nu_X))$. In this way we obtain an element
\[ 
u^{\G/\TOP} (\nu,h) \in H^k (\Th (\nu_X); \bMSGTOP) 
\]
given by 
\[
 u^{\G/\TOP} (\nu,h) = (u^{\bMSG} (h) , u^{\bMSTOP} (\nu) - u^{\bMSTOP} (\nu_X)).
\]
The Pontrjagin-Thom isomorphism (\ref{eqn:normal-topological-transversality}) together with the normal and symmetric signature $\nssign \co \Sigma^{-1} \bOmega^{N,\STOP}_\bullet \ra \bF$ provide us with the pair 
\[
u^{\bNL^\bullet,\bL^\bullet} (\nu,h) = (u^{\bNL^\bullet} (h),u^{\bL^\bullet} (\nu) - u^{\bL^\bullet} (\nu_X)) \in H^k (\Th (\nu_X) ; \bF). 
\]

\begin{prop} \label{prop:S-duals-of-orientations-are-signatures-relative-case}
Let $(f,b) \co M \ra X$ be a degree one normal map of $n$-dimensional simply-connected topological manifolds with $X$ triangulated, corresponding to the pair $(\nu,h)$, where $\nu \co X \ra \BSTOP$ and $h \co J(\nu) \simeq \nu_X$. Then we have
\[
 S (u^{\bNL^\bullet,\bL^\bullet} (\nu,h)) = \nssign_X (f,b) \in H_n (X ; \bF).
\]
\end{prop}

\begin{proof}
The proof is analogous to the proof of Proposition \ref{prop:S-duals-of-orientations-are-signatures}. Recall that the signature $\sign^{\G/\TOP}_X (f,b)$  is constructed using a dissection of the degree one normal map $(f,b)$. Using this dissection we inspect that we have a commutative diagram
\[
\xymatrix{
\Sigma^m/\overline X \ar[d]_{i} \ar[rr]^{\sign^{\G/\TOP}_X (f,b)} & & \Sigma^{-1} \Omega_{-k}^{N,\STOP} \ar[d]^{c} \\
\textup{Sing} \; F(\nu,\nu_X) \ar[rr]_-{u^{\G/\TOP} (\nu,h)} & & \textup{Sing} \; \bMSGTOP (k) }
\]
where we use the notation $\bMSGTOP (k) := \textup{Fiber} \; (\bMSTOP (k) \ra \bMSG (k))$ and $F(\nu,\nu_X) : = \textup{Pullback} \; (\Th (\nu) \ra \Th (\nu_X) \leftarrow \Th (\nu_X))$. Composing with the signature map $\nssign \co \Sigma^{-1} \Omega_{-k}^{N,\STOP} \ra \bF$ proves the claim. 
\end{proof}

\subsection{Assembly}  \label{subsec:assembly}

\

Keep $X$ a Poincar\'e complex with the SNF $\nu_X$ and suppose there
exists a topological block bundle reduction $\bar \nu_X$. Recall
that orientations with respect to ring spectra induce Thom
isomorphisms in corresponding cohomology theories. Hence we have Thom
isomorphisms induced by $u_{\bNL^\bullet} (\nu_X)$ and
$u_{\bL^\bullet} (\bar \nu_X)$ and these are compatible. Also recall
that although the spectrum $\bL_\bullet \langle  1 \rangle$ is not a
ring spectrum, it is a module spectrum over $\bL^\bullet \langle 0
\rangle$, see \cite[Appendix B]{Ranicki(1992)}. Therefore
$u_{\bL^\bullet} (\bar \nu_X)$ also induces a compatible Thom
isomorphism in $\bL_\bullet \langle  1 \rangle$-cohomology. In fact
we have a commutative diagram relating these Thom isomorphisms, the
S-duality and the assembly maps:

{\footnotesize
\[
\xymatrix{
H^0 (X;\bL_\bullet \langle 1 \rangle) \ar[r]^-{\cong} \ar[d] & H^k (\Th(\nu_X);\bL_\bullet \langle 1 \rangle) \ar[r]^-{\cong} \ar[d] & H_n (X;\bL_\bullet \langle 1 \rangle) \ar[r]^-{A} \ar[d] & L_n (\ZZ) \ar[d] \\
H^0 (X;\bL^\bullet \langle 0 \rangle) \ar[r]^-{\cong} \ar[d] & H^k (\Th(\nu_X);\bL^\bullet \langle 0 \rangle) \ar[r]^-{\cong} \ar[d] & H_n (X;\bL^\bullet \langle 0 \rangle) \ar[r]^-{A} \ar[d] & L^n (\ZZ) \ar[d] \\
H^0 (X;\bNL^\bullet \langle 1/2 \rangle) \ar[r]^-{\cong} & H^k (\Th(\nu_X);\bNL^\bullet \langle 1/2 \rangle) \ar[r]^-{\cong} & H_n (X;\bNL^\bullet \langle 1/2 \rangle) \ar[r]^-{A} & NL^n (\ZZ) \\
}
\]
}

If $X = S^n$ then the map $A \co H_n (S^n;\bL_\bullet \langle 1 \rangle) \ra L_n (\ZZ)$ is an isomorphism. This follows from the identification of the assembly map with the surgery obstruction map, which is presented in Proposition \ref{prop:identification} and the fact that the surgery obstruction map for $S^n$ is an isomorphism due to Kirby and Siebenmann \cite[Essay V, Theorem C.1]{Kirby-Siebenmann(1977)}. We note that Proposition \ref{prop:identification} is presented in a greater generality than needed here. We just need the case when $X = S^n$ and so it is a manifold and hence the degree one normal map $(f_0,b_0)$ in the statement of Proposition \ref{prop:identification} can be taken to be the identity on $S^n$, which is the version we need at this place. 

Further observe that all the homomorphisms in the diagram are induced homomorphisms on homotopy groups by maps of spaces, see definitions in sections \ref{sec:spectra}, \ref{sec:gen-hlgy-thies} for the underlying spaces.


\subsection{Classifying spaces for spherical fibrations with an orientation} \label{subsec:classifying-spaces}


\

Let $\bE$ be an Omega ring spectrum with $\pi_0 (\bE) = \ZZ$ and
recall the notion of an $\bE$-orientation of a $\ZZ$-oriented
spherical fibration $\alpha \co X \ra \BSG (k)$. In \cite{May(1977)}
a construction of a classifying space $\BEG$ for spherical
fibrations with such a structure was given. The construction is not
so important for us. Of more significance is a description of what
it means to have a map from a space to one of these classifying
spaces. If $X$ is a finite complex then there is a one-to-one
correspondence between homotopy classes of maps $\alpha_\bE \co X
\ra \textup{B}\bE\textup{G}$ and homotopy classes of pairs
$(\alpha,u_\bE(\alpha))$ where $\alpha \co X \ra \BSG (k)$ and
$u_\bE (\alpha) \co \Th(\alpha) \ra \bE_{k}$ is an $\bE$-orientation
of $\alpha$.

\begin{prop} \label{prop:canonical-L-orientations-on-class-spaces}
There is a commutative diagram
\[
\xymatrix{
\BSTOP \ar[r]^-{\ssign} \ar[d]_{J} & \textup{B} \bL^\bullet \langle 0 \rangle \textup{G} \ar[d]^{J} \\
\BSG \ar[r]_-{\nsign} & \textup{B} \bNL^\bullet \langle 1/2 \rangle
\textup{G} }
\]
\end{prop}

\begin{proof}
This follows from Proposition \ref{canonical-L-orientations}.
\end{proof}

We need to study what orientations do there exist for a fixed
spherical fibration $\alpha \co X \ra \BSG (k)$. Denote by
$\bE_\otimes$ the component of $1 \in \ZZ$ in any Omega ring
spectrum $\bE$ with $\pi_0 (\bE) = \ZZ$. Then there is a homotopy
fibration sequence \cite[section III.2]{May(1977)}
\begin{equation} \label{eqn:fibration-sequence-for-E-orientations}
 \bE_\otimes \xra{i} \textup{B}\bE\textup{G} \ra \BSG
\end{equation}

The map $i$ can be interpreted via the Thom isomorphism. Let $c \co
X \ra \bE_\otimes$ be a map. Then $i(c) \co X \ra \BEG$ is the map
given by the trivial fibration $\varepsilon \co X \ra \BSG (k)$ with
an $\bE$-orientation given by the composition
\[
u_\bE (i(c)) \co \Th(\varepsilon) \xra{\; \tilde \Delta \;} X_+ \wedge
\Th(\varepsilon) \xra{c \wedge \Sigma^k (1)} \bE_\otimes \wedge \bE_k
\ra \bE_k
\]

We will use the spectra $\bL^\bullet \langle 0 \rangle$ and
$\bNL^\bullet \langle 1/2 \rangle$, which are both ring spectra with
$\pi_0 \cong \ZZ$. We will need the following proposition.

\begin{prop} \label{prop:fibration-sequence-of-classifying-spaces}
There is the following homotopy fibration sequence of spaces
\[
\bL_0 \langle 1 \rangle \ra \textup{B} \bL^\bullet \langle 0 \rangle
\textup{G} \ra \textup{B} \bNL^\bullet \langle 1/2 \rangle
\textup{G}
\]
\end{prop}

\begin{proof}
Consider the sequences
(\ref{eqn:fibration-sequence-for-E-orientations}) for the spectra
$\bL^\bullet \langle 0 \rangle$ and $\bNL^\bullet \langle 1/2
\rangle$ and the map between them. The induced map between the
homotopy fibers fits into the fibration sequence
\begin{equation} \label{fib-seq:quad-sym-norm-on-component-of-1}
\bL_0 \langle 1 \rangle \ra \bL^\otimes \langle 0 \rangle \ra \bNL^\otimes \langle 1/2 \rangle
\end{equation}
which is obtained from the fibration sequence of Proposition \ref{prop:fib-seq-of-quad-sym-norm-connective-version} (more precisely from the space-level version of it on the $0$-th spaces) by replacing the symmetrization map $(1 + T) \co \bL_0 \langle 1 \rangle \ra \bL^0 \langle 0 \rangle$ by the map given on the $l$-simplices as
\begin{align*}
 (1 + T)^\otimes \co  \bL_0 \langle 1 \rangle & \ra \bL^\otimes \langle 0 \rangle \\
 (C,\psi) & \mapsto (1+T) (C,\psi) + (C (\Delta^l),\varphi([\Delta^l])).
\end{align*}
Its effect is to map the component of $0$ (which is the only component of $\bL_0 \langle 1 \rangle$) to the component of $1$ in $\bL^0 \langle 0 \rangle$ instead of the component of $0$. The proposition follows.
\end{proof}


\subsection{$L$-theory orientations versus reductions}   \label{subsec:main-thm}


\

The following theorem is a crucial result.

\begin{thm} \textup{\cite[pages 290-292]{Ranicki(1979)}}   \label{thm:lifts-vs-orientations}
There is a one-to-one correspondence between the isomorphism classes of
\begin{enumerate}
 \item stable oriented topological block bundles over $X$, and
 \item stable oriented spherical fibrations over $X$ with an $\bL^\bullet \langle 0 \rangle$-lift of the canonical $\bNL^\bullet \langle 1/2 \rangle$-orientation
\end{enumerate}
\end{thm}

\begin{proof}
In Proposition \ref{prop:canonical-L-orientations-on-class-spaces} a map from (1) to (2) was described. To prove that it gives a one-to-one correspondence is equivalent to showing that the square in Proposition \ref{prop:canonical-L-orientations-on-class-spaces} is a homotopy pullback square. This is done by showing that the induced
map between the homotopy fibers of the vertical maps in the square, which is indicated by the dashed arrow in the diagram below, is a homotopy equivalence.
\[
\xymatrix{
\G/\TOP \ar@{-->}[r]^-{\nssign} \ar[d] & \bL_0 \langle 1 \rangle \ar[d] \\
\BSTOP \ar[r]^-{\ssign} \ar[d]_{J} & \textup{B} \bL^\bullet \langle 0 \rangle \textup{G} \ar[d]^{J} \\
\BSG \ar[r]_-{\nsign} & \textup{B} \bNL^\bullet \langle 1/2 \rangle
\textup{G} }
\]
For this it is enough to show that it induces an isomorphism on the
homotopy groups, that means to show
\[
 \nssign \; \co \; [S^n ; \G/\TOP] \; \xrightarrow{\; \cong \;} \; [S^n ; \bL_0 \langle 1 \rangle].
\]
Recall that since $S^n$ is a topological manifold with the trivial
SNF we have a canonical identification of the normal invariants
\[
 [S^n ; \G/\TOP] \cong \sN (S^n) \quad ((\alpha,H) \co S^n \ra \G/\TOP) \mapsto( (f,b) \co M \ra S^n)
\]
where $\alpha \co S^n \ra \BSTOP$ and $H = J(\alpha) \simeq
\varepsilon \co S^n \times [0,1] \ra \BSG$ is a homotopy to a
constant map and $(f,b)$ is the associated degree one normal map.

On the other side we have
\[
A(S(i(-))) \co [S^n ; \bL_0 \langle 1 \rangle] \cong H^0
(S^n;\bL_\bullet \langle 1 \rangle) \cong H_n (S^n ; \bL_\bullet
\langle 1 \rangle) \cong L_n (\ZZ).
\]

It is well known that the surgery obstruction map $\qsign_\ZZ \co \sN (S^n) \ra L_n (\ZZ)$ from Definition \ref{defn:quad-sign} is an isomorphism for $n \geq 1$ \cite[Essay V, Theorem C.1]{Kirby-Siebenmann(1977)}. Therefore it is enough to show that
\[
 A(S(i(\nssign_\ZZ (\alpha,H)))) = \qsign_\ZZ (f,b).
\]
Denote
\[
 (J(\alpha),u_{\bL^\bullet} (\alpha)) = \ssign (\alpha) \quad (H,u_{\bNL^\bullet} (H)) = \nsign (H).
\]
Now we need to describe in more detail what the identification of
the homotopy fiber of the right hand column map means. That means to
produce a map
\[
\bar u_{\bL_\bullet} (\alpha,H) \co S^n \ra \bL_0 \langle 1 \rangle.
\]
from  $(J(\alpha),u_{\bL^\bullet} (\alpha))$ and
$(H,u_{\bNL^\bullet}(H))$. The spherical fibration $J(\alpha)$ is trivial because of the null-homotopy $H$
and therefore we obtain a map $\bar u_{\bL^\bullet} (\alpha) \co S^n
\ra \bL^{\otimes} \langle 0 \rangle$ such that $i (\bar u_{\bL^\bullet} (\alpha)) =
(J(\alpha),u_{\bL^\bullet} (\alpha))$.

Similarly the homotopy $(H,u_{\bNL^\bullet}(H)) \co S^n \times [0,1]
\ra \BNLG$ yields a homotopy $\bar u_{\bNL^\bullet} (H) \co S^n \times
[0,1] \ra  \bNL^\otimes \langle 1/2 \rangle$ between $J(\bar u_{\bL^\bullet} (\alpha))$ and the constant map.

The pair $(\bar u_{\bNL^\bullet} (H),\bar u_{\bL^\bullet} (\alpha))$ produces via the homotopy fibration sequence (\ref{fib-seq:quad-sym-norm-on-component-of-1}) a lift, which is the desired $\bar u_{\bL_\bullet} (\alpha,H)$. So we have
\[
 [\nssign (\alpha,H)] = [\bar u_{\bL_\bullet} (\alpha,H)]  \in [S^n;\bL_0 \langle 1 \rangle]
\]
and we want to investigate $A(S(i(\bar u_{\bL_\bullet}
(\alpha,H))))$. Recall now the commutative diagram from subsection
\ref{subsec:assembly}. It shows that $A(S(i(\bar u_{\bL_\bullet}
(\alpha,H))))$ can be chased via the lower right part of the
diagram. Here we consider maps from $S^n$ and $S^n \times [0,1]$ to the underlying spaces in this diagram rather than just elements in the homotopy groups.

Observe first, using Definition \ref{defn:sym-sign-over-X}, Example \ref{expl:assembly-of-symmetric-signature-over-K} and Proposition \ref{prop:S-duals-of-orientations-are-signatures}, that the assembly of the $S$-dual of the class $u_{\bL^\bullet} (\alpha)$ is  an $n$-dimensional SAPC $\ssign (M)$ over $\ZZ$.

Secondly, by Construction \ref{con:normal-symmetric-signature-over-X-for-a-deg-one-normal-map} and Proposition \ref{prop:S-duals-of-orientations-are-signatures-relative-case}, the assembly of the $S$-dual of the class $u_{\bNL^\bullet} (H)$ is an $(n+1)$-dimensional (normal, symmetric Poincar\'e) pair 
\begin{equation} \label{eqn:normal-sym-pair-of-deg-one-normal-map-to-sphere}
(\nsign (W),\ssign (M) - \ssign (S^n))
\end{equation}
over $\ZZ$, with $W = \cyl (f)$. We consider this as an element in the $n$-th homotopy group of the relative term in the long exact sequence of the homotopy groups associated to the map $\bL^\otimes \langle 0 \rangle \ra \bNL^\otimes \langle 1/2 \rangle$. This group is isomorphic to $L_n (\ZZ)$ by the isomorphism of Proposition \ref{propn:LESL}. The effect of this isomorphism on an element as in (\ref{eqn:normal-sym-pair-of-deg-one-normal-map-to-sphere}) is then described in detail in Example \ref{expl:normal-symm-poincare-pair-gives-quadratic}. It tells us that the $n$-dimensional QAPC corresponding to (\ref{eqn:normal-sym-pair-of-deg-one-normal-map-to-sphere}) is the surgery obstruction $\qsign (f,b)$. Finally we obtain
\[
 A(S(i(\bar u_{\bL_\bullet} (\alpha,H)))) = \qsign (f,b) \in L_n (\ZZ)
\]
which is what we wanted to show.
\end{proof}



Recall the exact sequence:
\[
\xymatrix{
 \cdots \ar[r] & H_n (X;\bL^\bullet \langle 0 \rangle) \ar[r] & H_n (X;\bNL^\bullet \langle 1/2 \rangle) \ar[r] & H_{n-1} (X; \bL_\bullet \langle 1 \rangle) \ar[r] & \cdots
}
\]
Putting all together we obtain
\begin{cor}
 Let $X$ be an $n$-dimensional geometric Poincar\'e complex with the Spivak normal fibration $\nu_X \co X \ra \BSG$. Then the following are equivalent
\begin{enumerate}
 \item There exists a lift $\bar \nu_X \co X \ra \BSTOP$ of $\nu_X$
 \item There exists a lift of the normal signature $\nsign_X (X) \in H_n (X;\bNL^\bullet \langle 1/2 \rangle)$ in the group $H_{n} (X;\bL^\bullet \langle 0 \rangle)$.
 \item $0 = t(X) \in H_{n-1} (X;\bL_\bullet \langle 1 \rangle)$.
\end{enumerate}
\end{cor}


\section{Proof of the Main Technical Theorem (II)} \label{sec:proof-part-2}


Let $X$ be a finite $n$-dimensional GPC and suppose that $t(X) = 0$ so that the SNF $\nu_X$ has a topological block bundle reduction and hence there exists a degree one normal map $(f,b) \co M \ra X$ from some $n$-dimensional topological manifold $M$. We want to show that the subset of $L_n (\ZZ [\pi_1 (X)])$ consisting of the inverses of the quadratic signatures of all such degree one normal maps is equal to the preimage of the total surgery obstruction $s (X) \in \SS_n (X)$ under the boundary map $\del \co L_n (\ZZ [\pi_1 (X)]) \ra \SS_n (X)$. 

Let us first look at this map. Inspecting the first of the two commutative braids in section \ref{sec:surgery-sequences} we see that it is in fact obtained from the boundary map $\partial \co L_n (\Lambda (\ZZ)(X)) \ra \SS_n (X)$ using the algebraic $\pi$-$\pi$-theorem of Proposition \ref{prop:algebraic-pi-pi-theorem}. This map is more suitable for investigation since both the source and the target are the $L$-groups of algebraic bordism categories over the same underlying additive category with chain duality, which is $\ZZ_\ast (X)$.

On the other hand there is a price to pay for this point of view. Namely, in the present situation we only have the quadratic signatures $\qsign_{\ZZ[\pi_1(X)]} (f,b)$ as $n$-dimensional QAPCs over the category $\ZZ[\pi_1 (X)]$, but we need a quadratic signature $\qsign_X (f,b)$ over the category $\ZZ_\ast (X)$.\footnote{As shown in Construction \ref{con:quad-construction-over-cplxs-lower-star} in case $X$ is a triangulated manifold we have such a signature but here $X$ is only a Poincar\'e complex with $t(X) = 0$.} A large part of this section will be devoted to constructing such a quadratic signature, it will finally be achieved in Definition \ref{defn:quad-signature-over-X-deg-one-normal-map-to-poincare}. More precisely, we define the quadratic signature 
\[
\qsign_X (f,b) \in L_n (\Lambda (\ZZ)(X)) 
\]
represented by an $n$-dimensional QAC in the algebraic bordism category $\Lambda (\ZZ)(X)$ from Definition \ref{defn:Lambda-K-category}, that means an $n$-dimensional quadratic complex over $\ZZ_\ast (X)$ which is globally Poincar\'e, such that it maps to $\qsign_{\ZZ[\pi_1(X)]} (f,b)$ under the isomorphism of the algebraic $\pi$-$\pi$-theorem. We emphasize that in general the quadratic signature  $\qsign _X (f,b)$ does not produce an element in $H_n (X,\bL_\bullet \langle 1 \rangle)$ since it is not locally Poincar\'e. 

Granting the definition of $\qsign_X (f,b)$ the proof of the desired statement starts with the obvious observation that the preimage $\del^{-1} s(X)$ is a coset of $\ker (\del) = \im (A)$, where $A \co H_n (X;\bL_\bullet \langle 1 \rangle) \ra L_n (\ZZ [\pi_1 (X)])$ is the assembly map. Then the proof proceeds in two steps as follows. 

\begin{enumerate}
\item[(1)] Show that the set of the inverses of the quadratic signatures $\qsign_X (f,b)$ of degree one normal maps with target $X$ is a subset of $\del^{-1} s(X)$ and hence the two sets have non-empty intersection.
\item[(2)] Show that the set of the inverses of the quadratic signatures $\qsign_X (f,b)$ of degree one normal maps with target $X$ is a coset of $\ker (\del) = \im (A)$. Hence we have two cosets of the same subgroup with a non-empty intersection and so they are equal.
\end{enumerate}

The definition of $\qsign_X (f,b)$ and Step (1) of the proof are concentrated in subsections \ref{subsec:general-discussion-of-signatures-over-X} to \ref{subsec:quad-sign-over-X}. The main technical proposition is Proposition \ref{prop:degree-one-normal-map-mfd-to-poincare-over-X} which says that the boundary of the quadratic signature of any degree one normal map from a manifold to $X$ is $s(X)$.

Step (2) of the proof is concentrated in subsections \ref{subsec:identification-of-quad-sign-with-assembly} to \ref{subsec:signatures-versus-orientations}. It starts with an easy corollary of Proposition \ref{prop:degree-one-normal-map-mfd-to-poincare-over-X} which says that although, as noted above, $\qsign_X (f,b)$ does not produce an element in $H_n (X,\bL_\bullet \langle 1 \rangle)$, the difference $\qsign _X (f,b) - \qsign _X (f_0,b_0)$ for two degree one normal maps does produce such an element. Therefore, fixing some $(f_0,b_0)$ and letting $(f,b)$ vary provides us with a map from the normal invariants $\sN (X)$ to $H_n (X,\bL_\bullet \langle 1 \rangle)$. The main technical proposition is then Proposition \ref{prop:identification}. Via the just mentioned difference map it identifies the set of the quadratic signatures of degree one normal maps with the coset of the image of the assembly map containing $\qsign_{\ZZ[\pi_1 (X)])} (f_0,b_0)$.

The principal references are: for Step (1) \cite[sections 7.3, 7.4]{Ranicki(1981)} and for Step (2) \cite[pages 293-298]{Ranicki(1979)} and \cite[section 17]{Ranicki(1992)}.


\subsection{A general discussion of quadratic signatures over $X$} \label{subsec:general-discussion-of-signatures-over-X}

\


As noted above, in case $X$ is a triangulated manifold, we have Construction
\ref{con:quad-construction-over-cplxs-lower-star} which produces from a degree one normal map $(f,b) \co M \ \ra X$ a quadratic signature $\qsign _X (f,b) \in H_n (X,\bL_\bullet \langle 1 \rangle)$. Let us first look at why it is not obvious how to generalize this to our setting. The idea in \ref{con:quad-construction-over-cplxs-lower-star} was to make $f$ transverse to the dual cells of $X$ and to consider the restrictions 
\begin{equation} \label{fragmented-deg-one-map}
(f(\sigma),b(\sigma)) \co (M(\sigma),\partial M(\sigma)) \ra
(D(\sigma),\partial D(\sigma)).
\end{equation}
These are degree one normal maps, but the target $(D(\sigma),\partial D(\sigma))$ is only a normal pair which can be non-Poincar\'e. Consequently we cannot define the Umkehr maps
$f(\sigma)^!$ as in
\ref{con:quad-construction-over-cplxs-lower-star}.

We need an alternative way to define the Umkehr maps. Such a
construction is a relative version of an absolute construction whose
starting point is a degree one normal map $(g,c) \co N \ra Y$ from
an $n$-dimensional manifold $N$ to an $n$-dimensional normal space
$Y$. In this case there is the normal signature $\nsign (Y) \in NL^n
(\ZZ)$ with boundary $\partial \nsign (Y) \in L_{n-1} (\ZZ)$. In
Definition \ref{QuadraticSignatureOfDegree1NormalMapNotPoincare}
below, we recall the definition of a quadratic signature $\qsign
(g,c)$ in this setting, which is an $n$-dimensional QAC over $\ZZ$, not
necessarily Poincar\'e.\footnote{The terminology
``signature'' is perhaps not the most suitable, since we do not
obtain an element in an $L$-group. It is used because this
``signature'' is defined analogously to the signatures of section
\ref{sec:algebraic-cplxs}.} As such it has a boundary, which is an
$(n-1)$-dimensional QAPC over $\ZZ$, and hence defines an element
$\del \qsign (g,c) \in L_{n-1} (\ZZ)$. The following proposition
describes the relationship between these signatures.

\begin{prop} \textup{\cite[Proposition 7.3.4]{Ranicki(1981)}} \label{prop:degree-one-normal-map-mfd-to-normal-cplx}
Let $(g,c) \co N \ra Y$ be a degree one normal map from an
$n$-dimensional manifold to an $n$-dimensional normal space. Then
there are homotopy equivalences of symmetric complexes
\[
h \co \partial \ssign (g,c) \xra{\simeq} - (1+T) \partial
\nsign (Y) 
\]
and a homotopy equivalence of quadratic refinements
\[
h \co \partial \qsign (g,c) \xra{\simeq} - \partial \nsign (Y).
\]
\end{prop}

\begin{rem}
Recall the situation in the case $Y$ is Poincar\'e. Then there is defined the algebraic Umkehr map $g^{!} \co C_\ast (\widetilde Y) \ra C_\ast (\widetilde N)$ and one obtains the symmetric signature $\ssign (g,c)$ with the underlying chain complex the algebraic mapping cone $\sC (g^{!})$. This can be further refined to a quadratic structure $\qsign (g,c)$. In addition one has (see Remark \ref{rem:symmetrization-of-surgery-obstruction})
\[
 \sC(g^{!}) \oplus C_\ast (\widetilde Y) \simeq  C_\ast (\widetilde
 N) \qquad \textup{and} \qquad \ssign (g,c) \oplus \ssign (Y) =
 \ssign (N).
\]
In the situation of Proposition \ref{prop:degree-one-normal-map-mfd-to-normal-cplx} one obtains instead the formula
\[
 \ssign (N) \simeq \ssign (g,c) \cup_h \ssign (Y).
\]
where $\cup_h$ denotes the algebraic gluing of symmetric pairs from \cite[section 3]{Ranicki-I-(1980)}.
\end{rem}

Before going into the proof of Proposition
\ref{prop:degree-one-normal-map-mfd-to-normal-cplx} remember that we
still need its relative version. For that again some preparation is
needed. In particular we need the concept of a boundary of a
symmetric pair. Let $(f \co C \ra D, (\delta \varphi, \varphi))$ be
an $(n+1)$-dimensional symmetric pair which is not necessarily
Poincar\'e. Its boundary is the $n$-dimensional symmetric pair
$(\del f \co \del C \ra \del_+ D,\del_+ \delta \varphi,\del
\varphi))$ with the chain complex
\[
\del_+ D = \sC \big( \smallpairmap \co D^{n+1-\ast} \ra \sC (f)
\big).
\]
defined in \cite{Milgram-Ranicki(1990)}. It is Poincar\'e. Similarly
one can define the boundary of a quadratic pair, which is again a
quadratic pair and also Poincar\'e. Finally the boundary of a normal
pair is a quadratic Poincar\'e pair.

Given $((g,c),(f,b)) \co (N,A) \ra (Y,B)$ a degree one normal map
from an $n$-dimensional manifold with boundary to an $n$-dimensional
normal pair, there are relative versions of the signatures appearing
in Proposition \ref{prop:degree-one-normal-map-mfd-to-normal-cplx}
that are defined in Definition
\ref{QuadraticSignatureOfDegree1NormalMapNotPoincare-relative-version}
and Construction \ref{con:quad-boundary-of-pair} below. Their
relationship is described by the promised relative version of the
previous proposition:

\begin{prop} \label{prop:degree-one-normal-map-mfd-with-boundary-to-normal-pair}
Let $((g,c),(f,b)) \co (N,A) \ra (Y,B)$ be a degree one normal map
from an $n$-dimensional manifold with boundary to an $n$-dimensional
normal pair. Then there are homotopy equivalences of symmetric pairs
\[
h \co \partial \ssign ((g,c),(f,b)) \xra{\simeq} - (1+T) \partial
\nsign (Y,B) 
\]
and a homotopy equivalence of quadratic refinements
\[
h \co \partial \qsign ((g,c),(f,b)) \xra{\simeq} - \partial \nsign
(Y,B).
\]
\end{prop}

Finally recall that our aim is to prove a certain statement about
quadratic chain complexes in the category $\Lambda(\ZZ)\langle 1
\rangle (X)$. Generalizing the definitions above one obtains for a
degree one normal map $(f,b) \co M \ra X$ from an $n$-dimensional
manifold to an $n$-dimensional Poincar\'e complex the desired
quadratic signature $\qsign_X (f,b)$ in Definition
\ref{defn:quad-signature-over-X-deg-one-normal-map-to-poincare}. Its
relationship to the normal signature of $X$ over $X$, which was
already discussed in section \ref{sec:normal-signatures-over-X}, is
described in the following proposition, which can be seen as a
global version of Proposition
\ref{prop:degree-one-normal-map-mfd-with-boundary-to-normal-pair}:

\begin{prop}  \textup{\cite[page 192]{Ranicki(1992)}} \label{prop:degree-one-normal-map-mfd-to-poincare-over-X}
Let $(f,b) \co M \ra X$ be a degree one normal map from an
$n$-dimensional manifold to an $n$-dimensional Poincar\'e complex.
Then there is a homotopy equivalence of symmetric complexes over
$\ZZ_\ast (X)$
\[
h \co \partial \ssign_X (f,b) \xra{\simeq} - (1+T) \partial \nsign_X (X),  
\]
a homotopy equivalence of quadratic refinements over $\ZZ_\ast
(X)$
\[
h \co \partial \qsign_X (f,b) \xra{\simeq} - \partial \nsign_X (X)
\]
and consequently a homotopy equivalence
\[
h \co \partial \qsign_X (f,b) \xra{\simeq} - \partial \vsign_X (X) = -s(X).
\]
\end{prop}

In the following subsections we will define the concepts used above and provide the proofs.



\subsection{Quadratic signature of a degree one normal map to a normal space}

\

Recall that the quadratic construction \ref{constrn:quadconstrn} was
needed in order to obtain the quadratic signature $\qsign (f,b)$ of
a degree one normal map $(f,b) \co M \ra X$ from an $n$-dimensional
manifold to an $n$-dimensional Poincar\'e space. To define $\qsign
(g,c)$ for a degree one normal map $(g,c) \co N \ra Y$ from an
$n$-dimensional manifold to an $n$-dimensional normal space we need
the spectral quadratic construction
\ref{con:SpectralQuadraticConstruction}.

\begin{defn} \cite[Proposition 7.3.4]{Ranicki(1981)} \label{QuadraticSignatureOfDegree1NormalMapNotPoincare}
Let~$(g,c) \co N \ra Y$ be a degree one normal map from a Poincar\'e
complex~$N$ to a normal complex~$Y$. The \emph{quadratic signature}
of~$(g,c)$ is the $n$-dimensional QAC
\[
 \qsign (g,c) = (C,\psi)
\]
which is not necessarily Poincar\'e, obtained from a choice of the
Thom class $u (\nu_Y) \in \tilde{C}^k (\Th (\nu_Y))$ as follows.

Consider the commutative diagrams
\[
\xymatrix{ C^{n-\ast} (N) \ar[d]_{(\varphi(N))_0}^{\simeq} &
C^{n-\ast} (Y) \ar[l]_-{g^\ast} \ar[d]^{(\varphi(Y))_0}
& & \Th(\nu_N)^\ast \ar[d]_{\Gamma_N}^{\simeq} & \Th (\nu_Y)^{\ast} \ar[l]_{\Th (c)^{\ast}} \ar[d]^{\Gamma_Y} \\
C (N) \ar[r]_{g_\ast} & C(Y) & & \Sigma^p N_+ \ar[r]_{\Sigma^p g_+}
& \Sigma^p Y_+ }
\]
The maps $\Gamma_{\_}$ in the right diagram are obtained using the
$S$-duality as in Construction
\ref{con:S-duality-and-Thom-is-Poincare}. In fact using the
properties of the $S$-duality explained in Construction
\ref{con:S-duality-and-Thom-is-Poincare} we see that the left
diagram can be considered as induced from the right diagram by
applying the chain complex functor $C_{\ast+p} (-)$.

Set $g^{!} = (\varphi(N))_0 \circ g^\ast$ and define~$C = \sC
(g^{!})$ and $\psi = \Psi (u (\nu_Y)^\ast)$, where $\Psi$ denotes
the spectral quadratic construction on the map $\Gamma_N \circ \Th
(c)^\ast$.

Also note that by the properties of the spectral quadratic
construction we have
\[
(1+T) \psi \equiv e_{g^!}^{\%} (\varphi(N))
\]
where $e_{g^!} \co C(N) \ra \sC (g^!)$ is the inclusion.
\end{defn}

For the proof of Proposition \ref{prop:degree-one-normal-map-mfd-to-normal-cplx} we also need an
additional property of the spectral quadratic construction.

\begin{prop} \textup{\cite[Proposition 7.3.1. (v)]{Ranicki(1981)}} \label{PropertiesOfSpectralQuadraticConstruction}
Let~$F \co X \lra \Sigma^{p} Y$ and $F' \co X' \lra \Sigma^{p} Y'$
be semi-stable maps fitting into the following commutative diagram
 \[
  \xymatrix{
     X       \ar[d]_{G_X}  \ar[r]^-F
   & \Sigma^{p} Y   \ar[d]^{G_Y}
   \\ 
     X'      \ar[r]^-{F'}
   & \Sigma^{p} Y'
  }
 \]
 inducing the commutative diagram of chain complexes
 \[
  \xymatrix{
     \Sigma^{-p}\tilde{C}(X)       \ar[d]^{g_X}  \ar[r]^-{f}
   & \tilde{C}(Y)                  \ar[d]^{g_Y}  \ar[r]^-{e}
   & \cone(f)                    \ar[d]^{\smalltwoxtwomatrix{g_Y}{0}{0}{g_X}}
   \\ 
     \Sigma^{-p}\tilde{C}(X')      \ar[r]^-{f'}
   & \tilde{C}(Y')                 \ar[r]^-{e'}
   & \cone(f')
  }
 \]
 Then the spectral quadratic constructions of~$F$ and $F'$ are related
 by
 \[
  \Psi (F') \circ g_X \equiv \smalltwoxtwomatrix{g_Y}{0}{0}{g_X}_\% \circ \Psi (F) + (e')_\% \circ \psi (G_Y) \circ f
 \]
 where $\Psi (-)$ and $\psi(-)$ denote the (spectral) quadratic constructions on the respective maps.
\end{prop}


\begin{proof}[Proof of Proposition \ref{prop:degree-one-normal-map-mfd-to-normal-cplx}]
For ease of notation, let
\begin{itemize}
\item $\ssign (g,c) = (\cone(g^!),\varphi (g^{!}))$,
\item $\nsign (Y) = (C(Y),\psi(Y))$,
\end{itemize}
and set
\begin{itemize}
\item $\del \ssign (g,c) = (\del\cone(g^!),\del\varphi(g^!))$,
\item $\del \nsign (Y) = (\del C(Y),\del \psi (Y))$.
\end{itemize}

Consider the following commutative diagram where all rows and
columns are cofibration sequences (the diagram also sets the
notation $\mu$, $q_g$ and $e_{g^!}$)
\begin{equation} \label{dgrm:umkehr-for-mfd-to-normal-cplx}
\begin{split}
  \xymatrix{
     0 \ar[r] \ar[d]
     & C(Y)^{n-\ast} \ar[r]^{\id} \ar[d]^{g^!}
     & C(Y)^{n-\ast} \ar[d]^{\varphi(Y)_0}
    \\
    \dsc(g)           \ar[r]^{q_g}      \ar[d]^\id
      & C(N)          \ar[r]^g          \ar[d]_{e_{g^!}}
      & C(Y)
    \\
    \dsc(g)        \ar[r]^\mu
      & \cone(g^!)
      &
  }
\end{split}
\end{equation}
We obtain a homotopy equivalence, say $h' \co \sC (\mu) \raeq \sC (\varphi(Y)_0) \simeq \Sigma \del C(Y)$. Inspection shows that $\sC (g^!)^{n-\ast} \simeq \dsc (g)$ and that under this homotopy equivalence the duality map $\varphi (g^{!})_0$ is identified with the map $\mu$. Hence we have
\begin{equation} \label{eqn:naturality-of-htpy-equiv-on-boundary}
e_Y \circ g = h' \circ e_{\varphi(g^!)_0} \circ e_{g^{!}}
\end{equation}
where $e_{\varphi(g^!)_0} \co \sC (g^{!}) \ra \Sigma \del \sC (g^{!})$ and $e_Y
\co C(Y) \ra \Sigma \del C (Y)$ are the inclusions. We also obtain a commutative braid of chain complexes, which we leave for the reader to draw, with chain homotopy equivalences
\[
h \co \del \cone(g^!) \xra{\simeq} \del C(Y) \quad \textup{and} \quad  h' \co \sC (\varphi (g^!)_0) \simeq \sC (\mu)
\xra{\simeq} \sC (\varphi (Y)_0)
\]
which are related by $h' = - \Sigma(h)$, thanks to the sign conventions used for definitions of mapping cones and suspensions.

Next we consider the symmetric structures. Recall that, by definition we have $S(\del\varphi(g^{!})) = e_{\varphi(g^!)_0}^\% \circ e_{g^{!}}^\% (\varphi (N))$, and $S(\del\varphi(Y))=e_Y^\% (\varphi(Y))$. Further we have $g^\% (\varphi (N)) = \varphi (Y)$ and hence
\[
  (h')^\% S(\del\varphi(g^{!}))) = S(\del\varphi(Y)).
\]
By the injectivity of the suspension we also have~$h^\% (\del\varphi(g)) = - \del\varphi(Y)$.

Finally we study the quadratic structures. Set~$\del\nsign (Y) =
(\del C(Y),\del\psi(Y))$, $\qsign (g,c) =
(\del\cone(g^!),\del\psi(g^{!}))$. Recall that the spectral
quadratic constructions are employed to define the quadratic
structures $\del\psi(Y)$ and $\del\psi(g^{!})$. By properties of the
$S$-duality the semi-stable maps used in these constructions fit
into the commutative diagram
\begin{equation} \label{dgrm:umkehr-for-mfd-to-normal-cplx-on-level-of-spaces}
\begin{split}
 \xymatrix{
  \Th(\nu_Y)^\ast \ar[d]_{\Gamma_N \circ \Th(c)^\ast} \ar[r]_{\id}   & \Th(\nu_Y)^\ast \ar[d]^{\Gamma_Y}  \\
  \Sigma^p N_+ \ar[r]^{\Sigma^p g} & \Sigma^p Y_+
 }
\end{split}
\end{equation}
which in fact induces the upper right part of the Diagram
(\ref{dgrm:umkehr-for-mfd-to-normal-cplx}). By Proposition
\ref{PropertiesOfSpectralQuadraticConstruction} the spectral
quadratic construction applied to the maps~$\Gamma_Y$, $\Gamma_N
\circ \Th(c)^\ast$ and $\Sigma^p g_+$ satisfy the following relation
\[
 \Psi (\Gamma_Y) = \smalltwoxtwomatrix{g}{0}{0}{1}_\% \circ \Psi (\Gamma_N \circ \Th(c)^\ast) + (e_{g_G})_\% \circ \psi (\Sigma^p g_+) \circ g_F
\]
where the symbols in the brackets specify the map to which the
(spectral) quadratic constructions are applied. However, the
map~$\Sigma^p g_+$ comes from the map~$g \co N \ra Y$ and so~$\psi
(\Sigma^p g_+) = 0$. This leads to the commutative diagram
\begin{equation} \label{dgrm:comparing-two-quad-str-on-cones}
\begin{split}
 \xymatrix{
    \tilde{C}_{n+p}(\Th(\nu_Y)^\ast)  \ar[rr]^{\Psi (\Gamma_N \circ \Th(c)^\ast)} \ar[drr]_{\Psi (\Gamma_Y)}
  &
  & (\Wq{\cone(g^!)})_n                                      \ar[d]^{\smalltwoxtwomatrix{g}{0}{0}{1}_\% = (h' \circ e_{\varphi(g^!)_0})_\%}
  \\
  &
  & (\Wq{\cone(\varphi(Y)_0)})_n
 }
\end{split}
\end{equation}
The identification of the vertical map comes from Diagram
(\ref{dgrm:umkehr-for-mfd-to-normal-cplx}) and equation
(\ref{eqn:naturality-of-htpy-equiv-on-boundary}). Hence we obtain
that
\begin{align*}
h'_\% (S \del \psi(g^{!})) & = h'_\% \circ (e_{\varphi(g^!)_0})_\%(S \psi(g^{!})) =
(h' \circ e_{\varphi(g^!)_0})_\% \Psi (\Gamma_N \circ \Th(c)^\ast) (u(\nu_Y)^\ast)
= \\ & = \Psi (\Gamma_Y) (u(\nu_Y)^\ast) = S \del \psi(Y).
\end{align*}
The uniqueness of desuspension as presented in Construction \ref{con:normal-con-via-spectral-quadratic-con} yields the desired
\[
h_\% (\del \psi(g^{!})) = - \del \psi(Y)
\]
thanks again to $h' = - \Sigma (h)$.
\end{proof}


\subsection{Quadratic signature of a degree one normal map to a normal pair}

\

Now we aim at proving Proposition
\ref{prop:degree-one-normal-map-mfd-with-boundary-to-normal-pair},
which is a relative version of the proposition just proved. First we
need a relative version of Definition
\ref{QuadraticSignatureOfDegree1NormalMapNotPoincare}. For that we
need to know how to apply the spectral quadratic construction in the
relative setting and for that, in turn, how to apply $S$-duality in
the relative setting.

\begin{construction} \cite[Proposition 7.3.1]{Ranicki(1981)} \label{con:RelativeSpectralQuadraticConstruction}
Let~$(G,F) \co (X,A) \lra \Sigma^{p} (Y,B)$ be a semi-stable map
between pointed pairs. Consider the following diagram of induced
chain maps
\[
\xymatrix{
 \tilde{C}(A)_{p+\ast} \ar[r]^-{f} \ar[d]^{i} & \tilde{C}(B) \ar[r] \ar[d]^{j} & \sC (f) \ar[d]^{(j,i)} \\
 \tilde{C}(X)_{p+\ast} \ar[r]^-{g} & \tilde{C}(Y) \ar[r] & \sC (g)
 }
\]
The \emph{relative spectral quadratic construction} on~$(G,F)$ is a
chain map
\[
\Psi \co \Sigma^{-p} \tilde{C}(X,A) \lra \sC ((j,i)_{\%})
\]
such that
\[
 (1+T) \circ \Psi \equiv e^\% \circ \varphi \circ (g,f)
\]
where $\varphi \co \tilde{C}(Y,B) \lra \sC (j^{\%})$ is the relative
symmetric construction on $(Y,B)$ (Construction \ref{con:rel-sym})
and $e^{\%} \co \sC (j^{\%}) \ra \sC ((j,i)^{\%})$ is the map induced by the right hand square in the diagram above. The existence of $\Psi$ follows from the naturality of the diagram in Construction
\ref{con:SpectralQuadraticConstruction}.
\end{construction}

The $S$-duality is applied to normal pairs as follows. Recall that
an $(n+1)$-dimensional geometric normal pair $(Y,B)$ comes with the
map of pairs
\[
(\rho_Y,\rho_B) \co (D^{n+k+1},S^{n+k}) \ra (\Th (\nu_Y),\Th
(\nu_B))
\]
In the absolute case the map $\rho$ composed with the diagonal map
gave rise to a map $\Gamma$ which was the input for the spectral
quadratic construction. Now we have three diagonal maps producing
three compositions:
\begin{align*}
S^{n+k} \xra{\rho_B} \Th (\nu_B) & \xra{\Delta} B_+ \wedge \Th (\nu_B) \\
S^{n+k+1} \xra{\rho_Y/\rho_B} \Th (\nu_Y) / \Th (\nu_B) & \xra{\Delta} Y_+ \wedge \Th (\nu_Y) / \Th (\nu_B) \\
S^{n+k+1} \xra{\rho_Y/\rho_B} \Th (\nu_Y) / \Th (\nu_B) &
\xra{\Delta} Y/B \wedge \Th (\nu_Y)
\end{align*}
These induce three duality maps which fit into a commutative diagram
as follows:
\begin{equation} \label{diag:relative-S-duality}
 \begin{split}
  \xymatrix{
  \Sigma^{-1} \Th (\nu_B)^\ast \ar[d]^{\Sigma^{-1} \Gamma_B} \ar[r]^(0.42){i} & (\Th (\nu_Y)/\Th (\nu_B))^\ast \ar[d]^{\Gamma_Y} \ar[r] & \Th (\nu_Y)^\ast \ar[d]^{\Gamma_{(Y,B)}} \ar[r] & \Th (\nu_B)^\ast \ar[d]^{\Gamma_B} \\
  \Sigma^p B \ar[r]^{j} & \Sigma^p Y \ar[r] & \Sigma^p (Y/B) \ar[r] & \Sigma^{p+1} B.
  }
 \end{split}
\end{equation}

\begin{con} \label{con:quad-boundary-of-pair}
The quadratic boundary of the normal pair $(Y,B)$ is an
$n$-dimensional QAPP
\[
\del \nsign (Y,B) = (\del C(B) \ra \del_+ C(Y),(\psi(Y),\psi(B)))
\]
obtained by applying the relative spectral quadratic construction on
the pair of maps $(\Gamma_Y,\Sigma^{-1} \Gamma_B)$ with
$(\psi(Y),\psi(B))$ the image of $u(\nu(Y))^\ast \in \tilde{C} (\Th
(\nu_Y)^\ast)_{n + p}$ under
\[
\Psi \co \Sigma^{-p} \tilde{C} (\Th (\nu_Y)^\ast) \ra \sC ((j,i)_{\%})
\]
where $i$ and $j$ are as in Diagram (\ref{diag:relative-S-duality}).
\end{con}

\begin{defn}\label{QuadraticSignatureOfDegree1NormalMapNotPoincare-relative-version}
Let~$((g,c),(f,b)) \co (N,A) \ra (Y,B)$ be a degree one normal map
from a Poincar\'e pair~$(N,A)$ to a normal pair~$(Y,B)$ of dimension
$(n+1)$. The \emph{quadratic signature} of~$((g,c),(f,b))$ is the
$n$-dimensional quadratic pair
\[
 \qsign ((g,c),(f,b)) = (j \co C \ra D,(\delta \psi,\psi))
\]
which is not necessarily Poincar\'e, obtained from a choice of the
Thom class $u (\nu_Y)$ as follows.

The $S$-duality produces a commutative diagram
\[
\xymatrix{
\Sigma^{p} A_+ \ar[d] & & \Sigma^{-1} \Th (\nu_B)^\ast \ar[ll]_{\Gamma_A \circ \Th(b)^\ast} \ar[d] \\
\Sigma^p N_+ & & (\Th (\nu_Y) / \Th (\nu_B) )^\ast
\ar[ll]^(0.6){\Gamma_N \circ (\Th(c)/\Th(b))^\ast} }
\]
inducing the diagram of chain complexes
\[
\xymatrix{
C(A) \ar[d] & C^{n-\ast} (B) \ar[l]_{f^!} \ar[d] \\
C(N) & C^{n+1-\ast} (Y,B) \ar[l]^(0.6){g^!} }
\]

Define~$C = \sC (f^{!})$, $D = \sC (g^!)$ and $(\delta \psi, \psi) =
\Psi (u (\nu_Y)^\ast)$, where the spectral quadratic construction is
on the pair of maps $(\Gamma_N \circ (\Th(c)/\Th(b))^\ast,\Gamma_A \circ
\Th(b)^\ast)$.
\end{defn}

\begin{proof}[Proof of Proposition
\ref{prop:degree-one-normal-map-mfd-with-boundary-to-normal-pair}]
The proof follows the same pattern as the proof of Proposition
\ref{prop:degree-one-normal-map-mfd-to-normal-cplx}. With the
notation of Construction \ref{con:quad-boundary-of-pair} and
Definition
\ref{QuadraticSignatureOfDegree1NormalMapNotPoincare-relative-version}
one first observes that we have a homotopy equivalence of pairs
\[
(\del h \ra h) \co (\del C(B) \ra \del_+ C(Y)) \xra{\simeq} (j \co C
\ra D)
\]
by studying a map of diagrams of the same shape as Diagram
(\ref{dgrm:umkehr-for-mfd-to-normal-cplx}). For the symmetric
structures observe that the homotopy equivalence $(\del h \ra h)$
satisfies an equation analogous to
(\ref{eqn:naturality-of-htpy-equiv-on-boundary}) and again use the
naturality. Finally to obtain the desired equivalence of quadratic
structures there is a again a map of commutative squares of the form
as in Diagram
(\ref{dgrm:umkehr-for-mfd-to-normal-cplx-on-level-of-spaces}). A
diagram chase shows that the relative spectral construction
satisfies a formula analogous to the one appearing in Proposition
\ref{PropertiesOfSpectralQuadraticConstruction}. This leads to a
diagram analogous to Diagram
(\ref{dgrm:comparing-two-quad-str-on-cones}). As in the absolute
case unraveling what it means and using the desuspension produces
the desired equation.
\end{proof}

\begin{rem} \label{pairs-vs-k-ads}
Just as explained in section \ref{sec:cat-over-cplxs} the relative
version just proved has a generalization for $k$-ads.
\end{rem}


\subsection{Quadratic signature over $X$ of a degree one normal map to $X$} \label{subsec:quad-sign-over-X}

\

Now we want to prove Proposition \ref{prop:degree-one-normal-map-mfd-to-poincare-over-X}. The
preparation starts with discussing $\nsign_X (X)$ for an $n$-dimensional GPC. This was defined in section \ref{sec:normal-signatures-over-X} by first passing to $\gnsign_X
(X)$ and then applying the spectrum map $\nsign \co \Omega^N_\bullet
\ra \bNL^\bullet$ from \cite{Weiss-II(1985)}. By the proof of
\cite[Theorem 7.1]{Weiss-II(1985)} the spectrum map $\nsign$
composed with the boundary fits with the quadratic boundary
construction as described in subsection
\ref{subsec:quadratic-boundary-of-gnc}, so we can think of $\del
\nsign_X (X)$ as a collection of $(n-|\sigma|)$-dimensional
quadratic $(m-|\sigma|)$-ads indexed by simplices of $X$ which fit
together and are obtained by the relative spectral construction
described in Construction \ref{con:quad-boundary-of-pair}.

\begin{defn} \label{defn:quad-signature-over-X-deg-one-normal-map-to-poincare}
Let $(f,b) \co M \ra X$ be a degree one normal map from a closed
$n$-dimensional topological manifold to an $n$-dimensional GPC. Make
$f$ transverse to the dual cell $D(\sigma,X)$ for each $\sigma \in
X$ so that we have a degree one normal map
\[
 (f[\sigma],f[\del \sigma]) \co  (M[\sigma],\del M[\sigma]) \ra (X[\sigma],\del X[\sigma])
\]
from an $(n-|\sigma|)$-dimensional manifold with boundary to an
$(n-|\sigma|)$-dimensional normal pair. Define the \emph{quadratic
signature over $X$} of $(f,b)$ to be the element
\[
 \qsign_X (f,b) \in L_n (\Lambda (\ZZ) (X))
\]
represented by the $n$-dimensional QAC $(C,\psi)$ in $\Lambda (\ZZ)
(X)$ whose component over $\sigma \in X$ is the relative quadratic
signature
\[
\qsign \big( (f(\sigma),b(\sigma)), \del (f(\sigma),b(\sigma)) \big)
\]
obtained as in Definition
\ref{QuadraticSignatureOfDegree1NormalMapNotPoincare-relative-version}.
The resulting element is independent of all the choices.
\end{defn}

\begin{proof}[Proof of Proposition
\ref{prop:degree-one-normal-map-mfd-to-poincare-over-X}] In order to
prove the proposition it is necessary to prove that one has homotopy
equivalences as in the statement but for each simplex and so that
they fit together. However, for each simplex this is exactly the
statement of Proposition
\ref{prop:degree-one-normal-map-mfd-with-boundary-to-normal-pair}.
Since one can proceed inductively from simplices of the top
dimension to smaller simplices the homotopy equivalences can be made
to fit together.

To obtain the last homotopy equivalence recall that $X$ is an $n$-dimensional Poincar\'e complex and hence the same complex that defines the normal signature defines the visible signature over $X$, see section \ref{sec:normal-signatures-over-X} if needed.
\end{proof}



\subsection{Identification of the quadratic signature with the assembly} \label{subsec:identification-of-quad-sign-with-assembly}

\


Now we proceed to Step (2) of the proof of the Main Technical Theorem part (II). We first state the following preparatory proposition which describes what happens when we consider the difference of the quadratic signatures of two degree one normal maps. 

\begin{prop}  \label{prop:difference-of-quad-sign-is-poincare}
Let $(f_i,b_i) \co M_i \ra X$ with $i = 0,1$ be two degree one normal maps from $n$-dimensional topological manifolds to an $n$-dimensional GPC. Then the difference of their quadratic signatures over $\ZZ_\ast (X)$ is an $n$-dimensional QAC in the algebraic bordism category $\Lambda \langle 1 \rangle (\ZZ)_\ast (X)$ and hence represents an element 
\[
\qsign_X (f_1,b_1) - \qsign_X (f_0,b_0) \in H_n (X ; \bL_\bullet \langle 1 \rangle). 
\]
\end{prop}

\begin{proof}
A quadratic chain complex in $\ZZ_\ast (X)$ is an $n$-dimensional QAC in the algebraic bordism category $\Lambda \langle 1 \rangle (\ZZ)_\ast (X)$ if and only if it is locally Poincar\'e which is equivalent to saying that its boundary is contractible. So it is enough to prove that the two quadratic complexes representing $\qsign_X (f_i,b_i)$ have homotopy equivalent boundaries. This follows from Proposition \ref{prop:degree-one-normal-map-mfd-to-poincare-over-X} since they are both homotopy equivalent to $-\del \nsign_X (X)$.
\end{proof}

The degree one normal maps with the target $X$ are organized in the normal invariants $\sN(X)$. The above proposition tells us that the quadratic signature over $X$ relative to $(f_0,b_0)$ defines a map
\begin{equation} \label{eqn:quad-sign-rel-to-x-0}
 \qsign_X (-,-) - \qsign_X (f_0,b_0) \co \sN (X) \ra H_n (X ; \bL_\bullet \langle 1 \rangle). 
\end{equation}
The following proposition is the main result in the proof of Step (2). It says that for $X$ an $n$-dimensional Poincar\'e complex such that $t(X) = 0$ the surgery obstruction map $\qsign_{\ZZ [\pi_1 (X)]} \co \sN (X) \ra L_n (\ZZ [\pi_1 (X)])$ can be identified with the assembly map. When $X$ is already a manifold the map $(f_0,b_0)$ can be taken to be the identity. 

\begin{prop} \textup{\cite[pages 293-297]{Ranicki(1979)}, \cite[proof of Theorem 17.4]{Ranicki(1992)}} \  \label{prop:identification} 
\
Let $X$ be an $n$-dimensional GPC with $\pi = \pi_1 (x)$ such that $t(X) = 0$ and let $(f_0,b_0) \co M_0 \ra X$ be any choice of a degree one normal map. Then the diagram 
\[
 \xymatrix{
  \sN (X) \ar[rrr]^(0.43){\qsign_{\ZZ [\pi]} (-,-) - \qsign_{\ZZ [\pi]} (f_0,b_0)} \ar[d]_{\qsign_X (-,-) - \qsign_X (f_0,b_0)}^{\cong} & & & L_n (\ZZ [\pi_1 (X)]) \ar[d]^{=} \\
  H_n (X ; \bL_\bullet \langle 1 \rangle) \ar[rrr]_{A} & & & L_n (\ZZ [\pi_1 (X)])
 }
\]
is commutative and the left vertical map is a bijection.
\end{prop}

\begin{proof}[Overview of the proof]

\

To see the commutativity consider $(f,b) \co M \ra X$ a degree one normal map from an $n$-dimensional manifold to an $n$-dimensional Poincar\'e complex $X$ and $A \co \Lambda (X) \ra \Lambda (\ZZ[\pi_1 (X)])$  the assembly functor. Then we have
\[
 A(\qsign_X (f,b)) = \qsign_{\ZZ[\pi_1 (X)])} (f,b) \in L_n (\Lambda (\ZZ[\pi_1 (X)]))
\]
since the assembly corresponds to geometric gluing, see Remark \ref{rem:hlgical-assembly}.

We are left with showing that the left hand vertical map is a bijection which will be done by identifying this map with a composition of four maps which are all bijections. In order to save space we will abbreviate (using $x_0$ for $(f_0,b_0)$)
\begin{align*}
 \qsign_X (-;x_0) & := \qsign_X (-,-) - \qsign_X (f_0,b_0)
\end{align*}
The strategy of the proof can be summarized in the following diagram:
\[
  \xymatrix{
   \sN (X) \ar[d]^{t (-,x_0)}_{\cong} \ar@/_8pc/[dddd]^(0.35){\qsign_X (-;x_0)} \ar[r]^{=} & \sN (X) \ar[d]^{t (-,x_0)}_{\cong} \ar@/^9pc/[dddd]_(0.35){\sign^{\G/\TOP}_X (-;x_0)} \\
   [X;\G/\TOP] \ar[d]^{\qsign}_{\cong} \ar[r]^{=} & [X;\G/\TOP] \ar[d]^{\widetilde \Gamma} & \\ 
   H^0 (X ; \bL_\bullet \langle 1 \rangle) \ar[d]^{- \cup u^{\bL_\bullet} (\nu_0)}_{\cong} & H^0 (X ; \Sigma^{-1} \bOmega^{N,\STOP}_\bullet) \ar[d]^{- \cup u^{\STOP}(\nu_0)}_{\cong} \ar[l]_{\qsign} \\
   H^k (\Th (\nu_X) ; \bL_\bullet \langle 1 \rangle) \ar[d]^{S-\textup{dual}}_{\cong} & H^k (\Th (\nu_X) ; \Sigma^{-1} \bOmega^{N,\STOP}_\bullet) \ar[d]^{S-\textup{dual}}_{\cong} \ar[l]_{\qsign} \\
   H_n (X ; \bL_\bullet \langle 1 \rangle) & H_n (X ; \Sigma^{-1} \bOmega^{N,\STOP}_\bullet) \ar[l]_{\qsign}
  }
 \]
Some of the maps in the diagram have been defined already, the remaining ones will be defined shortly. We will show that those marked with $\cong$ are bijections or isomorphisms. Once this is done it is enough to show that the left hand part of the diagram is commutative, because then we have indeed identified $\qsign_X (-;x_0)$ with a composition of four bijections. The commutativity of the left hand part will be shown by proving that:
\begin{enumerate}
 \item[(a)] the outer square commutes in subsection \ref{subsec:quad-sign-versus-normal-sign},
 \item[(b)] the middle part commutes in subsection \ref{subsec:proof-of-2}, 
 \item[(c)] the right hand part commutes in subsection \ref{subsec:signatures-versus-orientations}.
\end{enumerate}
The intermediate subsections contain the necessary definitions.
\end{proof}


\subsection{Proof of (a) - Quadratic signatures versus normal signatures} 
\label{subsec:quad-sign-versus-normal-sign}

\


Recall that for a degree one normal map $(f,b) \co M \ra X$ from an $n$-dimensional manifold to an $n$-dimensional GPC we have two ways how to obtain its quadratic signature over $\ZZ [\pi_1 (X)]$. Namely via the Umkehr map $f^!$ as in Construction \ref{constrn:quadconstrn} or via the normal structure on the mapping cylinder $W$ as in Example \ref{expl:normal-symm-poincare-pair-gives-quadratic}. It was further shown in Example \ref{expl:normal-symm-poincare-pair-gives-quadratic} that these two constructions yield the same result. In subsection \ref{subsec:quad-sign-over-X} we have defined the quadratic signature of $(f,b)$ over $\ZZ_\ast (X)$ using Umkehr maps providing an analogue to Construction \ref{constrn:quadconstrn}. Here we provide an analogue to Example \ref{expl:normal-symm-poincare-pair-gives-quadratic} over $\ZZ_\ast (X)$. 

\begin{expl} \label{expl:normal-pair-gives-quadratic-cplx-not-Poincare}
This is a generalization of the results of subsection \ref{subsec:quadratic-boundary-of-gnc}. Let us consider a degree one normal map $(g,c) \co N \ra Y$ from an $n$-dimensional manifold to an $n$-dimensional GNC. We obtain an $(n+1)$-dimensional normal pair $(W,N \sqcup Y)$ where $W$ is the mapping cylinder of $f$. However, in contrast to the situation in Example \ref{expl:normal-symm-poincare-pair-gives-quadratic}, the disjoint union $N \sqcup Y$ is no longer a Poincar\'e complex. Therefore the associated algebraic normal pair 
\[
 (\nsign (W),\ssign (N) - \nsign (Y))
\]
does not have a Poincar\'e boundary. Nevertheless we can still perform algebraic surgery on this pair, just as in Lemma \ref{lem:normal-sym-pair-gives-quad-poincare-cplx} and thanks to the spectral quadratic construction the result of the surgery is an $n$-dimensional quadratic complex, which however, will not be Poincar\'e. The proof from Example \ref{expl:normal-symm-poincare-pair-gives-quadratic} translates almost word-for-word to an identification of this quadratic complex with $\qsign (g,c)$ from Definition \ref{QuadraticSignatureOfDegree1NormalMapNotPoincare} (the only difference being the fact that the map $\varphi_0|_Y$ is no longer an equivalence). In symbols we have
\begin{equation}
\textup{Lemma } \ref{lem:normal-sym-pair-gives-quad-poincare-cplx}
\co (\nsign (W),\ssign (N) - \nsign (Y)) \mapsto \qsign (g,c).
\end{equation}
Using the relative version of $S$-duality and of the spectral quadratic construction from earlier in this section one also obtains a relative version of this identification.
\end{expl}

\begin{expl} \label{expl:normal-symmetric-pair-gives-quadratic-over-X}
Starting now with a degree one normal map $(f,b) \co M \ra X$ from an $n$-dimensional manifold to an $n$-dimensional GPC consider the dissection 
\[
 X = \bigcup_{\sigma \in X} X(\sigma) \quad (f,b) = \bigcup_{\sigma \in X} (f(\sigma),b(\sigma)) \co M (\sigma) \ra X (\sigma)
\]
where each $(f(\sigma),b(\sigma))$ is a degree one normal map from an $(n-|\sigma|)$-dimensional manifold $(m-|\sigma|)$-ad to an $(n-|\sigma|)$-dimensional normal $(m-|\sigma|)$-ad. As such it gives rise to an $(n+1-|\sigma|)$-dimensional pair of normal $(m-|\sigma|)$-ads
\[
 (W(\sigma),\nu(b(\sigma)),\rho(b(\sigma))).
\]
Applying Example \ref{expl:normal-pair-gives-quadratic-cplx-not-Poincare} shows that the quadratic chain complex over $\ZZ_\ast (X)$ obtained this way coincides with the quadratic signature $\qsign_X (f,b)$ from Definition \ref{defn:quad-signature-over-X-deg-one-normal-map-to-poincare}.
\end{expl}

The following Lemma is a generalization of ideas from Construction \ref{con:normal-symmetric-signature-over-X-for-a-deg-one-normal-map}. In its statement a use is made of the quadratic signature map $\qsign \co \Sigma ^{-1} \bOmega^{N,\STOP}_\bullet \ra \bL_\bullet \langle 1 \rangle$ from Proposition \ref{prop:signatures-on-spectra-level}.

\begin{lem} \label{lem:normal-symmetric-signature-over-X-for-a-difference-of-deg-one-normal-maps}
Let $x_i = (f_i,b_i) \co M_i \ra X$ with $i = 0,1$ be two degree one normal maps from $n$-dimensional topological manifolds to an $n$-dimensional GPC. Then there exists a $\G/\TOP$-signature 
\[
\sign^{\G/\TOP}_X (x_1,x_0) \in H_n (X ; \Sigma ^{-1} \bOmega^{N,\STOP}_\bullet)
\]
such that 
\[
 \qsign_X (f_1,b_1) - \qsign_X (f_0,b_0) = \qsign (\sign^{\G/\TOP}_X (x_1,x_0)) \in H_n (X ; \bL_\bullet \langle 1 \rangle). 
\]
\end{lem}

\begin{proof}
Consider the dissections of $(f_i,b_i) \co M_i \ra X$ as in Example \ref{expl:normal-symmetric-pair-gives-quadratic-over-X}. The assignments
\[
 \sigma \mapsto (W_i(\sigma),\nu(b_i(\sigma)),\rho(b_i(\sigma))).
\]
fit together to produce cobordisms of $\bOmega^N_\bullet$-cycles in the sense of Definition \ref{defn:E-cycles}. However, they do not produce a $\Sigma ^{-1} \bOmega^{N,\STOP}_\bullet$-cycle since the ends of the cobordisms given by $X$ are not topological manifolds. But the two ends for $i=0,1$ are equal and so we can glue the two cobordisms along these ends and we obtain for each $\sigma \in X$ the $(n+1-|\sigma|)$-dimensional pairs of normal $(m-|\sigma|)$-ads
\[
 (W_1(\sigma)) \cup_{X(\sigma)} W_0(\sigma)),\nu(b_1(\sigma)) \cup_{\nu_{X(\sigma)}} \nu(b_0(\sigma)),\rho(b_1(\sigma)) \cup_{\rho(\sigma)} \rho(b_0(\sigma))).
\]
which now fit together to produce a $\Sigma^{-1} \bOmega^{N,\STOP}_\bullet$-cycle in the sense of Definition \ref{defn:E-cycles}. This produces the desired signature 
\[
\sign^{\G/\TOP}_X (x_1,x_0) \in H_n (X ; \Sigma^{-1} \bOmega^{N,\STOP}_\bullet).
\]

To prove the equation recall from Proposition \ref{prop:connective-signatures-on-spectra-level} the quadratic signature map $\qsign \co \Sigma^{-1} \bOmega^{N,\STOP}_\bullet \ra \bL_\bullet \langle 1 \rangle$. We have to investigate the value of the induced map on the just defined $\G/\TOP$-signature. By definition this value is given on each simplex $\sigma \in X$ as the $(n-|\sigma|)$-dimensional quadratic Poincar\'e $(m-|\sigma|)$-ad obtained by the algebraic surgery on the algebraic pair extracted from the (normal,topological manifold) pair $(W_1(\sigma) \cup_{X(\sigma)} W_0(\sigma),M_1 (\sigma) \sqcup M_0 (\sigma))$. 

Consider now the left hand side of the desired equation. By Example \ref{expl:normal-symmetric-pair-gives-quadratic-over-X} the value of each summand on a simplex $\sigma$ is obtained via algebraic surgery on the algebraic pair extracted from the normal pair $(W_i(\sigma),M_i(\sigma) \sqcup X(\sigma))$ (whose boundaries are not Poincar\'e and so the resulting complexes are also not Poincar\'e). Subtracting these corresponds to taking the disjoint union of the normal pairs above and reversing the orientation on the one labeled with $i = 0$. On the other hand there is a geometric normal cobordism between geometric normal pairs 
\[
 (W_1(\sigma) \cup_{X(\sigma)} W_0(\sigma), M_1(\sigma) \sqcup -M_0(\sigma))
\]
and
\[  
(W_1(\sigma) \sqcup - W_0(\sigma),M_1(\sigma) \sqcup X(\sigma)\sqcup -M_1(\sigma) \sqcup -X(\sigma) )
\]
which induces an algebraic cobordism and hence the extracted algebraic data are also cobordant.
\end{proof}


\subsection{Normal invariants revisited} \label{subsec:normal-invariants-revisited}

\


Let $X$ be an $n$-dimensional GPC which admits a topological block bundle reduction of its SNF. For such an $X$ we will now discuss in more detail the bijection $\sN (X) \cong [X ; \G/\TOP]$ which was already used in the proof of Theorem \ref{thm:lifts-vs-orientations}. 

We first set up some notation. An element $x \in \sN (X)$ is represented either by a degree one normal map $(f,b) \co M \ra X$ from an $n$-dimensional topological manifold $M$ to $X$ or by a pair $(\nu,h)$ where $\nu \co X \ra \BSTOP$ is a stable topological block bundle on $X$ and $h \co J(\nu) \simeq \nu_X$ is a homotopy from the underlying spherical fibration to the SNF. The two descriptions of normal invariants are identified via the usual Pontrjagin-Thom construction, see \cite[chapter 10]{Wall(1999)} if needed. 

An element in the set $[X ; \G/\TOP]$ of homotopy classes of maps from $X$ to $\G/\TOP$ can be thought of as represented by a pair $(\bar \nu,\bar h)$, where $\bar \nu \co X \ra \BSTOP$ is a stable topological block bundle on $X$ and $h \co J(\bar \nu) \simeq \ast$ is a homotopy from the underlying spherical fibration to the constant map (which represents the trivial spherical fibration). The set $[X ; \G/\TOP]$ is a group under the Whitney sum operation and it has an action on $\sN (X)$ by
\begin{align*}
 [X ; \G/\TOP] \times \sN (X) & \ra \sN (X) \\
 ((\bar \nu,\bar h) , (\nu,h)) & \mapsto (\bar \nu \oplus \nu, \bar h \oplus h).
\end{align*}
The action is free and transitive \cite[chapter 10]{Wall(1999)} and hence any choice of a point $x_0 = (\nu_0,h_0) = (f_0,b_0) \in \sN (X)$ gives a bijection $[X ; \G/\TOP] \cong \sN (X)$ whose inverse is denoted by 
\begin{equation}
t (-,x_0) \co \sN (X) \rightarrow [X ; \G/\TOP]
\end{equation}
So for $x = (\nu,h)$ and $t (x,x_0) = (\bar \nu,\bar h)$ we have
\begin{equation}
 (\nu,h) = (\bar \nu \oplus \nu_0,\bar h \oplus h_0).   
\end{equation}


\subsection{From normal invariants to cohomology} \label{subsec:normal-invariants-to-cohomology}

\


\begin{con}  \label{con:widetilde-Gamma}
Now we construct the map 
\[
 \widetilde \Gamma \co \G/\TOP \ra \Sigma^{-1} \bOmega^{N,\STOP}_0.
\]
To an $l$-simplex in $\G/\TOP$ alias a degree one normal map $(f,b) \co M \ra \Delta^l$ the map $\widetilde \Gamma$ associates an $l$-simplex of $\Sigma^{-1} \bOmega^{N,\STOP}_0$ alias an $(l+1)$-dimensional $l$-ad of (normal,topological manifold) pairs $(W,M \sqcup -\Delta^l)$ where $W$ is the mapping cylinder of $f$ and the normal structure comes from the bundle map $b$.
\end{con}

The proof of Theorem \ref{thm:lifts-vs-orientations} shows that the surgery obstruction alias the quadratic signature map $\qsign \co \G/\TOP \ra \bL_0 \langle 1 \rangle$ can be thought of as a composition of two maps: 
\[
 \widetilde \Gamma \co \G/\TOP \ra \Sigma^{-1} \bOmega^{N,\STOP}_0 \quad \textup{and} \quad \qsign \co \Sigma^{-1} \bOmega^{N,\STOP}_0 \ra \bL_0 \langle 1 \rangle,
\]
where the second map comes from Proposition \ref{prop:connective-signatures-on-spectra-level}.


\subsection{Products and Thom isomorphism} \label{subsec:products}

\

                                                                                                                                   
Now we briefly review the cup products in ring and module spectra. We only concentrate on the cases which are used in this paper. In fact we only need the cup products realizing the Thom isomorphism. Let $X$ be a CW-complex and let $\xi$ be a $k$-dimensional spherical fibration over $X$. Further suppose that $\bF$ is a module spectrum over the ring spectrum $\bE$. Then there are the cup products
\begin{align*}
 - \cup - \co H^p (X ; \bF) \otimes H^q (\Th (\xi) ; \bE) & \ra H^{p+q} (\Th (\xi) ; \bF) \\
 x \otimes y & \mapsto x \cup y
\end{align*}
given by the composition
\[
 x \cup y \co \Th (\xi) \xra{\Delta} X_+ \wedge \Th (\xi) \xra{x \wedge y} \bF_p \wedge \bE_q \ra \bF_{p+q}
\]
with $\Delta$ the diagonal map already mentioned for example in section \ref{sec:normal-cplxs}.
If we have an $\bE$-orientation $u \in H^k (\Th (\xi) ; \bE)$, then the resulting homomorphism
\begin{equation}
 - \cup u \co H^p (X ; \bF) \ra H^{p+k} (\Th (\xi) ; \bF) 
\end{equation}
is the Thom isomorphism.

Remember now that we are working with simplicial complexes and $\Delta$-sets rather than topological spaces. The above constructions work in this setting as long as we choose a simplicial approximation of the diagonal map $\Delta$. An explicit description of such an approximation in our situation is given in \cite[Remark 12.5]{Ranicki(1992)}. However, it does not help us much since we do not understand its behavior with respect to the signatures we have defined earlier in this section. On the other hand, as we will see, we understand the behavior of the diagonal map $\Delta$ of spaces with respect to the orientations discussed in section \ref{sec:proof-part-1}. Since we also know that the orientations correspond to the signatures via the $S$-duality (Proposition \ref{prop:S-duals-of-orientations-are-signatures}) we can work with them and therefore we can work with the version of the cup product for spaces.


\subsection{Proof of (b) - Naturality of orientations} \label{subsec:proof-of-2}


\

The commutativity of the second square follows from the second paragraph of subsection \ref{subsec:normal-invariants-to-cohomology}. The commutativity of the third square follows from the naturality of the cup product with respect to the coefficient spectra and from the fact that the canonical $\bL^\bullet$-orientation of a stable topological block bundle is the image of the canonical $\bMSTOP$-orientation (Proposition \ref{canonical-L-orientations}). The commutativity of the fourth square follows from the naturality of the $S$-duality with respect to the coefficient spectra.


\subsection{Proof of (c) - Signatures versus orientations revisited} \label{subsec:signatures-versus-orientations}


\

If $x = (f,b) \co M \ra X$ represents an element from $\sN (X)$ then (c) can be expressed by the formula
\begin{equation} \label{eqn:cup-product-on-wgamma-and-orientation-gives-signature}
 \widetilde \Gamma (t(x,x_0)) \cup u^{\STOP} (\nu_0) = S^{-1} (\sign^{\G/\TOP}_X (x,x_0))
\end{equation}
in the group $H^k (\Th (\nu_X) ; \Sigma^{-1} \bOmega^{N,\STOP}_\bullet)$. Here $x_i \in \sN(X)$ are represented either by degree one normal maps $(f_i,b_i) \co M_i \ra X$ or pairs $(\nu_i,h_i)$ as in subsection \ref{subsec:normal-invariants-revisited} and we keep this notation for the rest of this section. For the proof an even better understanding of the relationship between various signatures and  orientations is needed. To put the orientations into the game we use the Pontrjagin-Thom map 
\[
 \Sigma^{-1} \bOmega^{N,\STOP}_\bullet \simeq \bMSGTOP := \textup{Fiber} \; (\bMSTOP \ra \bMSG).
\]
We will show (\ref{eqn:cup-product-on-wgamma-and-orientation-gives-signature}) in two steps, namely we show that both sides are equal to a certain element $u^{\G/\TOP} (\nu,\nu_0) \in H^k (\Th (\nu_X) ; \bMSGTOP)$.

\begin{con}
Recall the canonical $\STOP$-orientations (subsection \ref{subsec:L-theory-orientations})
\[
 u^{\STOP} (\nu) - u^{\STOP} (\nu_0) \in H^k (\Th (\nu_X) ; \bMSTOP)
\]
and also the fact that we have the homotopy $h_0 \cup h \co \Th (\nu_X) \times [-1,1] \ra \bMSG$ between $J (\nu)$ and $J (\nu_0)$. This homotopy can also be viewed as a null-homotopy of the map $J(u^{\STOP} (\nu) - u^{\STOP} (\nu_0))$. Hence we obtain a preferred lift which we denote
\[
 u^{\G/\TOP} (\nu,\nu_0) \in H^k (\Th (\nu_X) ; \bMSGTOP).
\]
\end{con}

\begin{prop} \label{prop:S-duals-of-orientations-are-signatures-relative-case-non-mfd}
Let $X$ be an $n$-dimensional GPC and let $x$, $x_0$ be two topological block bundle reductions of the SNF. Then we have
\[
 S (u^{\G/\TOP} (\nu,\nu_0)) = \sign^{\G/\TOP}_X (x,x_0) \in H_n (X ; \bMSGTOP).
\]
\end{prop}

\begin{proof}
The proof is analogous to the proof of Proposition \ref{prop:S-duals-of-orientations-are-signatures-relative-case}. Recall that the signature $\sign^{\G/\TOP}_X (x,x_0)$  is constructed using the dissections of the degree one normal maps $(f_i,b_i)$. From these dissections we inspect that we have a commutative diagram
\[
\xymatrix{
\Sigma^m/\overline X \ar[d]_{i} \ar[rr]^{\sign^{\G/\TOP}_X (x_1,x_0)} & & \Sigma^{-1} \Omega_{-k}^{N,\STOP} \ar[d]^{c} \\
\textup{Sing} \; F(\nu_1;\nu_0) \ar[rr]_-{u^{\G/\TOP} (\nu,\nu_0)} & & \textup{Sing} \; \bMSGTOP (k) }
\]
where we use the notation $\bMSGTOP (k) := \textup{Fiber} \; (\bMSTOP (k) \ra \bMSG (k))$ and $F(\nu_1;\nu_0) : = \textup{Pullback} \; (\Th (\nu_1) \ra \Th (\nu_X) \leftarrow \Th (\nu_0))$. This proves the claim. 
\end{proof}

Now we turn to the left hand side of the formula (\ref{eqn:cup-product-on-wgamma-and-orientation-gives-signature}). We first need to understand the composition (abusing the notation slightly):
\[
 \widetilde \Gamma \co [X;\G/\TOP] \xra{\widetilde \Gamma} H^0 (X;\Sigma^{-1} \bOmega^{N,\STOP}_\bullet) \ra H^0 (X;\bMSGTOP).
\]
Let $(\bar \nu,\bar h)$ represent an element on $[X;\G/\TOP]$. Recall that $\bar h \co J(\bar \nu) \simeq \varepsilon$ and that we have the canonical orientations $u^{\bMSTOP} (\nu)$ and $ u^{\bMSTOP} (\varepsilon)$ and the homotopy $u^{\bMSG} (\bar h) \co u^{\bMSG} (J(\nu)) \simeq u^{\bMSG} (\varepsilon)$. We obtain
\[
 \widetilde \Gamma (\bar \nu,\bar h)) = (u^{\bMSTOP} (\nu) - u^{\bMSTOP} (\varepsilon),u^{\bMSG} (\bar h) \co u^{\bMSG} (J(\nu)) - u^{\bMSG} (\varepsilon) \simeq \ast)
\]
Hence the element $\widetilde \Gamma (\bar \nu,\bar h)$ is the unique lift of $u^{\bMSTOP} (\nu) - u^{\bMSTOP} (\varepsilon)$ obtained from the homotopy $\bar h$.

Now consider our $x,x_0 \in \sN (X)$ and denote $t := t(x,x_0) = (\bar \nu,\bar h)$. As a warm up before proving the equation (\ref{eqn:cup-product-on-wgamma-and-orientation-gives-signature}) we consider its push-forward in the group $H^k (\Th (\nu_X) ; \bMSTOP)$. Denote the composition
\[
 \Gamma \co [X;\G/\TOP] \xra{\widetilde{\Gamma}} H^0 (X;\bMSGTOP) \xra{\textup{incl}} H^0 (X;\bMSTOP)
\]
This simply forgets the homotopy $u^{\bMSG} (\bar h)$. So we have
\[
 \Gamma (\bar \nu,\bar h) = u^{\bMSTOP} (\bar \nu) - u^{\bMSTOP} (\varepsilon) \co \Th (\bar \nu) \simeq \Sigma^k \Delta^l_+ \simeq \Th (\varepsilon) \ra \bMSTOP.
\]
Define the following two maps
\[
\Phi \co [X;\G/\TOP] \ra H^0 (X ; \bMSTOP) \quad \textup{and} \quad 1 \co [X;\G/\TOP] \ra H^0 (X;\bMSTOP)                                             
\]
by 
\[
\Phi (\bar \nu,\bar h) = u^{\bMSTOP} (\nu) \quad \textup{and} \quad 1 (\bar \nu,\bar h) = u^{\bMSTOP} (\varepsilon)                                             
\]
so that we have $\Gamma = \Phi - 1$ and consider $\Gamma (t) = (\Phi - 1) (t)$. The Thom isomorphism
\begin{equation}
 - \cup u^{\STOP}(\nu_0) \co H^0 (X ; \bMSTOP) \ra H^{k} (\Th (\nu_X) ; \bMSTOP) 
\end{equation}
applied to an element $\Phi (t)\in H^0 (X ; \bMSTOP)$ is given by the composition
\[
 \Th (\nu) \xra{\Delta} \Sigma^l X_+ \wedge \Th (\nu_0) \xra{\Phi(t) \wedge u^{\STOP}(\nu_0)} \bMSTOP \wedge \bMSTOP \xra{\oplus} \bMSTOP.
\]
From the relationship between the Whitney sum and the cross product and the diagonal map we obtain that
\[
 u^{\STOP} (\nu) = \Phi (t (x,x_0)) \cup u^{\STOP} (\nu_0).
\]
Analogously we obtain
\[
 u^{\STOP} (\nu_0) = 1 (t (x,x_0)) \cup u^{\STOP} (\nu_0).
\]

\begin{lem} \label{lem:orientations-vs-cup-product}
Let $X$ be an $n$-dimensional GPC and let $\nu$, $\nu_0 \co X \ra \BSTOP$ be two topological block bundles such that $J (\nu) \simeq \nu_X \simeq J(\nu_0)$. Then the canonical $\STOP$-orientations satisfy
\[
 u^{\STOP} (\nu) - u^{\STOP} (\nu_0) = \Gamma ( \tilde t (x,x_0)) \cup u^{\STOP} (\nu_0).
\]
\end{lem}

\begin{proof}
The desired equation follows from the definition $\Gamma = \Phi -1$.
\end{proof}


The final step is the following lemma which is a refinement of Lemma \ref{lem:orientations-vs-cup-product}.

\begin{lem} \label{lem:refined-orientations-vs-cup-product}
Let $X$ be an $n$-dimensional GPC and let $\nu$, $\nu_0 \co X \ra \BSTOP$ be two topological block bundles such that $J (\nu) \simeq \nu_X \simeq J(\nu_0)$. Then the canonical $\STOP$-orientations satisfy
\[
 u^{\G/\TOP} (\nu,\nu_0) = \widetilde \Gamma (t (x,x_0)) \cup u^{\STOP} (\nu_0).
\]
\end{lem}

\begin{proof}
The left hand side is obtained from the left hand side of Lemma \ref{lem:orientations-vs-cup-product} using the null-homotopy of $\Gamma (t(x,x_0))$ coming from $h \cup h_0 \co J(\nu)) \simeq J (\nu_0)$. The right hand side is obtained from the right hand side of Lemma \ref{lem:orientations-vs-cup-product} using the null-homotopy of $\Gamma (t (x,x_0)$ coming from $\bar h \co J(t(x,x_0)) \simeq J (t (x_0,x_0)) = \nu_X$. Applying the cup product with $u^{\STOP} (\nu_0)$ to this null-homotopy corresponds to taking the Whitney sum with $\nu_0$ and produces the homotopy $\bar h \oplus \id_{\nu_0} \co J(\nu)) \simeq J (\nu_0)$. The claim now follows from the property of the SNF that any two fiberwise homotopy equivalences between the stable topological block bundle reductions of the SNF are stably fiberwise homotopic.
\end{proof}


\subsection{Proof of the Main Technical Theorem (II)} \label{subsec:proof-of-main-tech-thm-2}


\begin{proof}[Proof of the Main Technical Theorem (II) assuming Propositions \ref{prop:degree-one-normal-map-mfd-to-poincare-over-X} and \ref{prop:identification}] \ \vspace{-0.4cm}

Consider the set
 \begin{align*}
  Q := \{ \, - &\text{sign}^{\mathbf{L}_{\bullet}}_{\mathbb{Z}[\pi_1(X)]} (f,b) \, \in \, L_n (\mathbb{Z}[\pi_1(X)]) \quad | \\ 
                         &(f,b) \colon M \rightarrow X \text{ degree one normal map, } M \text{ manifold} \}.
 \end{align*}
Fix a degree one normal map~$(f_0,b_0) \colon M_0 \rightarrow X$ from a manifold~$M_0$ to our Poincar\'e complex~$X$. Proposition \ref{prop:identification} tells us that
\begin{align*}
 \qsign_{\ZZ [\pi_1(X)]} (f_0,b_0) + Q = & \im \; (A \co H_n
(X;\bL_\bullet \langle 1 \rangle \ra L_n (\ZZ[\pi_1 (X)])) \\ = & \ker \; (\del \co L_n (\ZZ[\pi_1 (X)]) \ra \SS_n (X))
\end{align*}
and it follows that $Q$ is a coset of $\ker (\del)$. The preimage~$\partial^{-1}s(X) \subseteq L_n(\mathbb{Z}[\pi_1(X)])$ is also a coset of $\ker (\del)$. Moreover, from Proposition \ref{prop:degree-one-normal-map-mfd-to-poincare-over-X} we have~$Q \subseteq \partial^{-1}s(X)$. Hence~$Q$ and~$\partial^{-1}s(X)$ are the same coset of~$\ker (\del)$ and thus~$Q = \partial^{-1}s(X)$.
\end{proof}



\section{Concluding remarks} \label{sec:conclusion}


In Part II of the book \cite{Ranicki(1992)} interesting generalizations and applications of the theory can be found.

One important such generalization is the theory when one works with the spectrum $\bL_\bullet \langle 0 \rangle$ rather than with $\bL_\bullet \langle 1 \rangle$. This yields an analogous theory for the ANR-homology manifolds rather than for topological manifolds. The Quinn resolution obstruction also fits nicely into this theory. For details see \cite[chapters 24,25]{Ranicki(1992)} and \cite{Bryant-et-al(1996)}.

We note that, as already mentioned in the introduction, this generalization is especially interesting in view of the recent progress in studying the assembly maps associated to the spectrum $\bL_\bullet$. For example, thanks to the generalization, the results about the assembly maps in \cite{Bartels-Lueck(2009)} can be used to obtain an application in \cite{Bartels-Lueck-Weinberger(2009)}, which discusses when does a torsion-free word-hyperbolic group $G$ have a topological manifold model for its classifying space $BG$.

Another important application is that the total surgery obstruction can be used to identify the geometric structure set of an $n$-dimensional manifold $M$ with $\SS_{n+1} (M)$. This is closely related to subsection \ref{subsec:identification-of-quad-sign-with-assembly} and in fact the geometric surgery exact sequence can be identified with the algebraic surgery exact sequence, with more details to be found in \cite[chapter 18]{Ranicki(1992)}.

Interesting examples of geometric Poincar\'e complexes with non-trivial total surgery obstruction can be found in \cite[chapter 19]{Ranicki(1992)}.

\small
\bibliography{tso-all}  
\bibliographystyle{alpha}

\end{document}